\documentclass[reqno]{amsart}
\usepackage{url}
\usepackage{tikz}
\usetikzlibrary{decorations.pathreplacing, matrix, arrows,tikzmark}
\usepackage{enumerate}
\usepackage{accents}
\usepackage{tabu}
\usepackage{longtable}
\usepackage{multirow}
\usepackage{amssymb}
\usepackage{amsmath}
\usepackage{mathtools}
\usepackage{amscd}
\usepackage{amsbsy}
\usepackage{comment}
\usepackage[matrix,arrow]{xy}
\usepackage{arydshln}
\usepackage{soul}

\DeclareMathOperator{\Cl}{Cl}

\DeclareMathOperator{\Norm}{Norm}

\DeclareMathOperator{\Ker}{Ker}
\DeclareMathOperator{\Img}{Image}

\newcommand{\vv}{\mathbf{v}}

\DeclareMathOperator{\norm}{Norm}

\DeclareMathOperator{\ord}{ord}

\newcommand{\Q}{\mathbb{Q}}
\newcommand{\Z}{\mathbb{Z}}
\newcommand{\C}{{\mathbb C}}
\newcommand{\R}{\mathbb{R}}
\newcommand{\F}{{\mathbb F}}

\newcommand{\cU}{\mathcal{U}}
\newcommand{\cA}{\mathcal{A}}

\newcommand{\cD}{\mathcal{D}}
\newcommand{\cB}{\mathcal{B}}
\newcommand{\cK}{\mathcal{K}}
\newcommand{\cL}{\mathcal{L}}

\newcommand{\cM}{\mathcal{M}}

\newcommand{\cP}{\mathcal{P}}

\newcommand{\fQ}{\mathfrak{Q}}

\newcommand{\cT}{\mathcal{T}}
\newcommand{\cE}{\mathcal{E}}

\newcommand{\OO}{{\mathcal O}}
\newcommand{\ga}{{\mathfrak{a}}}

\newcommand{\sS}{\mathfrak{S}}
\newcommand{\fa}{\mathfrak{a}}
\newcommand{\fb}{\mathfrak{b}}
\newcommand{\fc}{\mathfrak{c}}
\newcommand{\fp}{\mathfrak{p}}
\newcommand{\fq}{\mathfrak{q}}

\newcommand{\mP}{\mathfrak{P}}
\newcommand{\xx}{\mathbf{x}}
\newcommand{\zz}{\mathbf{z}}

\newcommand{\hh}{\mathbf{h}}
\newcommand{\ww}{\mathbf{w}}
\newcommand{\uu}{\mathbf{u}}
\newcommand{\bfl}{\mathbf{l}}
\newcommand{\bb}{\mathbf{b}}
\newcommand{\aaf}{\mathbf{a}}
\newcommand{\cc}{\mathbf{c}}
\newcommand{\mm}{\mathbf{m}}
\newcommand{\nn}{\mathbf{n}}
\newcommand{\rr}{\mathbf{r}}

\newcommand{\fA}{\mathfrak{A}}
\newcommand{\fB}{\mathfrak{B}}

\newcommand\Tstrut{\rule{0pt}{6ex}}         
\newcommand\Bstrut{\rule[-5ex]{0pt}{0pt}}   

\newcommand{\edit}[1]{\textcolor{red}{#1}}
\newcommand{\noedit}[1]{\textcolor{blue}{#1}}

\begin{document}

\newtheorem{thm}{Theorem}[section]
\newtheorem{proc}[thm]{Procedure}
\newtheorem{alg}[thm]{Algorithm}
\newtheorem{lem}[thm]{Lemma}
\newtheorem{prop}[thm]{Proposition}
\newtheorem{cor}[thm]{Corollary}
\newtheorem*{conj}{Conjecture}

\theoremstyle{definition}
\newtheorem{defn}[thm]{Definition}
\newtheorem{example}[thm]{Example}

\theoremstyle{remark}
\newtheorem{rem}[thm]{Remark}

\newtheorem{ack}{Acknowledgement}

\title[]{Efficient Resolution of Thue--Mahler Equations}

\author{Adela Gherga}
\address{Tutte Institute for Mathematics and Computing, Ottawa, Ontario, Canada}
\email{adela.gherga@cse-cst.gc.ca}
\author{Samir Siksek}
\address{Mathematics Institute, University of Warwick, Coventry CV4 7AL, United Kingdom}
\email{s.siksek@warwick.ac.uk}
\thanks{
The authors are supported by the
EPSRC grant \emph{Moduli of Elliptic curves and Classical Diophantine Problems}
(EP/S031537/1).
}

\date{\today}


\makeatletter
\@namedef{subjclassname@2020}{%
  \textup{2020} Mathematics Subject Classification}
\makeatother


\begin{abstract}
A Thue-Mahler equation is a Diophantine equation of the form
\[
	F(X,Y) = a\cdot p_1^{z_1}\cdots p_v^{z_v}, \qquad \gcd(X,Y)=1
\]
where $F$ is an irreducible binary form of degree at least $3$ with integer coefficients, $a$ is a non-zero integer and $p_1, \dots, p_v$ are rational primes.  Existing algorithms for resolving such equations require computations in the field $L=\Q(\theta,\theta^\prime,\theta^{\prime\prime})$, where $\theta$, $\theta^\prime$, $\theta^{\prime\prime}$ are distinct roots of $F(X,1)=0$. We give a new algorithm that requires computations only in $K=\Q(\theta)$ making it far more suited for higher degree examples.  We also introduce a lattice sieving technique reminiscent of the Mordell--Weil sieve that makes it practical to tackle Thue--Mahler equations of higher degree and with larger sets of primes than was previously possible. We give several examples including one of degree $11$.

Let $P(m)$ denote the largest prime divisor of an integer $m \ge 2$.  As an application of our algorithm we determine all pairs $(X,Y)$ of coprime non-negative integers such that $P(X^4-2Y^4) \le 100$, finding that there are precisely $49$ such pairs.
\end {abstract}
\maketitle


\section{Introduction} \label{sec:intro}

Let
\begin{equation}\label{eqn:F}
  F(X,Y)=a_0 X^d+a_1 X^{d-1} Y+ \cdots + a_d Y^d
\end{equation}
be a binary form of degree $d \ge 3$ with coefficients $a_i \in
\Z$. Suppose $F$ is irreducible over $\mathbb{Q}$. Let $a$ be a non-zero integer and let
$p_1,\dotsc,p_v$ be distinct primes such that $p_i \nmid a$. The purpose of
this paper is to give an efficient algorithm to solve the Thue--Mahler equation
\begin{equation}\label{eqn:preTM}
  F(X,Y)=a \cdot
  p_1^{z_1} \cdots p_v^{z_v}, \qquad
  X,~Y \in \Z, \;
  \gcd(X,Y)=1, 
\end{equation}
for unknown integers $X,Y$, and unknown
non-negative integers $z_1, \dots, z_v$. The set of solutions
is known to be finite by a famous result of Mahler \cite{Mahler}
which extends classical work of Thue \cite{Thue}.
Mahler's theorem is ineffective. The first effective bounds
on the size of the solutions are due to
Vinogradov and Sprind\v{z}uk \cite{VinogradovSprindzuk}
and to
Coates \cite{CoatesTM}.
Vastly improved effective bounds have since been given by
Bugeaud and Gy\H{o}ry \cite{BGTM}. Evertse \cite[Corollary 2]{EvertseTM}
showed that the number of solutions to \eqref{eqn:preTM} is at most $2 \times 7^{d^3(2v+3)}$.
For $d \ge 6$, this has been improved by Bombieri \cite[Main Theorem]{BombieriTM}
who showed that the number of solutions is at most $16 (v+1)^2 \cdot (4d)^{26(v+1)}$.


Besides being of independent interest, Thue--Mahler equations frequently arise
in a number of contexts:
\begin{itemize}
\item The problem of determining all elliptic curves over
$\Q$ with good reduction outside a given set of primes algorithmically
reduces to the problem of solving certain cubic Thue--Mahler equations (here
cubic means $d=3$).
The earliest example appears to be due to
Agrawal, Coates, Hunt, and van der
              Poorten \cite{AgrawalCoates} who used it to
determine all elliptic curves over $\Q$ of conductor $11$.
The recent paper of Bennett, Gherga and Rechnitzer \cite{BeGhRe} gives
a systematic and general treatment of this approach.
In fact, the link between cubic Thue--Mahler
equations and elliptic curves can be used in conjunction
with modularity of elliptic curves to give an algorithm
for solving cubic Thue--Mahler equations as in the
work of von K\"{a}nel and Matschke \cite[Section 5]{KanelMatschke}
and of Kim \cite{KimTM}.
\item Many Diophantine problems naturally reduce
to the resolution of Thue--Mahler equations. These include
the Lebesgue--Nagell equations (e.g.\ \cite{Cangul}, \cite{SoTz}),
and Goormaghtigh's equation (e.g.\ \cite{BeGhKr}). The most striking
of such applications is the reduction, due to Bennett and Dahmen \cite{BennettDahmen},
of asymptotic cubic superelliptic equations to cubic Thue--Mahler equations, via the modularity of
Galois representation attached to elliptic curves.
\end{itemize}

\medskip

Before the current paper, the only general algorithm for solving
Thue--Mahler equations was the one due to Tzanakis and de Weger \cite{TW}.
A modern implementation of this algorithm, due Hambrook \cite{Hambrook},
has been profitably used to solve a number of low degree Thue--Mahler
equations, for example in \cite{Cangul}, \cite{SoTz}.

\medskip

Instead of \eqref{eqn:preTM}, we consider the equation
\begin{equation}\label{eqn:TM}
  F(X,Y)=a \cdot
  p_1^{z_1} \cdots p_v^{z_v}, \qquad
  X,~Y \in \Z, \;
  \gcd(X,Y)=\gcd(a_0,Y)=1.
\end{equation}
Thus we have added the assumption $\gcd(a_0,Y)=1$,
where $a_0$ is the leading coefficient of $F$
as in \eqref{eqn:F}.
This is a standard computational simplification
in the subject, and is also applied by Tzanakis and de Weger.
There is no loss of generality in adding
this assumption in the following sense:
an algorithm for solving \eqref{eqn:TM}
yields an algorithm for solving \eqref{eqn:preTM}.
To see this, let $(X,Y)$ be a solution to \eqref{eqn:preTM}
and let $b=\gcd(a_0,Y)$.
Write $Y=bY^\prime$ with $Y^\prime \in \Z$.
The possible values for $b$ are the divisors of $a_0$;
for each divisor $b$ we need to solve
$F(X,bY^\prime)=a \cdot p_1^{z_1} \cdots p_v^{z_v}$.
Note that $F(X,bY^\prime)=b G(X,Y^\prime)$
where $G$ has integral coefficients
and leading coefficient $a_0^\prime=a_0/b$,
which satisfies $\gcd(a_0^\prime,Y^\prime)=1$.
The equation $b G(X,Y^\prime)=a \cdot p_1^{z_1} \cdots p_v^{z_v}$
is impossible unless $b/\gcd(a,b)=p_1^{w_1} \cdots p_v^{w_v}$
where $w_i \ge 0$, in which case
\[
G(X,Y^\prime)=(a/\gcd(a,b)) \cdot p_1^{z_1-w_1} \cdots p_v^{z_v-w_v}
\]
which is now a Thue--Mahler equation of the form \eqref{eqn:TM}.

\medskip

The approach of Tzanakis and de Weger can be summarized as follows.
\begin{enumerate}[(I)]
\item First \eqref{eqn:TM} is reduced to a number of ideal equations:
\begin{equation}\label{eqn:ideal}
(a_0 X- \theta Y)\OO_K=\fa \cdot \fp_1^{n_1}\cdots \fp_s^{n_s}.
\end{equation}
Here $\theta$ is a root of the monic polynomial $a_0^{d-1} F(X/a_0,1)$,
and $K=\Q(\theta)$. Moreover, $\fa$ is a fixed ideal of the ring
of integers $\OO_K$, and $\fp_1,\dotsc,\fp_s$ are fixed prime ideals of $\OO_K$.
The variables $X$, $Y$, $n_1,\dotsc,n_s$ represent the unknowns.
\item Next, each ideal equation \eqref{eqn:ideal} is reduced
to a number of equations of the form
\begin{equation}\label{eqn:unitequation}
a_0 X-\theta Y = \tau \cdot \delta_1^{b_1} \cdots \delta_r^{b_r}
\end{equation}
where $\tau$, $\delta_1,\cdots,\delta_r \in K^{\times}$ are fixed and $X$, $Y$, $b_1,\dotsc,b_r$
are unknowns.
\item The next step generates a very large upper bound
for the exponents $b_1,\dotsc,b_r$
using the theory of real, complex, and $p$-adic linear forms in
logarithms. This bound is then considerably reduced using the LLL algorithm
\cite{LLL} applied to approximation lattices associated to these linear forms,
and finally, all solutions below this reduced bound are found using the
algorithm of Fincke and Pohst \cite{FinckePohst}.
\end{enumerate}
To compute these approximation lattices alluded to in step (III), the algorithm of Tzanakis and de Weger
relies on extensive computations in the number field $K^\prime=\Q(\theta_1,\theta_2,\theta_3)$
where $\theta_1$, $\theta_2$, $\theta_3$ are distinct roots of $a_0^{d-1} F(X/a_0,1)$,
as well as $\fp$-adic completions of $K^\prime$.
The field $K^\prime$ typically has degree $d(d-1)(d-2)$,
making their algorithm impractical if the degree
$d$ is large. Even if the degree $d$ is small (say $d=3$), we have
found that the Tzanakis-de Weger algorithm runs into a combinatorial
explosion of cases in step (I) if the number of primes $v$ is large,
and in step (II) if the class number $h$ of $K$ is large.


In this paper, we present an algorithm that builds on many of the
powerful ideas in the paper of Tzanakis and de Weger but
avoids computations in number fields
other than the field $K=\Q(\theta)$ of degree $d$, and avoids all
computations in $p$-adic fields or their extensions. The algorithm includes
a number of refinements that circumvent the explosion of cases in steps (I) and (II).
For example, to each ideal equation \eqref{eqn:ideal} we associate at most one
equation \eqref{eqn:unitequation}; by contrast, the algorithm of Tzanakis
and de Weger, typically associates $h^{s-1}$ equations \eqref{eqn:unitequation}
to each ideal equation \eqref{eqn:ideal}, where $h$ is the class number of $K$.
Moreover, inspired by the Mordell--Weil sieve (e.g.\ 
\cite{BruinStollMWI}, 
\cite{BruinStollMWII},
\cite{IntegralHyp}), we introduce
a powerful \lq\lq Dirichlet sieve\rq\rq\ that vastly improves the determination of the solutions
to \eqref{eqn:unitequation} after the LLL step, even if the
remaining bound on the exponents $b_i$ is large.

We have implemented the algorithm described in this paper in the computer algebra
system \texttt{Magma} 
\cite{Magma}; our implementation is available from \newline
{\small \url{https://github.com/adelagherga/ThueMahler/tree/master/Code/TMSolver}}.

\bigskip

Below we give four examples of Thue--Mahler equations solved using our
implementation.
Our solutions will always be subject to the assumptions
\[
	\gcd(X,Y)=\gcd(a_0,Y)=1.
\]
They will be given
in the form $[X,Y,z_1,z_2,\dotsc,z_v]$.
We will revisit these examples later on in the paper to
illustrate the differences between our algorithm and that of
Tzanakis and de Weger \cite{TW}.
We do point out that a number of recent papers also
make use of our implementation, or the ideas in the present
paper, to solve various Thue--Mahler equations
where the degree $d$, or the number of primes $v$ are large. For example in
\cite{BeGhKr}, \cite{BGPS}, \cite{BeSiNew}, \cite{BeSiLN}, our algorithm
is used to solve Thue--Mahler equations of degrees $5$, $20$, $7$ and $11$ respectively.



\begin{example}\label{ex:Ex1}
An ongoing large-scale computational project, led by Bennett, Cremona, Gherga and Sutherland,
aims to provably compute all elliptic curves of conductor at most $10^6$. The method combines the approach in \cite{BeGhRe} with our Thue--Mahler solver
described in the current paper. We give an example to illustrate
this application.
Consider the problem of computing all elliptic curves $E/\Q$
with trivial $2$-torsion and conductor
\[
	771456 \; = \; 2^7 \cdot 3 \cdot 7^2 \cdot 41.
\]
Applying Theorem 1 of \cite{BeGhRe}  results in $13$ cubic Thue--Mahler
equations of the form
\[
a_0 X^3+a_1 X^2 Y+a_2 XY^2+a_3 Y^3=
3^{z_1} \cdot  7^{z_2}\cdot 41^{z_3},
        \qquad \gcd(X,Y)=1,
\]
whose resolution algorithmically yields the desired set
of elliptic curves. The coefficients for these
$13$ forms are as follows:
\[\begin{array}{c|c}
  (a_0,a_1,a_2,a_3) & (a_0,a_1,a_2,a_3) \\ \hline
  (1,7,2,-2) & (2,1,0,3) \\
  (1,4,3,6) & (3,4,4,4) \\
  (4,4,6,3) & (2,5,0,6) \\
  (1,7,4,12) & (3,3,-1,7) \\
  (3,7,14,14) & (1,3,17,43) \\
  (8,12,13,8) & (4,1,12,-6) \\
  (3,9,5,19) &
\end{array}\]
Our implementation solved all $13$ of these Thue--Mahler equations and
computed the corresponding elliptic curves in a total of $5.4$ minutes
on a single core. Full computational details, as well as a list of all corresponding elliptic curves can be found at \newline
{\small \url{https://github.com/adelagherga/ThueMahler/tree/master/GhSiData/Example1}}.

For illustration, we consider one of these $13$ Thue--Mahler equations:
\[
	4X^3+ X^2 Y + 12X Y^2 - 6Y^3= 3^{z_1} \cdot  7^{z_2}\cdot 41^{z_3},
	\qquad \gcd(X,Y)=1.
\]
Our implementation solved this in $41$ seconds.
The solutions are
\begin{gather*}
  [ -3, -7, 1, 0, 1 ],\quad
  [ -1, -5, 2, 2, 0 ],\quad
  [ 1, -1, 1, 1, 0 ], \\
  [ 3, 1, 1, 2, 0 ], \quad
  [ 5, 11, 0, 2, 0 ],\quad
  [ 9, 17, 1, 2, 1 ],\quad
  [ 19, 23, 5, 3, 0 ].
\end{gather*}
Of the seven solutions, only
$(X,Y) = (1,-1)$ gives rise to elliptic curves of conductor $771456$: \[E_1:
y^2 = x^3 - x^2 +  13655x + 2351833,\]
\[E_2: y^2 = x^3 - x^2 +  3414x - 295686.\]
\end{example}

\begin{example}\label{ex:Ex2}
Consider the Thue--Mahler equation
\begin{equation}\label{eqn:TMBad2}
7 X^3 + X^2 Y + 29 X Y^2 - 25 Y^3 \; = \; 2^{z_1}\cdot 3^{z_2} \cdot 17^{z_3}
\cdot 37^{z_4} \cdot 53^{z_5} \, .
\end{equation}
This in fact is one of the Thue--Mahler equations whose resolution, via the
method of \cite{BeGhRe}, is needed
to determine all elliptic curves of conductor $2^{\alpha}\cdot 3^{\beta}\cdot 17 \cdot 37 \cdot 53$, where $\alpha \in \{2,3,4,6,7\}$ and $\beta \in \{1,2\}$.
The class number of the cubic field associated to the cubic form in \eqref{eqn:TMBad2}
is $33$. As we shall see later (at the end of Section~\ref{sec:making-ideals-princ}),
our approach to dealing with the class group requires us to solve only $30$
equations of the form \eqref{eqn:unitequation}, whereas in comparison
the method of Tzanakis and de Weger requires the resolution of approximately $80990$
equations of the form \eqref{eqn:unitequation}. For now, we merely point out
that our implementation solved \eqref{eqn:TMBad2} in $2$ minutes.
The solutions are
\begin{gather*}
  [ 19, -23, 2, 4, 0, 1, 1 ],\quad
  [ 13, -6, 0, 0, 1, 1, 1 ],\quad
  [ -343, -463, 2, 11, 1, 0, 0 ],\\
  [ 79, -8, 0, 2, 2, 2, 0 ],\quad
  [ 37, -13, 2, 1, 1, 0, 2 ],\quad
  [ 1, 1, 2, 1, 0, 0, 0 ],\\
  [ 3, 4, 0, 0, 1, 0, 0 ].
\end{gather*}
\end{example}

\begin{example}\label{ex:Ex3}
Most explicit examples of the resolution of Thue--Mahler equations
\eqref{eqn:TM}
found in the literature involve a relatively small
set of primes $\{p_1,\dotsc,p_v\}$. The following
example, aside from being an interesting Diophantine
application in its own right, is intended to
illustrate that our algorithm can cope with a relatively
large set of primes.

For a non-zero integer $m$, let $P(m)$
denote the maximum prime divisor of $m$ (where
we take $P(1)=P(-1)=0$).
In this example, we solve the inequality
\[
	P(X^4-2Y^4) \le 100, \qquad \gcd(X,Y)=1.
\]
Let $F(X,Y)=X^4-2Y^4$.
We therefore would like to solve \eqref{eqn:TM} with $a= \pm 1$
and with $p_1,\dotsc,p_v$
being the set of primes $\le 100$ (of which there are $25$).
However, it is clear that if $2$ is not a fourth power
modulo $p$, then $p \nmid (X^4-2Y^4)$. Thus we reduce
to the much smaller set of primes $p \le 100$
for which $2$ is a fourth power. This is the set
\[
\{2, 7, 23, 31, 47, 71, 73, 79, 89\}.
\]
Therefore the Thue--Mahler equation we shall consider is
\[
X^4-2Y^4 \; = \; \pm 2^{z_1} \cdot 7^{z_2} \cdot 23^{z_3}
\cdot 31^{z_4} \cdot 47^{z_5} \cdot 71^{z_6} \cdot 73^{z_7} \cdot
79^{z_8} \cdot 89^{z_9}, \qquad \gcd(X,Y)=1.
\]
Our implementation took
roughly 3.5 days to solve this Thue--Mahler equation.
There are $49$ solutions (up to changing the signs of $X$, $Y$):
\begin{gather*}
 [ 0, 1, 1, 0, 0, 0, 0, 0, 0, 0, 0 ],\quad
    [ 1, 0, 0, 0, 0, 0, 0, 0, 0, 0, 0 ],\displaybreak[0] \\
    [ 1, 1, 0, 0, 0, 0, 0, 0, 0, 0, 0 ],\quad
    [ 1, 2, 0, 0, 0, 1, 0, 0, 0, 0, 0 ],\displaybreak[0] \\
    [ 1, 3, 0, 1, 1, 0, 0, 0, 0, 0, 0 ],\quad
    [ 1, 4, 0, 1, 0, 0, 0, 0, 1, 0, 0 ],\displaybreak[0] \\
    [ 1, 11, 0, 1, 0, 0, 1, 0, 0, 0, 1 ],\quad
    [ 2, 1, 1, 1, 0, 0, 0, 0, 0, 0, 0 ],\displaybreak[0] \\
    [ 2, 3, 1, 0, 0, 0, 0, 0, 1, 0, 0 ],\quad
    [ 2, 27, 1, 1, 0, 2, 0, 0, 0, 1, 0 ],\displaybreak[0] \\
    [ 3, 1, 0, 0, 0, 0, 0, 0, 0, 1, 0 ],\quad
    [ 3, 2, 0, 2, 0, 0, 0, 0, 0, 0, 0 ],\displaybreak[0] \\
    [ 3, 14, 0, 0, 1, 0, 1, 1, 0, 0, 0 ],\quad
    [ 4, 3, 1, 0, 0, 0, 1, 0, 0, 0, 0 ],\displaybreak[0] \\
    [ 4, 5, 1, 1, 0, 0, 0, 1, 0, 0, 0 ],\quad
    [ 5, 1, 0, 1, 0, 0, 0, 0, 0, 0, 1 ],\displaybreak[0] \\
    [ 5, 8, 0, 1, 1, 0, 1, 0, 0, 0, 0 ],\quad
    [ 6, 5, 1, 0, 1, 0, 0, 0, 0, 0, 0 ],\displaybreak[0] \\
    [ 6, 19, 1, 0, 0, 1, 1, 0, 0, 0, 1 ],\quad
    [ 8, 1, 1, 0, 1, 0, 0, 0, 0, 0, 1 ],\displaybreak[0] \\
    [ 10, 23, 1, 2, 0, 0, 0, 1, 0, 1, 0 ],\quad
    [ 11, 9, 0, 2, 0, 1, 0, 0, 0, 0, 0 ],\displaybreak[0] \\
    [ 11, 20, 0, 0, 0, 0, 1, 0, 1, 0, 1 ],\quad
    [ 15, 1, 0, 0, 1, 1, 0, 1, 0, 0, 0 ],\displaybreak[0] \\
    [ 15, 13, 0, 0, 0, 0, 0, 0, 1, 0, 1 ],\quad
    [ 16, 21, 1, 0, 1, 0, 0, 0, 0, 1, 1 ],\displaybreak[0] \\
    [ 17, 5, 0, 2, 1, 0, 0, 0, 1, 0, 0 ],\quad
    [ 19, 20, 0, 4, 0, 0, 0, 0, 0, 1, 0 ],\displaybreak[0] \\
    [ 21, 11, 0, 0, 0, 1, 0, 0, 2, 0, 0 ],\quad
    [ 22, 49, 1, 0, 1, 1, 0, 0, 0, 0, 2 ],\displaybreak[0] \\
    [ 33, 13, 0, 1, 0, 0, 2, 0, 1, 0, 0 ],\quad
    [ 37, 19, 0, 0, 1, 2, 0, 0, 1, 0, 0 ],\displaybreak[0] \\
    [ 40, 13, 1, 1, 0, 1, 0, 0, 1, 1, 0 ],\quad
    [ 52, 51, 1, 2, 1, 1, 0, 0, 0, 0, 1 ],\displaybreak[0] \\
    [ 53, 44, 0, 1, 1, 1, 0, 0, 0, 1, 0 ],\quad
    [ 59, 56, 0, 0, 0, 1, 1, 1, 1, 0, 0 ],\displaybreak[0] \\
    [ 61, 48, 0, 1, 0, 0, 0, 1, 1, 0, 1 ],\quad
    [ 66, 101, 1, 0, 2, 1, 0, 0, 1, 1, 0 ],\displaybreak[0] \\
    [ 68, 43, 1, 1, 1, 2, 1, 0, 0, 0, 0 ],\quad
    [ 95, 58, 0, 1, 0, 1, 1, 0, 1, 1, 0 ],\displaybreak[0] \\
    [ 118, 101, 1, 2, 1, 0, 0, 1, 0, 0, 1 ],\quad
    [ 142, 57, 1, 1, 3, 1, 0, 0, 1, 0, 0 ],\displaybreak[0] \\
    [ 162, 137, 1, 2, 0, 0, 2, 0, 1, 0, 0 ],\quad
    [ 181, 124, 0, 0, 1, 0, 1, 0, 0, 2, 1 ],\displaybreak[0] \\
    [ 221, 295, 0, 0, 2, 0, 1, 0, 1, 1, 1 ],\quad
    [ 281, 199, 0, 1, 1, 0, 1, 1, 1, 1, 0 ],\displaybreak[0] \\
    [ 286, 283, 1, 3, 1, 0, 0, 0, 3, 0, 0 ],\quad
    [ 389, 96, 0, 4, 0, 1, 1, 0, 1, 0, 1 ],\displaybreak[0] \\
    [ 420, 437, 1, 0, 0, 1, 3, 0, 1, 0, 1 ].
\end{gather*}
\end{example}

\begin{example}\label{ex:Ex4}
All examples found in the literature are of Thue--Mahler
equations where the form $F$ has the property that
the field $K^\prime=\Q(\theta_1,\theta_2,\theta_3)$
(defined above) has small degree.
As indicated above, a distinguishing feature of our algorithm is that
the computations are carried out in the much
smaller extension $K=\Q(\theta)$ (also defined above).
Our last example is intended to illustrate this difference.
Consider the Thue--Mahler equation,
\begin{multline*}
5 X^{11} + X^{10} Y + 4 X^9 Y^2 + X^8 Y^3 + 6 X^7 Y^4 + X^6 Y^5 + 6 X^5 Y^6 +
    6 X^3 Y^8 + 4 X Y^{10} - 2 Y^{11}\\
=
2^{z_1} \cdot 3^{z_2} \cdot 5^{z_3} \cdot 7^{z_4} \cdot 11^{z_5} \, .
\end{multline*}
As usual, we let $F$ denote the form on the left-hand side.
The Galois group of $F$ is $S_{11}$, therefore the field
$K$ has degree $11$, whereas the field $K^\prime$ has degree
$11 \times 10 \times 9=990$.
Our program solved this Thue--Mahler equation in around $6.8$ hours. However this
time was almost entirely taken up with the computation of the class group and
the units of $K$. Once the class group and unit computations were complete, it took
only $3.6$ minutes to provably determine the solutions.
These are
\begin{gather*}
    [ 0, -1, 1, 0, 0, 0, 0 ], \qquad
    [ 1, -1, 1, 1, 1, 0, 0 ],\\
    [ 1, 1, 5, 0, 0, 0, 0 ],\qquad
    [ 1, 2, 0, 3, 0, 1, 1 ]\, .
\end{gather*}
\end{example}

\subsection{Notation and Organization of the paper}
As before $F \in \Z[X,Y]$ will be a
binary form of degree $d \ge 3$, irreducible in $\Q[X,Y]$,
and with coefficients $a_0,\dotsc,a_d$ as in  \eqref{eqn:F}.
Let $a$ be a non-zero integer and $p_1,\dotsc,p_v$ be distinct primes
satisfying $p_i \nmid a$.
Let
\[
f(x)=a_0^{d-1}\cdot F(x/a_0,1)=x^d+a_{1} x^{d-1}+a_0 a_{2} x^{d-2}+\cdots
+a_0^{d-1} a_d.
\]
This is an irreducible monic polynomial with coefficients in $\Z$.
Let $\theta$ be a root
of $f$ and let $K=\Q(\theta)$. Note that $K$ is a number field
of degree $d$.
Write $\OO_K$ for the ring of integers of $K$.
We can rewrite our Thue--Mahler equation
\eqref{eqn:TM} as
\begin{equation}\label{eqn:TM2}
\Norm(a_0 X- \theta Y)= a_0^{d-1} \cdot a \cdot p_1^{z_1} \cdots p_v^{z_v}.
\end{equation}
Note that we do not assume that $(a_0,p_i) = 1$.

\bigskip

The paper is organized as follows.
\begin{enumerate}[(a)]
\item In Section~\ref{sec:ppart} we
consider the decomposition of the ideal $(a_0X - \theta Y)\OO_K$
as a product of prime ideals. In particular, we introduce an
algorithm to compare and restrict the possible valuations of all prime ideals
above each of $p_1, \dots, p_v$.
\item We summarize the results of applying this
algorithm in Section~\ref{sec:eq-in-ideals}, wherein we reduce solving
\eqref{eqn:TM2} to solving a family of ideal equations of the form
\eqref{eqn:ideal}.
\item In Section~\ref{sec:making-ideals-princ}, we show
that such ideal equations are either impossible due to a class group
obstruction, or reduce to a single
equation of the form
\eqref{eqn:unitequation}.
The remainder of the paper is devoted to solving equations of the form
\eqref{eqn:unitequation} where the unknowns are coprime integers
$X$, $Y$ and non-negative integers $b_1,\dotsc,b_r$.
\item In Section~\ref{sec:lower-bounds}, we recall key
theorems from the theory of lower bounds for linear forms in complex and $p$-adic logarithms
due to Matveev and to Yu.
\item In Section~\ref{sec:sunit}, with the help of these theorems,
we obtain a very large upper bound on
the exponents $b_1,\dotsc,b_r$ in \eqref{eqn:unitequation}.
\item \label{fp} In Section~\ref{sec:controlling-vals} we show how an application
of close vector algorithms allows us to obtain
a substantially improved bound on the $\fp$-adic valuation of $a_0X-\theta Y$
for any prime $\fp$. This step avoids the $\fp$-adic logarithms of earlier
approaches.
\item \label{rfrl} Section~\ref{sec:LFRL}
uses the real and complex embeddings of $K$ applied to \eqref{eqn:unitequation}
to obtain $d-2$ approximate relations involving the exponents
$b_i$. In Section~\ref{sec:red}, we set up an
\lq approximation lattice\rq\ using these $d-2$
approximate relations. We explain how
close vector algorithms can be used
to substantially reduce our bound
for the exponents $b_1,\dotsc,b_r$ in \eqref{eqn:unitequation}.
Earlier approaches used just one
of the $d-2$ relations to construct the approximation
lattice, but we explain why using just one
approximate relation can fail in certain situations.
\item Steps (\ref{fp}) and (\ref{rfrl}) are applied
repeatedly until no further
improvements in the bounds are possible.
In Section~\ref{sec:sieve}
we introduce an analogue
of the Mordell--Weil sieve,
which we call the \lq\lq Dirichlet sieve\rq\rq\,
which is capable of efficiently sieving
for the solutions up to the remaining bounds,
thereby finally resolving the Thue-Mahler equation \eqref{eqn:TM}.
\end{enumerate}

\subsection*{Acknowledgements}
The authors are grateful to Mike Bennett,
Rafael von K\"{a}nel and Benjamin Matschke
for stimulating discussions. The authors are indebted to the
referee for many pertinent corrections and improvements.


\section{The $p$-part of $(a_0 X- \theta Y) \OO_K$}\label{sec:ppart}

If $\fc$ is a fractional ideal of $\OO_K$, and $p$ is
a rational prime, we define the \textbf{$p$-part of $\fc$} to be the fractional ideal
\[
\prod_{\fp \mid p} \fp^{\ord_\fp(\fc)}.
\]
For each rational prime $p\in \{p_1, \dots, p_v\}$ of \eqref{eqn:TM2}, we want to study the $p$-part of ${(a_0 X- \theta Y) \OO_K}$
coming from the prime ideals above $p$.
The so-called Prime Ideal Removal Lemma in Tzanakis and de Weger
compares the possible valuations of ${(a_0 X-\theta Y)\OO_K}$ at two
prime ideals $\fp_1$, $\fp_2 \mid p$ to help cut down the possibilities
for the the $p$-part of $(a_0 X- \theta Y) \OO_K$.
However if $\fp_1 \mid (a_0 X-\theta Y)\OO_K$
then this restricts the values of $X$ and $Y$ modulo $p$.
Indeed, any choice of $X$ and $Y$ modulo $p$ affects the valuations
of $(a_0 X- \theta Y)\OO_K$ at all primes $\fp \mid p$. So we study all
valuations at the same time, not just two of them. This enables us to give a much smaller list of possibilities for the $p$-part of $(a_0 X- \theta Y) \OO_K$ than in Tzanakis and de Weger, as we will see in Section~\ref{sec:making-ideals-princ}.

\begin{defn}\label{defn:sat}
Let $p$ be a rational prime. Let $L_p$ be a subset of the ideals
$\fb$ supported at the prime ideals
above $p$. Let $M_p$ be a subset of the set of pairs $(\fb,\fp)$ where $\fb$ is supported at the prime
ideals above $p$, and $\fp \mid p$ is a prime ideal
satisfying $e(\fp|p)=f(\fp|p)=1$, where $e(\fp|p)$ and $f(\fp|p)$ are, respectively, the ramification index and inertial degree of $\fp$ over $p$. We call the pair $L_p$, $M_p$ \textbf{satisfactory}
if for every solution $(X,Y)$
 to \eqref{eqn:TM},
\begin{enumerate}
\item[(i)] either the $p$-part of $(a_0 X-\theta Y)\OO_K$ is in $L_p$,
\item[(ii)] or there is a pair $(\fb,\fp) \in M_p$
and a non-negative integer $l$ such that the $p$-part of
$(a_0 X-\theta Y)\OO_K$ is equal to $\fb \fp^{l}$.
\end{enumerate}
\end{defn}
At this point the definition
is perhaps mysterious. Lemma~\ref{lem:dimyst}
and the following remark give an explanation for
the definition and for the existence of finite
satisfactory sets $L_p$, $M_p$.
We will give an algorithm to produce
(hopefully small) satisfactory sets $L_p$ and $M_p$.
Before that we embark on a simplification.
The expression $a_0 X-\theta Y$ is a linear form
in two variables $X$, $Y$. It is easier to scale
so that we are dealing with a linear expression
in just one variable.

For a rational prime $p$, let
\[
\Z_{(p)}=\{U \in \Q \; : \; \ord_p(U) \ge 0\}.
\]
\begin{defn}\label{defn:adeq}
Let $p$ be a rational prime.
Let $\alpha \in K$ and $\beta \in K^{\times}$.
Let $L$ be a subset of the ideals $\fb$ supported on
the prime ideals of $\OO_K$ above $p$. Let $M$ be a subset of the set of pairs
$(\fb,\fp)$ where $\fb$ is supported on the prime ideals
above $p$, and where $\fp$ is a prime ideal above $p$ satisfying
$e(\fp|p)=f(\fp|p)=1$. We call $L$, $M$ \textbf{adequate
for $(\alpha,\beta)$} if for every $U \in \Z_{(p)}$,
\begin{enumerate}
\item[(i)] either the $p$-part of $\beta \cdot (U+\alpha) \OO_K$ is in $L$,
\item[(ii)] or there is a pair $(\fb,\fp) \in M$
and a non-negative integer $l$ such that
the $p$-part  of  $\beta \cdot (U+\alpha) \OO_K$
equals $\fb \fp^l$.
\end{enumerate}
\end{defn}

\begin{lem}\label{lem:adequate}
Let $L$, $M$ be adequate for $(-\theta/a_0,a_0)$ and let $L_p=L \cup \{1\cdot\mathcal{O}_K\}$ and $M_p=M$. Then the pair $L_p$, $M_p$ is satisfactory.
\end{lem}
\begin{proof}
  Recall that $\gcd(X,Y)=\gcd(a_0,Y)=1$.

  If $p \mid Y$ then $\ord_{\fp}(Y) > 0$ for any $\fp$ above $p$, and thus
\[
	\ord_{\fp}(a_0X-\theta Y) =0.
\]

  If $p \nmid Y$, we write
    \[
      U=\frac{X}{Y}, \qquad \alpha=\frac{-\theta}{a_0}, \qquad
      \beta=a_0.
    \]
    Then $U \in \Z_{(p)}$ and $\ord_{\fp}(a_0 X - \theta Y) = \ord_\fp(\beta \cdot (U+\alpha))$ for all prime ideals $\fp$ above $p$. Thus the $p$-part of $\beta \cdot (U+\alpha)$ is equal to the $p$-part of $\ord_{\fp}(a_0 X - \theta Y)$.

  The lemma follows.
\end{proof}

We now demystify Definitions~\ref{defn:sat}
and~\ref{defn:adeq}.
\begin{lem}\label{lem:dimyst}
Let $p$ be a rational prime and $\gamma$ a generator of $K$. Then there is a bound $B$ depending only on $p$ and $\gamma$ such that
the following hold:
\begin{enumerate}
\item[(a)] For any $U \in \Z_{(p)}$ and any pair
of distinct prime ideals $\fp_1$, $\fp_2$ lying over $p$,
\[
\ord_{\fp_1}(U+\gamma) \le B, \qquad \text{or} \qquad
\ord_{\fp_2}(U+\gamma) \le B.
\]
\item[(b)] For any $U \in \Z_{(p)}$ and any prime ideal $\fp$ over $p$ with $e(\fp|p) \ne 1$ or ${f(\fp|p) \ne 1}$,
\[
\ord_\fp(U+\gamma) \le B.
\]
\end{enumerate}
\end{lem}
\begin{proof}
Let $\fp$ be a prime ideal above $p$, and suppose that $\ord_\fp(U+\gamma)$
is unbounded for $U \in \Z_{(p)}$. Thus
there is an infinite sequence
$\{U_i\} \subset \Z_{(p)}$
such that
\[
\lim_{i \rightarrow \infty} \ord_{\fp}(U_i+\gamma)
\; = \; \infty.
\]
However, $\Z_{(p)} \subset \Z_p$, where the latter is
compact. Thus $\{U_i\}$ contains
an infinite subsequence
$\{U_{n_i}\}$ converging to, say, $U \in \Z_p$.
Write $\phi_{\fp} \; : \; K \hookrightarrow \C_p$ for the embedding
of $K$ corresponding to the prime ideal $\fp$.
It follows that ${\phi_{\fp}(\gamma)=-U \in \Z_p}$. Recall the assumption that
$K=\Q(\gamma)$. Thus $K_\fp$, the topological closure
of $\phi_\fp(K)$ in $\C_p$, is in fact $\Q_p$. Thus $e(\fp|p)=f(\fp|p)=1$.
This proves (b).

For (a), suppose that there is a pair of distinct primes
$\fp_1$, $\fp_2$ above $p$ and an infinite sequence
$\{U_i\} \subset \Z_{(p)}$
such that
\begin{equation}\label{eqn:limit}
\lim_{i \rightarrow \infty} \ord_{\fp_1}(U_i+\gamma)
\; = \;
\lim_{i \rightarrow \infty} \ord_{\fp_2}(U_i+\gamma)
\; = \; \infty.
\end{equation}
Again, let
$\{U_{n_i}\}$ be an infinite subsequence of $\{U_i\}$
converging to, say, $U \in \Z_p$. Then
$\phi_{\fp_1}(\gamma)=-U=\phi_{\fp_2}(\gamma)$. As $K=\Q(\gamma)$,
the embeddings $\phi_{\fp_1}$, $\phi_{\fp_2}$ are equal,
contradicting
$\fp_1 \ne \fp_2$.
\end{proof}

\noindent \textbf{Remark.} We apply Lemma~\ref{lem:dimyst} with $\gamma = \alpha$ as in Lemma~\ref{lem:adequate} in order to explain Definitions~\ref{defn:sat} and~\ref{defn:adeq}. The valuation of $\beta \cdot (U+\alpha)$ can be arbitrarily large only
for those $\fp$ above $p$ that satisfy $e(\fp|p)=f(\fp|p)=1$, and if it is sufficiently large for one such $\fp$ then it is bounded for all others. Thus there must exist finite adequate sets $L$, $M$. We now turn to the task of giving an algorithm to determine such adequate sets $L$, $M$.


\begin{lem}\label{lem:strucstab}
Let $\alpha \in K$ and let $p$ be a rational prime. Let $\fp$ be a prime ideal of $\OO_K$ above $p$.
Suppose $U \in \Z_{(p)}$ and
\begin{equation}\label{eqn:valhu}
\ord_\fp(U+\alpha) > \min\{0,\ord_\fp(\alpha)\}.
\end{equation}
Then the following hold:
\begin{enumerate}
\item[(i)] $\ord_{\fp}(\alpha) \ge 0$.
\item[(ii)] The image of $\alpha$ in $\F_{\fp}:=\OO_K/\fp$
belongs to the prime subfield $\F_p$. In particular, there is a unique $u \in \{0, \dots, p-1\}$ such that $u \equiv -\alpha \pmod{\fp}$.
\item[(iii)] With $u$ as in (ii), $U=p U^\prime+u$
where $U^\prime \in \Z_{(p)}$.
\end{enumerate}
\end{lem}
\begin{proof}
  Since $U \in \Z_{(p)}$, we have $\ord_\fp(U) \ge 0$. If $\ord_{\fp}(\alpha)<0$, it follows that ${\ord_\fp(U+\alpha)=\ord_\fp(\alpha)}$, contradicting \eqref{eqn:valhu}. Thus $\ord_\fp(\alpha) \ge 0$, proving (i).

Write $\overline{\alpha}$ for the image of $\alpha$ in $\F_{\fp}$,
and suppose this does not belong to the prime subfield $\F_p$.
In particular $\ord_\fp(\alpha)=0$. However, the image $\overline{U}$
of $U$ in $\F_{\fp}$ does belong to $\F_p$. Thus $U \not \equiv -\alpha
\pmod{\fp}$, or equivalently $\ord_\fp(U+\alpha)=0$, contradicting \eqref{eqn:valhu}. We deduce that $\overline{\alpha} \in \F_p$,
and thus (ii) holds.

Now, let $u$ be as in (ii). By \eqref{eqn:valhu}, we have $\ord_\fp(U+\alpha)>0$, and thus $\overline{U}=-\overline{\alpha}=\overline{u}$.
But $\overline{U}$, $\overline{u} \in \F_p$. Therefore, $\ord_p(U-u)>0$,
and so $U=pU^\prime+u$ for some $U^\prime \in \Z_{(p)}$.
\end{proof}

\begin{alg}\label{alg1}
Given $p$ a rational prime, $\alpha \in K$ satisfying
$K=\Q(\alpha)$, and $\beta \in K^{\times}$,
to compute $L$, $M$ adequate for $(\alpha,\beta)$:
\begin{enumerate}
\item[Step (a)] Let
\[
\cB=\{\fp \mid p \; : \;
\ord_\fp(\alpha) \ge 0 \text{ and }
\text{the image of $\alpha$ in $\F_\fp$ belongs to $\F_p$}\},
\]
and
\[
\fb=\prod_{\fp \mid p} \fp^{\ord_\fp(\beta)+\min\{0,\ord_\fp(\alpha)\}}.
\]
\item[Step (b)]
If $\cB=\emptyset$ then return $L=\{\fb\}$, $M=\emptyset$ and terminate
the algorithm.
\item[Step (c)]
If $\cB$ consists of a single prime ideal $\fp^\prime$ satisfying
$e(\fp^\prime|p)=f(\fp^\prime|p)=1$ then return
$L=\emptyset$, $M=\{(\fb,\fp^\prime)\}$ and terminate the algorithm.
\item[Step (d)]
Let
\[
 \cU=\{ 0 \le u \le p-1 \; : \; \text{there is some $\fp \in \cB$
 such that $\alpha \equiv -u \pmod{\fp}$}\}.
\]
Loop through the elements $u \in \cU$. For each $u$,
use Algorithm~\ref{alg1} to compute adequate $L_u$, $M_u$
for the pair $((u+\alpha)/p,p\beta)$.
\begin{enumerate}
\item[Step (d1)] If $\cU=\{0,1,2,\dotsc,p-1\}$ then return
\begin{equation}\label{eqn:return}
L=\bigcup_{u \in \cU} L_u, \qquad M=\bigcup_{u \in \cU} M_u,
\end{equation}
and terminate the algorithm.
\item[Step (d2)] Else, return
\begin{equation}\label{eqn:returnII}
L=\{\fb\} \cup \bigcup_{u \in \cU} L_u, \qquad M=\bigcup_{u \in \cU} M_u,
\end{equation}
and terminate the algorithm.
\end{enumerate}
\end{enumerate}
\end{alg}

\noindent \textbf{Remarks.}
\begin{itemize}
\item Algorithm~\ref{alg1} is recursive.
If the hypotheses of (b) and (c) fail
then the algorithm replaces the linear form $\beta \cdot (U+\alpha)$
with a number of  linear forms
\[
p \beta \cdot (U^\prime+ (u+\alpha)/p)=\beta \cdot (pU^\prime+u+\alpha).
\]
In essence we are replacing $\Z_{(p)}$ with a number of the
cosets of $p \Z_{(p)}$. The algorithm is then applied to
each of these  linear forms individually.
\item Note that the number of prime ideals $\fp$ above $p$ is bounded by the degree $[K:\Q]$. In particular, $\# \cU \le [K:\Q]$. Therefore the number of branches
at each iteration of the algorithm is bounded independently
of $p$.
\end{itemize}
\begin{prop}
Suppose $\beta \in K^{\times}$, $\alpha \in K$ and moreover, $K=\Q(\alpha)$.
Then Algorithm~\ref{alg1} terminates in finite time
and produces adequate $L$, $M$ for $(\alpha,\beta)$.
\end{prop}
\begin{proof}
Let $\cB$ and $\fb$ be as in Step (a).
Observe that, for any $\fp$ above $p$,
\[
\ord_\fp(\beta \cdot (U+\alpha)) \ge
\ord_\fp(\beta)+\min\{0,\ord_\fp(\alpha)\}=\ord_{\fp}(\fb).
\]
It follows that $\fb$ divides the $p$-part of
$\beta \cdot (U+\alpha)$.
Lemma~\ref{lem:strucstab} tells us that
\begin{equation}\label{eqn:equality}
\ord_\fp(\beta \cdot (U+\alpha))=
\ord_\fp(\beta)+\min\{0,\ord_\fp(\alpha)\}=\ord_{\fp}(\fb)
\end{equation}
for all prime ideals $\fp$ lying above $p$, except possibly for $\fp \in \cB$.
If $\cB=\emptyset$ (i.e.\ the hypothesis of (b)
is satisfied), then the $p$-part
of $\beta \cdot (U+\alpha)$ is $\fb$, and hence the pair
$L=\{\fb\}$, $M=\emptyset$ is
adequate for $(\alpha,\beta)$.
If $\cB=\{\fp^\prime\}$ where $e(\fp^\prime|p)=f(\fp^\prime|p)=1$
(i.e.\ the hypothesis of step (c) is satisfied) then
the $p$-part of $\beta \cdot (U+\alpha)$ has the form
$\fb \cdot {\fp^\prime}^{l}$ for some $l \ge 0$. Hence
$L=\emptyset$, $M=\{(\fb,\fp^\prime)\}$ are adequate for
$(\alpha,\beta)$.

Suppose the hypotheses of Step (b) and Step (c) fail.
Let $\cU$ be as in Step (d).
If $U \equiv u \pmod{p}$ for some $u \in \cU$,
then $U=p U^\prime+u$ for some $U^\prime \in \Z_{(p)}$.
Thus $\beta \cdot (U+\alpha)=p \beta \cdot (U^\prime+(u+\alpha)/p)$,
and so naturally the $p$-parts of $\beta \cdot (U+\alpha)$
and $p \beta \cdot (U^\prime+(u+\alpha)/p)$ agree.
In (d1), $\cU$ represents all of the congruence classes
modulo $p$,
and this justifies \eqref{eqn:return}.
In (d2), $\cU$ represents some of the congruence
classes. If $u \notin \cU$,
then by Lemma~\ref{lem:strucstab},
the equality \eqref{eqn:equality} holds for all prime ideals $\fp$ above $p$,
and hence $\fb$ is the $p$-part of $\beta \cdot (U+\alpha)$.
This justifies \eqref{eqn:returnII}.

Next we show that the algorithm terminates
in finitely many steps. Suppose otherwise.
Then there will be an infinite sequence of
$u_i \in \{0,1,\dotsc,p-1\}$ and pairs
$(\alpha_i,\beta_i)$ with
\[
\alpha_0=\alpha, \qquad \beta_0=\beta, \qquad
\alpha_{i+1}=\frac{u_i+\alpha_i}{p}, \qquad \beta_{i+1}=p \beta_i.
\]
Let us denote by $\cB_i$ the set $\cB$ for the pair $(\alpha_i,\beta_i)$.
It is easy to see from the definition that $\cB_{i+1} \subseteq \cB_i$.
Suppose $\fp$ belongs to infinitely many of the $\cB_i$. Then,
infinitely often, $\ord_\fp(\alpha_i) \ge 0$. However,
\[
\alpha=\alpha_0=-u_0-u_1 p -u_2 p^2 - \cdots - u_{i-1} p^{i-1}+p^i \alpha_i.
\]
Let $\mu=-u_0-u_1 p -\cdots  \in \Z_p$. Let $\phi_\fp : K \hookrightarrow \C_p$
be the embedding corresponding to $\fp$. Then $\phi_\fp(\alpha)=\mu$.
Since $K=\Q(\alpha)$, the embedding $\phi_\fp$ is determined
by the image of $\alpha$. Since $\phi_\fp \ne \phi_{\fp^\prime}$
whenever $\fp \ne \fp^\prime$, we see that $\cB_i$
consists of at most one prime for $i$ sufficiently large. For such a prime, we must have $\phi_\fp(K)=\Q_p$ so that $e(\fp|p)=f(\fp|p)=1$. Thus for sufficiently large $i$
the algorithm must terminate at Step (b) or Step (c),
giving a contradiction.
\end{proof}

Following Lemma~\ref{lem:adequate}, we thus let $L_p =L \cup \{1\cdot\mathcal{O}_K\}$ and $M_p=M$, where we compute $L$ and $M$ using Algorithm~\ref{alg1} with $\alpha = -\theta/a_0$ and $\beta = a_0$.

\medskip

\noindent \textbf{Refinements.}
Let $L_p$, $M_p$ be a satisfactory pair (for example,
produced by Algorithm~\ref{alg1}). We explain
here some obvious refinements that will reduce or simplify
these sets, whilst maintaining the satisfactory property.
\begin{itemize}
\item If some pair $(\fb,\fp)$ is in $M_p$ then
we may replace this with the pair $(\fb^\prime,\fp)$ where
\[
\fb^\prime=\frac{\fb}{\fp^{\ord_\fp(\fb)}} \, .
\]
\item If some $\fb$ is contained in $L_p$, and some $(\fb^\prime,\fp)$
is contained in $M_p$ with $\fb^\prime \mid \fb$ and $\fb/\fb^\prime=\fp^w$
for some $w \ge 0$, then we may delete $\fb$ from $L_p$.
\item Suppose $p \mid a_0$.
Observe that if
 $\fb \in L_p$ is the $p$-part of $(a_0 X- \theta Y) \OO_K$ then
 $\ord_p(\Norm(a_0 X- \theta Y))=\ord_p(\Norm(\fb))$.
However, it is clear that $\ord_p(\Norm(a_0 X-\theta Y)) \ge (d-1) \ord_p(a_0)$. Thus, we may delete $\fb$ from $L_p$ if $\ord_p(\Norm(\fb))< (d-1) \ord_p(a_0)$.
\end{itemize}


\section{An Equation in Ideals} \label{sec:eq-in-ideals}

Let $(X,Y)$ be a solution of \eqref{eqn:TM}. For every $p \in P:=\{p_1,\dotsc,p_v\}$ we let $L_p$, $M_p$ be a corresponding satisfactory pair. From Definition \ref{defn:sat}, we see that there is some partition  $P=P_1 \cup P_2$ such that for every $p \in P_1$, the $p$-part of $(a_0 X- \theta Y) \OO_K$ equals $\fb \fp^{l}$ for some $(\fb,\fp) \in M_p$ and $l \ge 0$, and for every $p \in P_2$, it equals some $\mathfrak{b} \in L_p$. Let $P=P_1 \cup P_2$ be a partition of $P$ and write
\[
P_1=\{q_1,\dotsc,q_{s}\}, \qquad P_2=\{q_{s+1},\dotsc,q_v\}.
\]
Let $\mathcal{Z}_{P_1,P_2}$ be the set of all pairs $(\fa,S)$
such that
there are $(\fb_i,\fp_i) \in M_{q_i}$ for $1 \le i \le s$,
and $\fb_j \in L_{q_j}$ for $s+1 \le j \le v$
satisfying
\[
\fa= \fb_0\cdot\fb_1 \fb_2 \cdots \fb_s \cdot \fb_{s+1} \cdots \fb_v, \qquad
S=\{\fp_1,\dotsc,\fp_s\},
\]
where $\fb_0$ denotes an ideal of $\OO_K$ of norm
  \[R = \bigg|\frac{a\cdot a_0^{d-1}}{\gcd(\Norm(\fb_1 \cdots\fb_v),a\cdot a_0^{d-1})}\bigg|.\]

Let
\[
\mathcal{Z} \; := \; \bigcup_{P_1 \subseteq P}
\mathcal{Z}_{P_1,P-P_1}.
\]
\begin{prop}\label{prop:Sprop}
Let $(X,Y)$ be a solution to
\eqref{eqn:TM}. Then there is some $(\fa,S) \in \mathcal{Z}$
such that
\begin{equation}\label{eqn:TMfactored}
(a_0 X- \theta Y)\OO_K=\fa \cdot \fp_1^{n_1}\cdots \fp_s^{n_s}
\end{equation}
where
$S=\{\fp_1,\dotsc,\fp_s\}$, and $n_1,\dotsc,n_s$
are non-negative integers. Moreover, the set $S$ has the following
properties:
\begin{enumerate}
\item[(a)] $e(\fp_i|q_i)=f(\fp_i|q_i)=1$ for $1 \le i \le s$.
\item[(b)] Let $1 \le i \le s$. Let $p$ be the unique rational
prime below $\fp_i$. Then $\fp_j \nmid p$ for all $1 \le j \le s$
with $j \ne i$.
\end{enumerate}
\end{prop}
\begin{proof}
The claims follow from the definitions of $\mathcal{Z}$ and $M_p$.
\end{proof}

To solve \eqref{eqn:TM}, we will solve \eqref{eqn:TMfactored}
for every possible choice of $(\fa,S) \in \mathcal{Z}$.

\medskip

\noindent\textbf{Remark.} Observe that for any $(\fa,S) \in \mathcal{Z}$, there may be several possibilities for $\fb_0$. Let $\mathcal{R}$ denote the set of all ideals $\fb_0$ having norm $R$ for some $(\fa,S) \in \mathcal{Z}$. Here, we provide a simple refinement to cut down the number of ideals in $\mathcal{R}$. In particular, we apply Algorithm~\ref{alg1} and Lemma~\ref{lem:adequate} to each rational prime $p$ dividing $R$, generating the corresponding sets $M_p$ and $L_p$. For each $\fb_0 \in \mathcal{R}$, if the $p$-part of $\fb_0$ cannot be made up of any of the elements of $M_p$ or $L_p$, we may remove $\fb_0$ from $\mathcal{R}$. Moreover, if this process yields $\mathcal{R} = \emptyset$, we may remove $(\fa,S)$ from $\mathcal{Z}$.

\section{Making the Ideals Principal}
\label{sec:making-ideals-princ}

From now on we fix $(\fa,S) \in \mathcal{Z}$ and
we focus on a solution of \eqref{eqn:TMfactored}.
The method of Tzanakis and de Weger \cite{TW} reduces
\eqref{eqn:TMfactored}
to at most $(m/2) \cdot h^s$ $S$-unit equations, where $m$
is the number of roots of unity and $h$ is the class number of $K$. Our method,
explained below, gives at most only $m/2$ $S$-unit equations.

Given an ideal $\fb$ of $\OO_K$, we denote its class in
the class group $\Cl(K)$ by $[\fb]$.
\begin{lem}
Let
\[
\phi : \Z^s \rightarrow \Cl(K), \qquad (m_1,\dotsc,m_s)
\mapsto [\fp_1]^{m_1}\cdots [\fp_s]^{m_s}.
\]
\begin{enumerate}
\item[(a)] If $[\fa]^{-1}$ is not in the image of $\phi$ then
\eqref{eqn:TMfactored} has no solutions.
\item[(b)] Suppose $[\fa]^{-1}=\phi(\rr)$,
where $\rr=(r_1,\dotsc,r_s)$. Let $\zeta$
be a generator of the roots of unity in $K$,
and suppose it has order $m$.
Let $\delta_1,\dotsc,\delta_r$ be a basis
for the group of $S$-units $\OO_S^\times$
modulo the torsion subgroup $\langle \zeta \rangle$.
Let $\alpha$ be a generator of the principal
ideal $\fa \cdot \fp_1^{r_1} \cdots \fp_s^{r_s}$.
Let $(X,Y)$ satisfy \eqref{eqn:TMfactored}.
Then, after possibly replacing $(X,Y)$ by $(-X,-Y)$,
we have
\begin{equation}\label{eqn:taudeltaorig}
a_0 X- \theta Y \; = \; \tau \cdot \delta_1^{b_1} \cdots \delta_r^{b_r}
\end{equation}
where $\tau=\zeta^{a} \cdot \alpha$ with $0 \le a \le \frac{m}{2}-1$,
and $b_1,\dotsc,b_r \in \Z$.
\end{enumerate}
\end{lem}
\begin{proof}
Note that if \eqref{eqn:TMfactored} has a solution
$\nn=(n_1,\dotsc,n_s)$ then
\[
\phi(\nn)=[\fa]^{-1}.
\]
This proves $(a)$.
For $(b)$, suppose $[\fa]^{-1}=\phi(r_1,\dotsc,r_s)$. Thus $\fa \cdot \fp_1^{r_1} \cdots \fp_s^{r_s}$ is principal
and we let $\alpha$ be a generator. Then the fractional ideal
$((a_0X-\theta Y)/\alpha) \OO_K$ is supported on $S=\{\fp_1,\dotsc,\fp_s\}$.
Hence $(a_0X-\theta Y)/\alpha \in \OO_S^\times$.
Now $\zeta,\delta_1,\dotsc,\delta_r$ is a set of generators for the $S$-unit
group, where $\delta_1,\dotsc,\delta_r$ is in fact a basis
for $\OO_S^\times/\langle \zeta \rangle$. Thus \eqref{eqn:taudeltaorig} holds
for some $0\le a \le m-1$, and $b_1,\dotsc,b_r \in \Z$.
However $\zeta^{m/2}=-1$. Thus we can suppose
$0 \le a \le m/2-1$ by replacing $(X,Y)$ by $(-X,-Y)$
if necessary.
\end{proof}

\begin{lem}\label{lem:tauden}
The denominator of the fractional ideal $\tau \OO_K$
is supported on the set of prime ideals ${S=\{\fp_1,\dotsc,\fp_s\}}$.
\end{lem}
\begin{proof}
This follows immediately from \eqref{eqn:taudeltaorig}
since $\delta_1,\dotsc,\delta_r$ are $S$-units.
\end{proof}

We have reduced the task of solving our original
Thue--Mahler equation \eqref{eqn:TM} to solving
equations of the form
\begin{equation}\label{eqn:TMdelta}
a_0 X- \theta Y=\tau \cdot \delta_1^{b_1} \cdots \delta_r^{b_r},
\end{equation}
subject to the conditions
\begin{equation}\label{eqn:restrictions}
X,~Y \in \Z, \quad \gcd(X,Y)=1, \quad \gcd(a_0,Y)=1, \quad
b_i \in \Z.
\end{equation}
For technical reasons we would like to exclude
the case $b_1=b_2=\cdots=b_r=0$;
of course we can trivially test if this case leads to
a solution. Hence we shall henceforth suppose in addition to \eqref{eqn:restrictions} that
\begin{equation}\label{eqn:restrictions2}
	\max\{\lvert b_1\rvert,\dotsc,\lvert b_r\rvert\} \; \ge \; 1.
\end{equation}
We will tackle each of these equations
\eqref{eqn:TMdelta} separately.
In \cite{TW}, the authors work in the field
generated by three conjugates of $\theta$
and its completions. This is fine theoretically
but difficult computationally. We will work
with that extension theoretically simply to obtain
a bound for
\begin{equation}\label{eqn:B}
	B:=\max\{\lvert b_1\rvert,\dotsc,\lvert b_r\rvert\}.
\end{equation}
(The reason for restriction \eqref{eqn:restrictions2}
is that in Section~\ref{sec:sunit} we work with $\log{B}$).
To reduce the bound, we will need to carry out
certain computations;
these will take place only in $K$, $\R$, $\C$,
but not in extensions of $K$, and certainly not in
extensions of $\Q_p$.

To obtain our initial bound for $B$ we shall mostly
follow ideas found in \cite{BG}, \cite{IntegralHyp},
\cite{Homero}. However, we have a key advantage
that will allow us to obtain sharper bounds:
namely we assume knowledge of the $S$-unit basis
$\delta_1,\dotsc,\delta_r$ rather than working
with estimates for the size of a basis.

\bigskip

\subsection{Convention on the choice of $S$-unit basis}
\label{page:convention}
As before let $\zeta$ be
a generator for the roots of unity in $\OO_K^\times$.
We note that the unit group $\OO_K^\times$
is a subgroup of the $S$-unit group $\OO_S^\times$.
Moreover, it is saturated in the sense that
the quotient $\OO_S^\times/\OO_K^\times$ is
torsion-free. Hence there is a basis
$\delta_1,\dotsc,\delta_r$ for $\OO_S^\times/\langle \zeta \rangle$
such that $\delta_1,\dotsc,\delta_{u+v-1}$
is a basis for $\OO_K^\times/\langle \zeta \rangle$, where $(u,v)$
is the signature of $K$. We shall in fact
work with such a basis.

\subsection{Examples continued}
Table~\ref{table1} gives some data for Examples \ref{ex:Ex1}--\ref{ex:Ex4}.
At this stage of the algorithm, we would like to stress the main difference between our approach and that of Tzanakis and de Weger \cite{TW}. When dealing with
\[(a_0 X- \theta Y)\OO_K=\fa \cdot \fp_1^{n_1}\cdots \fp_s^{n_s},\]
they write each exponent $n_i$ as $n_i=k_i h_i + m_i$, where $h_i$ is the order of $\fp_i$ in the class group of $\OO_K$ and $0 \le m_i \le h_i-1$. Now $\fp_i^{h_i}$ is principal and we may write it as
$(\beta_i) \OO_K$. It follows from \eqref{eqn:TMfactored}
that $\fa \cdot \fp_1^{m_1} \cdots \cdot \fp_s^{m_s}$ is principal. Clearly $h_i|h$ and there are at most $h^s$ possibilities for this ideal. In the worst case, we expect around
$h^{s-1}$ ideals to be principal, and so of the form $(\tau) \OO_K$.
This results in a huge explosion of cases when $h$ is non-trivial, as it is often the case that $h_i = h$.
For instance, in our Example~\ref{ex:Ex2}, the class number is $33$.
There are $32$ possibilities for $(\fa,S)$ in $\mathcal{Z}$.
The $32$
possible values for $s=\#S$ are
\begin{gather*}
  0,~ 0,~ 1,~ 1,~ 1,~ 1,~ 1,~ 1,~ 1,~ 1,~ 2,~ 2,~ 2,~ 2,~ 2,~ 2,\\
  2,~ 2,~ 2,~ 2,~ 2,~ 2,~ 3,~ 3,~ 3,~ 3,~ 3,~ 3,~ 3,~ 3,~ 4,~ 4,
\end{gather*}
where all possible ideals in $S$ have order $33$ in $\OO_K$. Following the method of Tzanakis and de Weger, we would need to
check $2672672$ ideals if they are principal, and this approach
leads to approximately $80990$ equations of the form~\eqref{eqn:TMdelta}.
With our approach, as we need only to deal with
$32$ ideal equations, we expect to deal with at most $32$ equations of the form \eqref{eqn:TMdelta}.
Indeed, in doing so, we obtain merely $30$
equations of the
form~\eqref{eqn:TMdelta}.

\begin{table}[h]
\begin{tabular}{|c|c|c|c|}
  \hline
  & $\#\mathcal{Z}$ & number of $(\tau,\delta_1,\dotsc,\delta_r)$ & rank frequencies\\
  \hline\hline
  Example~\ref{ex:Ex1} & $16$ & $16$ & $(1, 2)$, $(2, 6)$, $(3, 6)$, $(4, 2)$\\
  \hline
  Example~\ref{ex:Ex2} & $32$ & $30$ & $ (2, 8)$, $(3, 12)$, $(4, 8)$, $(5,2)$\\
  \hline
  Example~\ref{ex:Ex3} & $4096$ & $4096$ & $(10,4096)$\\
  \hline
  Example~\ref{ex:Ex4} & $2$ & $2$ & $(9,1)$, $(10,1)$\\
\hline\hline
\end{tabular}
\vspace{0.5em}
\caption{This table gives the sizes of the set
$\mathcal{Z}$ and the number of resulting $(\tau,\delta_1,\dotsc,\delta_r)$
for Examples \ref{ex:Ex1}--\ref{ex:Ex4}. The last column is a list of pairs
$(r,t)$ meaning there are $t$ cases where the $S$-unit
rank is $r$.}\label{table1}
\end{table}


\section{Lower Bounds for Linear Forms in Logarithms}\label{sec:lower-bounds}

In this section, we state theorems of Matveev \cite{Matveev} and of Yu \cite{Yu} for lower bounds for linear forms in complex and $p$-adic logarithms. In the next section, we will use these results to obtain bounds for $b_1, \dots, b_r$. We begin by establishing some notation, as well as some key results which we will need for the lower bound.

\begin{tabular}{ccl}
$L$ & \qquad & a number field.\\
$D$ & & the degree $[L:\Q]$.\\
$M_L$ & \qquad & the set of all places of $L$.\\
$M_L^\infty$ & & the subset of infinite places.\\
$M_L^{0}$ & & the subset of finite places.\\
$\nu$ & & a place of $L$.\\
$D_\nu$ & & the local degree $[L_\nu : \Q_\nu]$.\\
$\lvert \cdot \rvert_\nu$ & & the usual normalized absolute value
associated to $\nu$. \\
&& If $\nu$ is infinite and associated to
a real or complex embedding $\sigma$ of $L$, \\
&& then $\lvert \alpha \rvert_\nu=\lvert \sigma(\alpha) \rvert$.\\
&& If $\nu$ is finite and above the rational prime $p$,
then $\lvert p \rvert_\nu=p^{-1}$.\\
$\lVert \cdot \rVert_\nu$ && $=\lvert \cdot \rvert_\nu^{D_\nu}$.\\
$h(\cdot)$ && the absolute logarithmic height, defined in
\eqref{eqn:heightdef}.
\end{tabular}
In the above notation, the product formula may be stated as
\begin{equation}\label{eqn:productformula}
\prod_{\nu \in M_L} \lVert \alpha \rVert_\nu=1
\end{equation}
for all $\alpha \in L^\times$.
In particular, if $\nu$
is infinite, corresponding to a real or complex
embedding $\sigma$ of $L$, then
\begin{equation}\label{eqn:arch}
\lVert \alpha \rVert_\nu=
\begin{cases}
\lvert \sigma(\alpha) \rvert & \text{if $\sigma$ is real,}\\
\lvert \sigma(\alpha) \rvert^2 & \text{if $\sigma$ is complex.}
\end{cases}
\end{equation}
If $\nu$ is finite, and $\mP$ is the prime ideal corresponding to
$\nu$, then for $\alpha \in L^\times$ we have
\begin{equation}\label{eqn:normval}
\lVert \alpha \rVert_\nu=\Norm(\mP)^{-\ord_\mP(\alpha)};
\end{equation}
this easily follows from $D_\nu=e(\mP|p)f(\mP|p)$, where
$e(\mP|p)$ and $f(\mP|p)$ are respectively the
ramification index and the inertial degree of $\mP$.

For $\alpha \in L$, we define the absolute logarithmic height
$h(\alpha)$ by
\begin{equation}\label{eqn:heightdef}
h(\alpha)=
\frac{1}{[L:\Q]} \sum_{\nu \in M_L}
D_\nu \log \max\{1, \lvert \alpha \rvert_\nu\}
=
\frac{1}{[L:\Q]} \sum_{\nu \in M_L}
\log \max\{1, \lVert \alpha \rVert_\nu\}.
\end{equation}
\begin{lem}\label{lem:conjheight}
The absolute logarithmic height of an algebraic number $\alpha$ is independent
of the number field $L$ containing $\alpha$.
Moreover, if $\alpha$ and $\beta$ are Galois conjugates, then
$h(\alpha)=h(\beta)$.
\end{lem}
For proofs of the following two lemmata,
see \cite[Lemma 4.1]{IntegralHyp}
and \cite[Lemma 3.2]{Homero}.
\begin{lem}\label{lem:heightsum}
For $\alpha_1,\dotsc,\alpha_n \in L$ we have
\[
h(\alpha_1\cdots \alpha_n) \le h(\alpha_1)+\cdots+h(\alpha_n),
\quad
h(\alpha_1+\cdots+\alpha_n) \le \log{n}+h(\alpha_1)+\cdots+h(\alpha_n).
\]
For any $\alpha \in L^\times$, we have $h(\alpha)=h(\alpha^{-1})$.
Moreover, for any place $\nu \in M_L$,
\begin{equation}\label{eqn:logheightbd}
\log{\lVert \alpha \rVert_\nu} \le [L:\Q] \cdot h(\alpha).
\end{equation}
\end{lem}

\begin{lem}\label{lem:minimal}
Let $L$ be a number field of degree $D$. Let $\sS$
be a finite set of finite places of $L$. Let $\varepsilon \in \OO_\sS^\times$.
Let $\eta \in M_L$ be a place of $L$ chosen so that
$\lVert \varepsilon \rVert_\eta$ is minimal. Then
$\lVert \varepsilon \rVert_\eta \le 1$ and
\[
h(\varepsilon) \le \frac{(\#M_L^\infty+\#\sS)}{D} \cdot \log(\lVert \varepsilon^{-1} \rVert_\eta).
\]
\end{lem}


\subsection{Lower bounds for linear forms in $\mP$-adic logarithms}
Let $L$ be a number field of degree $D$. Let $\mP$ be a prime
ideal of $\OO_L$ and let $p$ be the rational prime below $\mP$.
Let $\nu \in M_L^0$ correspond to $\mP$.
Let
$\alpha_1, \dotsc,\alpha_n \in L^\times$. Write
$e=\exp(1)$.

Let
\begin{equation}
\begin{aligned}
h_j \; & := \; \max\left\{ h(\alpha_j), \, \frac{1}{16 e^2 D^2 }\right\}, \qquad
j=1,\dotsc,n \, ; \\
c_1(n,D) \; & := \; (16 eD)^{2n+2} \cdot n^{5/2} \cdot \log(2nD) \cdot
\log(2D) \, ;\\
c_2(n,\mP) \; & := \; e(\mP|p)^{n} \cdot \frac{p^{f(\mP|p)}}{f(\mP|p) \cdot \log{p}} \, ;\\
c_3(n,D,\mP,\alpha_1,\dotsc,\alpha_n) \; & := \; c_1(n,D) \cdot c_2(n,\mP) \cdot
h_1 \cdots h_n \, .
\end{aligned}
\end{equation}
We shall make use of the following theorem of Yu \cite{Yu}.
\begin{thm}[K.\ Yu] \label{thm:Yu}
Let $b_1,\dotsc,b_n$ be rational integers and let
\[
	B=\max\{ \lvert b_1 \rvert,\dotsc,\lvert b_n \rvert\},
\]
and suppose $B \ge 3$.
Let
\[
\Lambda=\alpha_1^{b_1} \cdots \alpha_n^{b_n}-1,
\]
and suppose $\Lambda \ne 0$.
Then
\[
\log{\lVert \Lambda^{-1}\rVert_\nu} \; < \; c_3(n,D,\mP,\alpha_1,\dotsc,\alpha_n) \cdot  \log{B}.
\]
\end{thm}
\begin{proof}
Let
\[
  c_4(n,D,\mP) \; := \; \frac{c_1(n,D)\cdot c_2(n,\mP)}{n\cdot f(\mP|p)\cdot \log{p}}.
\]
As stated in \cite[page 190]{Yu},
a consequence of Yu's Main Theorem is
\[
\ord_\mP(\Lambda) \; <  \; n \cdot c_4(n,D,\mP) \cdot h_1 \cdots h_n \cdot \log{B}.
\]
By \eqref{eqn:normval} we have
\[
\log{\lVert \Lambda^{-1} \rVert_\nu}
\; =\; \log{(\Norm(\mP))} \cdot \ord_\mP(\Lambda)
\; = \; f(\mP|p) \cdot \log{p} \cdot \ord_\mP(\Lambda).
\]
The theorem follows.
\end{proof}


\subsection{Lower bounds for linear forms in real or complex logarithms}
We continue with the above notation.
Let
\begin{equation}\label{eqn:hjprime}
h_j^\prime= \sqrt{h(\alpha_j)^2+\frac{\pi^2}{D^2}}, \qquad j=1,\dotsc,n.
\end{equation}
The following theorem is a version of Matveev's bound for
linear forms in logarithms \cite{Matveev}.
\begin{thm}[Matveev]\label{thm:Matveev}
Let $\nu$ be an infinite place of $L$.
Suppose
 $\Lambda \ne 0$.
Let
\[
c_5(n,D,\alpha_1,\dotsc,\alpha_n)=
6 \cdot 30^{n+4} \cdot (n+1)^{5.5} \cdot
D^{n+2} \cdot \log(eD) \cdot h_1^\prime \cdots h_n^\prime.
\]
Then
\[
\log{\lVert \Lambda^{-1} \rVert_{\nu}} \; \le \;
  c_5(n,D,\alpha_1,\dotsc,\alpha_n) \cdot (\log(en)+\log{B}).
\]
\end{thm}
\begin{proof}
This in fact follows from a version of Matveev's theorem derived
in \cite{BMS1}. Let $\sigma$ be a real or complex
embedding of $L$ corresponding to $\nu$.
Let
\[
h_j^{\prime\prime}=\max\left\{D h(\alpha_j), \; \lvert \log(\sigma(\alpha_j)) \rvert, \; 0.16\right\},
\]
where $\log(\sigma(\alpha_j))$ is the principal determination of
the logarithm
(i.e.\ the imaginary part of
$\log$ lies in $(-\pi,\pi]$). Let
\[
c_6(n,D)=3 \cdot 30^{n+4} \cdot (n+1)^{5.5} \cdot
D^{2} \cdot \log(eD).
\]
Then Theorem 9.4 of \cite{BMS1} asserts that
\[
\log{\lvert \Lambda \rvert_\nu}
\ge -c_6(n,D) \cdot h_1^{\prime\prime} \cdots h_n^{\prime\prime} \cdot (\log(en)+\log{B}).
\]
Since $\lVert \Lambda \rVert=\lvert \Lambda \rvert^{D_\nu}$,
where $D_\nu$ is either $1$ or $2$, we have
\[
\log{\lVert \Lambda^{-1} \rVert} \le 2 \cdot c_6(n,D) \cdot h_1^{\prime\prime} \cdots h_n^{\prime\prime} \cdot (\log(en)+\log{B}).
\]
Thus it is sufficient to show that
$h_j^{\prime\prime} \le D h_j^\prime$. However,
\[
\log(\sigma(\alpha_j))=\log \lvert \sigma(\alpha_j) \rvert+i \theta
\]
where $-\pi<\theta\le \pi$. But by \eqref{eqn:logheightbd} and $\lVert \cdot \rVert_\nu=\lvert \cdot \rvert_\nu^{D_\nu},$
we have
\[
\log \lvert \sigma(\alpha_j) \rvert=\frac{1}{D_\nu}
\log \lVert \alpha_j \rVert_\nu \le \log \lVert \alpha_j \rVert_\nu
\le D \cdot h(\alpha_j).
\]
Thus
\[
\lvert \log(\sigma(\alpha_j)) \rvert \; \le \;
\sqrt{D^2 \cdot h(\alpha_j)^2+\pi^2}\; =\; D \cdot h_j^\prime.
\]
It is now clear that $h_j^{\prime\prime} \le D \cdot h_j^\prime$.
\end{proof}


\section{The $S$-Unit Equation}\label{sec:sunit}

We now return to the task of studying
the solutions of \eqref{eqn:TMdelta} satisfying
\eqref{eqn:restrictions}, \eqref{eqn:restrictions2}. Here, we use the theorems of Matveev and Yu (recalled in the previous section)
to establish bounds for $b_1, \dots, b_r$,
following the ideas of \cite{BG}, \cite{IntegralHyp} and \cite{Homero}, and taking
care to keep our constants completely explicit and as small as possible.

We begin by establishing the following notation:

\begin{tabular}{rcl}
$\theta$, $K$ & \qquad & as defined in Section~\ref{sec:intro}.\\
$d$ & & the degree $[K:\Q] \ge 3$.\\
$S$ & & a set $\{\fp_1,\dotsc,\fp_s\}$ of prime ideals of $K$
satisfying \\
& & conditions (a), (b) of Proposition~\ref{prop:Sprop}.\\
$s$ & & the cardinality $\#S$ of the set $S$.\\
$\delta_1,\dotsc,\delta_r$ & & a basis for the $S$-unit group
$\OO_S^\times$ modulo torsion,\\
& & also appearing in \eqref{eqn:TMdelta}.\\
$\tau$ & & a non-zero element of $K$, appearing in \eqref{eqn:TMdelta}.\\
$X$, $Y$, $b_1,\dotsc,b_r$ & & a solution to \eqref{eqn:TMdelta}
	satisfying \eqref{eqn:restrictions}, \eqref{eqn:restrictions2}.\\
$\varepsilon$ & & $=\delta_1^{b_1} \cdots \delta_r^{b_r}$. Thus
                  \eqref{eqn:TMdelta} can be rewritten as \\
& & $a_0 X-\theta Y= \tau \cdot \varepsilon$.\\
$\mu$ & & $=\tau \cdot \varepsilon=a_0 X- \theta Y$.\\
	$B$  & & $=\max\{\lvert b_1 \rvert, \dotsc,\lvert b_r\rvert\}$.\\
$\theta_1$, $\theta_2$, $\theta_3$ & & three conjugates of $\theta$
chosen below, with $\theta=\theta_1$.\\
$L$ & & the extension $\Q(\theta_1,\theta_2,\theta_3) \supseteq K$.\\
$D$ & & the degree $[L:\Q]$.\\
$\sigma_i$ & & the embedding $K \hookrightarrow L$, $\theta \mapsto \theta_i$.\\
$\mu_i$, $\varepsilon_i$, $\tau_i$, $\delta_{j,i}$
& & the images of $\mu$, $\varepsilon$, $\tau$, $\delta_j$ under the
embedding $\sigma_i$.\\
$\xi_1$, $\xi_2$, $\xi_3$ & & defined in \eqref{eqn:xi}.\\
\end{tabular}


We want to write down an $S$-unit equation starting
with \eqref{eqn:TMdelta}. For this we will need to work
with three conjugates of $\theta$. Let $d=[K:\Q]$,
and let $\theta_1,\dotsc,\theta_d$ be the conjugates
of $\theta$ in some splitting field $M \supseteq K$.
We shall not need $M$ explicitly, but we assume that
we are able to compute the Galois group $G$ of $M/\Q$
as a transitive permutation group on the conjugates $\theta_i$.
From this we are able to list
all subgroups (up to conjugacy), and for each subgroup
determine if it fixes at least three conjugates of $\theta$.
Let $H$ be a subgroup of $G$ fixing at least three conjugates
of $\theta$ with index $[G:H]$ as small as possible.
Let $L=M^H$ be the fixed field of $H$. Then $L=\Q(\theta_1,\theta_2,\theta_3)$
for some three conjugates of $\theta$ (after a possible reordering
of conjugates) and it has the property that its degree
is minimal amongst all extensions generated by three conjugates.
Write
\[
D:=[G:H]=[L:\Q].
\]
Again we shall not need the field $L$ explicitly,
but only its degree $D$, which we can deduce from the Galois group.
We identify $\theta=\theta_1$, and so can think of $K \subseteq L$.

%
%
Write
$\mu=a_0 X-\theta Y$.
Let $\varepsilon=\delta_1^{b_1} \cdots \delta_r^{b_r}$.
Then $\mu=\tau \cdot \varepsilon$.
Let $\mu_i$, $\varepsilon_i$, $\tau_i$, $\delta_{j,i}$ be the images of $\mu$, $\varepsilon$, $\tau$,
$\delta_j$ under the embeddings $\sigma_i : K \hookrightarrow L$,
$\theta \mapsto \theta_i$. We observe the following Siegel identity:
\[
(\theta_3-\theta_2) \mu_1+(\theta_1-\theta_3) \mu_2 + (\theta_2-\theta_1) \mu_3=0.
\]
Let
\begin{equation}\label{eqn:xi}
\xi_1=(\theta_3-\theta_2) \cdot \tau_1,
\qquad
\xi_2=(\theta_1-\theta_3) \cdot \tau_2,
\qquad
\xi_3=(\theta_2-\theta_1) \cdot \tau_3.
\end{equation}
Then
\begin{equation}\label{eqn:sunit}
\xi_1 \varepsilon_1+\xi_2 \varepsilon_2+\xi_3 \varepsilon_3=0.
\end{equation}
Note that $\varepsilon_1$ is an $S$-unit in $K$ and $\varepsilon_2$,
$\varepsilon_3$ are Galois conjugates of $\varepsilon$. This
equation will serve as our $S$-unit equation.
We would like to rewrite \eqref{eqn:sunit} in a manner
that makes it convenient to apply Theorems~\ref{thm:Yu}
and~\ref{thm:Matveev}. Observe that \eqref{eqn:sunit} can
be rewritten as
\begin{equation}\label{eqn:sunit2}
\frac{\xi_3 \varepsilon_3}{\xi_1 \varepsilon_1} \; = \;
\left(\frac{-\xi_2}{\xi_1}\right)
\left(\frac{\varepsilon_2}{\varepsilon_1}\right)
\; - \; 1.
\end{equation}
Let
\[
\alpha_j:=\frac{\delta_{j,2}}{\delta_{j,1}} \quad (j=1,\dotsc,r),
\qquad
\alpha_{r+1}:=\frac{-\xi_2}{\xi_1}, \qquad b_{r+1}=1.
\]
Then
\begin{equation}\label{eqn:sunit3}
\frac{\xi_3 \varepsilon_3}{\xi_1 \varepsilon_1} \; = \;
\Lambda
\end{equation}
where
 $\Lambda$ is the  \lq\lq linear form\rq\rq
\[
\Lambda: \; = \; \alpha_1^{b_1} \cdots \alpha_{r+1}^{b_{r+1}} \; - \; 1.
\]
We assume that we know $\theta$, $\tau$ and $\delta_1,\dotsc,\delta_r$
explicitly and can therefore compute their absolute logarithmic
heights. We will use this to estimate the heights of other
algebraic numbers, such as $\xi_i$, $\alpha_j$, without computing their
minimal polynomials.
By Lemmas~\ref{lem:conjheight} and \ref{lem:heightsum},
\[
h(\xi_i) \; \le \; c_7, \qquad c_7 \; :=\; \log{2}+2h(\theta)+h(\tau).
\]
\begin{lem}\label{lem:difference}
Let
\[
c_{8} \; := \; 2D c_7.
\]
Let $\nu$ be any place of $L$. Then
\[
\log{\lVert (\varepsilon_3/\varepsilon_1)^{-1} \rVert_\nu}
\; \le \;
\log{\lVert \Lambda^{-1} \rVert_\nu} \; + \; c_{8}.
\]
\end{lem}
\begin{proof}
Note that
\[
\log{\lVert (\varepsilon_3/\varepsilon_1)^{-1} \rVert_\nu}
\; = \;
\log{\lVert \Lambda^{-1} \rVert_\nu} \; + \;
\log{\lVert \xi_3/\xi_1 \rVert_\nu}.
\]
By Lemma~\ref{lem:heightsum}
\[
\log{\lVert \xi_3/\xi_1 \rVert_\nu}
\; \le \;  D \cdot h(\xi_3/\xi_1)
\; \le \; D \cdot (h(\xi_3)+h(\xi_1)) \, .
\]
\end{proof}
By definition $B=\max\{\lvert b_1 \rvert,\dotsc,\lvert b_r \rvert\}$.
However, by \eqref{eqn:restrictions2}, and since $b_{r+1}=1$,
we have
\[
B = \max\{ \lvert b_1 \rvert, \dotsc, \lvert b_r \rvert, \lvert b_{r+1} \rvert\}.
\]
We now apply Matveev's theorem in order to obtain a bound for $B$.
\begin{lem}\label{lem:Matveev}
Let
\[
h_j^* \; := \; \sqrt{4 h(\delta_j)^2 + \frac{\pi^2}{D^2}}
\quad \text{for $j=1,\dotsc,r$},
\quad \text{and} \qquad
h_{r+1}^* \; := \; \sqrt{4 c_7^2 + \frac{\pi^2}{D^2}} .
\]
Let
\[
c_{9} \; = \;
6 \cdot 30^{r+5} \cdot (r+2)^{5.5} \cdot
D^{r+3} \cdot \log(eD) \cdot h_1^* \cdots h_{r+1}^*\, ,
\]
and
\[
c_{10} \; = \; c_{8}+c_{9} \cdot \log(e(r+1)).
\]
Let $\nu$ be an infinite place of $L$.
Then
\[
\log{\lVert (\varepsilon_3/\varepsilon_1)^{-1} \rVert_\nu}
\; \le \;
c_{10} \; + \; c_{9} \cdot \log{B}.
\]
\end{lem}
\begin{proof}
We will apply Theorem~\ref{thm:Matveev} with $n=r+1$.
Observe that
\[
h(\alpha_j) \le 2 h(\delta_j) \qquad \text{for $j=1,\dotsc,r$}, \quad
\text{and} \qquad h(\alpha_{r+1}) \le 2 c_7.
\]
Thus $h_j^\prime \le h_j^*$ where $h_j^\prime$ is defined in
\eqref{eqn:hjprime}.
By Theorem~\ref{thm:Matveev},
\[
\log{\lVert \Lambda^{-1} \rVert_\nu} \; \le \; c_{9} \cdot (\log(e(r+1))+\log{B}).
\]
Lemma~\ref{lem:difference} completes the proof.
\end{proof}
We also apply Yu's theorem.
\begin{lem}\label{lem:Yu}
Let
\[
h_j^{\dagger} \;  := \; \max\left\{ 2 h(\delta_j), \, \frac{1}{16 e^2 D^2 }\right\} \qquad
\text{for $j=1,\dotsc,r$},
\]
and
\[
h_{r+1}^{\dagger} \;  := \; \max\left\{ 2 c_7, \, \frac{1}{16 e^2 D^2 }\right\} \, .
\]
Let $T$ be the set of rational primes $p$ below
the primes $\fp \in S$. Let
\[
c_{11}:=
\max_{p \in T} \;  \max\left\{  \frac{u^{r+1} \cdot p^v}{v \cdot \log{p}} \;
: \; \text{$u$, $v$ are positive integers and $\; uv \le D/d$}
 \right\},
\]
and
\[
c_{12} := c_1(r+1,D) \cdot c_{11} \cdot h_1^{\dagger} \cdots h_{r+1}^\dagger.
\]
Let $\nu$ be a finite place of $L$. Then
\[
\log{\lVert (\varepsilon_3/\varepsilon_1)^{-1} \rVert_\nu}
\; \le \;
c_{8}+c_{12} \log{B}.
\]
\end{lem}
\begin{proof}
Of course we may suppose that
$\lVert (\varepsilon_3/\varepsilon_1)^{-1} \rVert_\nu \ne 1$.
Let $\mP$ be the prime ideal of $\OO_L$ corresponding to $\nu$.
We will deduce the lemma from Theorem~\ref{thm:Yu} combined
with Lemma~\ref{lem:difference}. For this,
it suffices to show that
${c_3(r+1,D,\mP,\alpha_1,\dotsc,\alpha_{r+1}) \le c_{12}}$. Observe
that $h_j \le h_j^\dagger$ for $j=1,\dotsc,r+1$. Thus
it is enough to show that $c_2(r+1,\mP) \le c_{11}$.

Let $K_i=\Q(\theta_i) \subseteq L$.
Recall that $\varepsilon$ is an $S$-unit and that $\varepsilon_i$
is the image of $\varepsilon$ under the map $\sigma_i : K \rightarrow K_i$,
$\theta \mapsto \theta_i$.
As $\lVert (\varepsilon_3/\varepsilon_1)^{-1} \rVert_\nu \ne 1$,
we see that $\ord_\mP(\varepsilon_i) \ne 0$ for $i=1$ or $3$. Thus $\mP$ must be a prime above $\fp_i:=\sigma_i(\fp)$
of $\OO_{K_i}$ for some $\fp \in S$,
where $i=1$ or $3$. In particular $\mP$ is above some rational prime $p \in T$.
However, $e(\fp|p)=f(\fp|p)=1$ for all $\fp \in S$.
Thus $e(\mP|p)=e(\mP|\fp_i)$ and $f(\mP|p)=f(\mP|\fp_i)$
for $i=1$ or $3$. Let
$u=e(\mP|p)$ and $v=f(\mP|p)$. We see that
$uv=e(\mP|\fp_i)f(\mP|\fp_i)\leq[L:K_i]=D/d$.
Now $c_2(r+1,\mP) \le c_{11}$ follows from the definitions
of $c_2$ and $c_{11}$.
\end{proof}

\begin{lem}\label{lem:mulogB}
Let
\[
c_{13}\; :=\; \frac{\# M_K^\infty+2\cdot \#S}{d},
\]
and
\[
c_{14} \; := \; 2 h(\tau) + c_{13} \cdot \max(c_{8},c_{10}),
\qquad c_{15} \; := \; c_{13} \cdot \max(c_{9},c_{12}).
\]
Then
\[
h(\mu_3/\mu_1)
 \; \le \; c_{14} \, +\, c_{15} \cdot \log{B}.
\]
\end{lem}
\begin{proof}
Let $\sS$ be the prime ideals appearing in the support
of $\varepsilon_3/\varepsilon_1$.
We will show that $\# \sS \le (2D/d) \cdot \#S$.
Indeed, $\varepsilon_i$ belongs to $K_i=\Q(\theta_i)$
and its support in $K_i$ is contained in $\sigma_i(S)$. Now a prime belonging to
$\sigma_i(S)$ has at most $[L:K_i]=D/d$
primes above it in $L$. Thus
\[
\#\sS \; \le\; (D/d) \cdot \# \sigma_1(S)+(D/d) \cdot \# \sigma_3(S) \; \le
(2D/d) \cdot \#S
\]
as required.
Moreover, since $[L:K]=D/d$, we have $\#M_L^\infty \le (D/d) \cdot \# M_K^\infty$.

Let $\eta \in M_L$
be the place of $L$ such that $\lVert \varepsilon_3/\varepsilon_1 \rVert_\eta$
is minimal. By Lemma~\ref{lem:minimal}
\[
h(\varepsilon_3/\varepsilon_1) \; \le \; \frac{\#M_L^\infty+\#\sS}{D} \cdot \lVert (\varepsilon_3/\varepsilon_1)^{-1} \rVert_\eta.
\]
From the above inequalities for $\#M_L^\infty$ and $\#\sS$, we deduce that
\[
h(\varepsilon_3/\varepsilon_1) \; \le \;
c_{13} \cdot \lVert (\varepsilon_3/\varepsilon_1)^{-1} \rVert_\eta.
\]
We now apply Lemmas~\ref{lem:Matveev} and~\ref{lem:Yu} to obtain
\[
h(\varepsilon_3/\varepsilon_1) \; \le \;
c_{13} \cdot \left(\max(c_8,c_{10})\; +\; \max(c_9,c_{12}) \cdot \log{B} \right).
\]
Finally, observe that $\mu_3/\mu_1=(\tau_3/\tau_1) \cdot (\varepsilon_3/\varepsilon_1)$
and thus
\[
h(\mu_3/\mu_1)
\; \le \; 2 h(\tau) \, + \,  h(\varepsilon_3/\varepsilon_1).
\]
\end{proof}

We shall henceforth suppose $(X,Y) \ne (\pm 1,0)$. 
As $\gcd(X,Y)=1$, this is equivalent to $Y \ne 0$.
\begin{lem}\label{lem:quotient}
Let $\nu$ be a place of $L$.
Let
\[
\kappa_{\nu}=\begin{cases}
1 & \text{if $\nu \in M_L^0$},\\
1/2^{D_\nu} & \text{if $\nu \in M_L^\infty$}.
\end{cases}
\]
Then
\[
\max\{ \lVert \mu_1 \rVert_\nu \, ,\, \lVert \mu_3 \rVert_\nu \}^2 \; \ge \;
\kappa_\nu^2 \cdot \min\left\{1, \lVert \theta_1 -\theta_3 \rVert_\nu\right\}^2
\cdot
\max\{ 1 \, ,\, \lVert \mu_1 \rVert_\nu \} \cdot
\max\{ 1 \, ,\, \lVert \mu_3 \rVert_\nu \} \, .
\]
\end{lem}
\begin{proof}
There is nothing to prove unless, $\lVert \mu_i \rVert_\nu \le 1$
for both $i=1$, $3$. In this case, it is enough to show that
\begin{equation}\label{eqn:valineq}
\max\{ \lVert \mu_1 \rVert_\nu \, ,\, \lVert \mu_3 \rVert_\nu \} \; \ge \;
\kappa_\nu \cdot \lVert \theta_1 -\theta_3 \rVert_\nu .
\end{equation}
Suppose first that $\nu$ is finite and let
$\mP$ be the prime ideal of $\OO_L$ corresponding to $\nu$.
Then \eqref{eqn:valineq} is equivalent to
\[
\min\{ \ord_\mP(\mu_1) \, ,\, \ord_\mP(\mu_3)  \} \; \le \;
\ord_\mP(\theta_1 -\theta_3) .
\]
Let $k=\min\{ \ord_\mP(\mu_1) \, ,\, \ord_\mP(\mu_3)  \}$.
Then $\mP^k \mid \mu_i$ for $i=1$, $3$.
Recall that ${\mu_i=a_0 X- \theta_i Y}$. Thus $\mP^k$ divides both
\[
(\theta_1-\theta_3) a_0 X\; = \; \theta_1 \mu_3-\theta_3 \mu_1 \quad \text{and} \quad
(\theta_1-\theta_3) Y \; =\; \mu_3-\mu_1.
\]
However $\gcd(X,Y)=\gcd(a_0,Y)=1$, thus $\mP^k \mid (\theta_1-\theta_3)$
as desired.

Next suppose $\nu$ is infinite.
As $(\theta_1-\theta_3)Y=\mu_3-\mu_1$,
and $Y \ne 0$ is a rational integer, we have
\[
\begin{split}
\lVert \theta_1-\theta_3 \rVert_\nu \; & \le \;
\lVert \mu_1-\mu_3 \rVert_\nu \; \\
& = \;
\lvert \mu_1 -\mu_3 \rvert_\nu^{D_\nu}
\;\\
& \le \;
2^{D_\nu} \cdot \max\{ \lvert \mu_1 \rvert_\nu,~\lvert \mu_3 \rvert_\nu\}^{D_\nu} \\
& \le \; 2^{D_\nu} \cdot \max\{ \lVert \mu_1 \rVert_\nu,~\lVert \mu_3 \rVert_\nu\}.
\end{split}
\]
This completes the proof.
\end{proof}

\begin{lem}
Let
\[
c_{16} \; := \; c_{14}+ 2 \log{2}+ 2 h(\theta)+h(\tau) \, .
\]
Then
\begin{equation}\label{eqn:epslogB}
h(\varepsilon) \; \le \; c_{16}+c_{15} \cdot \log{B}\, .
\end{equation}
\end{lem}
\begin{proof}
First note that
\[
\begin{split}
h(\mu_3/\mu_1) \;
& = \;
\frac{1}{D} \sum_{\nu \in M_L} \log\max\left\{1,~\lVert \mu_3/\mu_1 \rVert_\nu \right\} \\
& = \;
\frac{1}{D} \sum_{\nu \in M_L} \log\max\left\{1,~\lVert \mu_3/\mu_1 \rVert_\nu \right\} + \frac{1}{D} \sum_{\nu \in M_L}\log{\lVert \mu_1 \rVert_\nu} \quad \text{(from
\eqref{eqn:productformula})}\\
& = \;
\frac{1}{D} \sum_{\nu \in M_L} \log\max\{\lVert \mu_1\rVert,~\lVert \mu_3 \rVert_\nu \} \\
& \ge \; \frac{1}{2}(h(\mu_1)+h(\mu_3)) + \frac{1}{D} \sum_{\nu \in M_L}\left(\log{\kappa_\nu} + \log \min\{1,~\lVert(\theta_1-\theta_3) \rVert_\nu\} \right)
\end{split}
\]
by Lemma~\ref{lem:quotient}. However $\mu_1$, $\mu_3$ are conjugates
of $\mu$, thus $h(\mu_1)=h(\mu_3)=h(\mu)$. Moreover,
\[
\frac{1}{D} \sum_{\nu \in M_L}\log{\kappa_\nu}
 \; = \;
- \frac{\log{2}}{D} \sum_{\nu \in M_L^\infty} D_\nu
\; = \; -\log{2} \, .
\]
Thus
\[
\begin{split}
h(\mu) \;  & \le \;
h(\mu_3/\mu_1) \, + \, \log{2} \, - \,
\frac{1}{D} \sum_{\nu \in M_L} \log \min\{1,~\lVert(\theta_1-\theta_3) \rVert_\nu\}\\
& = \;
h(\mu_3/\mu_1) \, + \, \log{2} \, + \,
\frac{1}{D} \sum_{\nu \in M_L} \log \max\{1,~\lVert(\theta_1-\theta_3)^{-1} \rVert_\nu\}\\
& = \;
h(\mu_3/\mu_1) \, + \, \log{2} \, + \,  h((\theta_1-\theta_3)^{-1})\\
& \le \;
h(\mu_3/\mu_1) \, + \, 2 \log{2} \, + \,  2 h(\theta) \, ,
\end{split}
\]
by Lemmas~\ref{lem:conjheight} and \ref{lem:heightsum}.
But $\varepsilon=\tau^{-1} \mu$, thus
\[
h(\varepsilon) \; \le \; h(\mu_3/\mu_1) \, + \, 2 \log{2} \, + \,  2 h(\theta) \, + \, h(\tau) \, .
\]
Applying Lemma~\ref{lem:mulogB} completes the proof.
\end{proof}

It is worthwhile to take stock for a moment. The inequality
\eqref{eqn:epslogB} relates the height of
$\varepsilon=\delta_1^{b_1} \cdots \delta_r^{b_r}$
to $B=\max\{\lvert b_1 \rvert,\dotsc,\lvert b_r \rvert\}$.
The constants $c_7,\dotsc,c_{16}$ are given explicitly
in terms of $\theta$, $\tau$, $\delta_1,\dotsc,\delta_r$
(all belonging to $K$), the prime ideals of $K$ belonging to $S$,
the signature of $K$,
 and the degree $D$, which can
be deduced from the Galois group of the minimal
polynomial of $\theta$. We do not need the field $L$
explicitly.

\begin{lem}\label{lem:htlbB}
Let $U$ be any subset of $S \cup M_K^{\infty}$ of size $r$.
Let $\cM$ be the $r\times r$-matrix
\[
\cM=(\, \log{\lVert \delta_j \rVert_\nu} \, )_{\nu \in U,~1 \le j \le r} \, .
\]
The matrix $\cM$ is invertible. Let $c_{17}$ be the largest
of the absolute values of the entries of $\cM^{-1}$.
Then
\[
B \;  \le \;  2d \cdot c_{17} \cdot h(\varepsilon).
\]
\end{lem}
\begin{proof}
The determinant of $\cM$ is in fact
\[
\pm \left(\prod_{\nu \in U} D_\nu\right) \cdot R(\delta_1,\dotsc,\delta_r)
\]
where $R(\delta_1,\dotsc,\delta_r)$ is the regulator of system
of $S$-units $\delta_1,\dotsc,\delta_r$,
and therefore does not vanish (c.f. \cite[Section 3]{BG}).
Consider the vectors
$\bb:=\big[b_j\big]_{j=1,\dotsc,r}$ and
$\uu:=\big[\log{\lVert \varepsilon \rVert_\nu} \big]_{\nu \in U}$
in $\R^r$. As $\varepsilon=\delta_1^{b_1} \cdots \delta_r^{b_r}$
we see that $\uu=\cM \bb$ and so $\bb=\cM^{-1} \uu$. It follows,
for $j=1,\dotsc,r$ that
\[
\begin{split}
\lvert b_j \rvert \; 
\; & \le \; c_{17} \cdot \sum_{\nu \in U} \lvert \log{\lVert \varepsilon \rVert_\nu}\rvert\\
& \le \; c_{17} \cdot \sum_{\nu \in M_K} \log{\max\{1,\lVert \varepsilon \rVert_\nu\}} + \log{\max\{1,\lVert \varepsilon^{-1} \rVert_\nu\}}\\
& = \; 2d \cdot c_{17} \cdot h(\varepsilon)
\end{split}
\]
as required.
\end{proof}

\noindent\textbf{Remark.} Observe that there are $r+1$ possibilities for the set $U$. To compute $c_{17}$ in practice, we iterate across all such sets and select $c_{17}$ as the smallest possible value across each of the associated $r+1$ matrices $\cM$.

\begin{prop}\label{prop:initialBd}
Let
\[
c_{18} \; := \; 2d \cdot c_{17} \cdot c_{16}\, ,
\qquad
c_{19} \; := \; 2d \cdot c_{17} \cdot c_{15}\, ,
\qquad
c_{20} \; := \; 2c_{18}\, +\,  \max\{2 c_{19} \log{c_{19}},~4e^2\}.
\]
Then
\begin{equation}\label{eqn:Bbound}
B \; \le \;  c_{20}.
\end{equation}
\end{prop}
\begin{proof}
Combining Lemma~\ref{lem:htlbB} with \eqref{eqn:epslogB} we have
\[
B
\; \le \; 2d \cdot c_{17} \cdot \big( c_{16}+c_{15} \cdot \log{B} \big)
\;  \le \; c_{18}+ c_{19} \log{B}.
\]
The Proposition follows from a result of
Peth\H{o} and de Weger (Lemma B.1 of \cite[Appendix B]{Smart}).
\end{proof}


\subsection{Example~\ref{ex:Ex4} continued} \label{page:taudeltatuple}

We give some further details for Example~\ref{ex:Ex4}. Here $a_0=5$ and
the minimal polynomial for $\theta$ is
\[
x^{11} + x^{10} + 20x^9 + 25 x^8 + 750 x^7 + 625 x^6 + 18750 x^5 + 468750 x^3 +
    7812500 x - 19531250.
  \]
The field $K=\Q(\theta)$ has degree $11$, and signature $(1,5)$.
In this case we have $2$ possibilities
for $(\tau,\delta_1,\dotsc,\delta_r)$; one has $S$-unit
rank $r=9$ and the other has $S$-unit rank $r=10$.
We take a closer look at one of these
$2$ possibilities, with $r=10$.
The set $S$ is composed of the following five primes of ramification
degree $1$ and inertial degree $1$:
\begin{multline}\label{eqn:primesS}
\fp_1=\langle 11, 3+\theta \rangle,
\quad
\fp_2=\langle 7, 1+\theta \rangle,
\quad
\fp_3=\langle 5, \phi \rangle, \\
\fp_4=\langle 3, 5+\theta \rangle,
\quad
\fp_5=\langle 2, 1+\theta \rangle,
\end{multline}
where
\begin{multline*}
\phi=
\frac{1}{5^9}(4\theta^{10} + 9\theta^9 + 185\theta^8 +
425\theta^7 + 4625\theta^6
        + 13750\theta^5 + 131250\theta^4\\
	+ 750000\theta^3 + 3203125\theta^2 +
        26953125\theta + 5859375).
\end{multline*}
The bound for $B$ given by Proposition~\ref{prop:initialBd}
is
\[
B \le 1.57 \times 10^{222}.
\]

\section{Controlling the Valuations of $a_0 X - \theta Y$}\label{sec:controlling-vals}

In this section and the next, we will suppose that we have a bound
\begin{equation}\label{eqn:cB0}
B \le \cB_{\infty}
\end{equation}
and we will explain a method for replacing this
bound with what is hopefully a smaller bound. Our
subsequent constants will depend on $\cB_{\infty}$.
Initially we may take $\cB_{\infty} =c_{20}$ by Proposition~\ref{prop:initialBd}.
However, if we succeed in obtaining a smaller bound for $B$,
we may replace $\cB_{\infty}$ by that bound and repeat the process.

We shall replace the reduction step
using linear forms in $\fp$-adic logarithms as in
the paper of Tzanakis and de Weger \cite{TW}. In particular
we will eliminate all computations with
completions of extensions of $K$, as these are extremely tedious
and error-prone.

\bigskip

\subsection{The bounds $\cB_\infty$, $\cB_1$ and $\cB_2$.}
Henceforth $\cB_\infty$, $\cB_1$ and $\cB_2$ will denote
the known bounds for the $\infty$-norm, $1$-norm and $2$-norm
of our exponent vector
$\bb=\big[b_j\big]_{j=1,\dotsc,r}$:
\begin{equation}\label{eqn:cB1}
	B:=\lVert \bb \rVert_\infty \le \cB_\infty, \qquad
\lVert \bb \rVert_1 \le \cB_1 \qquad
\lVert \bb \rVert_2 \le \cB_2.
\end{equation}
Initially,
thanks to Proposition~\ref{prop:initialBd}, we can make the assignments
\begin{equation}\label{eqn:cB0tocB1}
	\cB_{\infty} = c_{20}, \quad
\cB_2=\sqrt{r} \cdot \cB_{\infty}, \quad \cB_1 = r\cdot \cB_{\infty},
	\quad \text{(initial values for $\cB_\infty$, $\cB_1$, $\cB_2$)}.
\end{equation}
However, as we progress in our algorithm, we will update
the values of $\cB_\infty$, $\cB_1$, $\cB_2$
so that \eqref{eqn:cB1} is still satisfied.


Given a lattice $L \subseteq \Z^r$ and
a vector $\ww \in \Z^r$, we denote by
$D(L,\ww)$ the shortest length of a vector
belonging to the coset $\ww+L$. This value can be
computed using a closest vector algorithm.
Indeed, for $\vv \in \Z^r$, write $\cc(L,\vv)$ for the closest
vector in $L$ to $\vv$ (if there is more than one
at the closest distance, choose any of them).
\begin{lem}\label{lem:closest}
$D(L,\ww)=\lVert \ww+\cc(L,-\ww) \rVert_2$.
\end{lem}
\begin{proof}
Let $\bfl \in L$ and suppose
\[
\lVert \ww+\bfl \rVert_2 < \lVert \ww+\cc(L,-\ww) \rVert_2.
\]
Then
\[
\lVert \bfl-(-\ww) \rVert_2 < \lVert \cc(L,-\ww) -(-\ww)\rVert_2.
\]
Thus $\bfl$ is a vector belonging to $L$ that is strictly closer
to $-\ww$ than $\cc(L,-\ww)$ giving a contradiction.
\end{proof}

Our first goal is to use the bounds \eqref{eqn:cB1}
to deduce  bounds on the valuations $\ord_\fp(a_0 X- \theta Y)$
for $\fp \in S$.
\begin{prop}\label{prop:valbd}
Let $\fp \in S$ and let $p$ be the rational prime below $\fp$.
Let $k \ge 1$.
Then there is some $\theta_0 \in \Z$ such that
\[
\theta \equiv \theta_0 \pmod{\fp^k}.
\]
Write
\begin{equation}\label{eqn:facT1}
\fa:=(p \OO_K)/\fp, \qquad
\tau \OO_K =\cT_1/\cT_2,
\end{equation}
where $\cT_1$ and $\cT_2$ are coprime ideals. The following hold:
\begin{enumerate}
\item[(i)] If $\gcd(\fa^k,\theta-\theta_0) \ne \gcd(\fa^k,\cT_1)$
then
\begin{equation}\label{eqn:valub}
\ord_{\fp}(a_0 X-\theta Y) \le k-1.
\end{equation}
\end{enumerate}
Suppose $\gcd(\fa^k,\theta-\theta_0) = \gcd(\fa^k,\cT_1)$.
Let
\[
\fb:=\fa^k/\gcd(\fa^k,\cT_1) \, .
\]
Let
\[
k^\prime \; := \; \max_{\fq \mid \fb}\left\lceil \frac{\ord_{\fq}(\fb)}{e(\fq|p)}\right\rceil \, ;
\]
this satisfies $\fb \cap \Z=p^{k^\prime} \Z$
and therefore $(\Z/p^{k^\prime} \Z)^\times$
naturally injects into $(\OO_K/\fb)^\times$.
Given $u \in K^\times$ whose support is coprime with $\fb$, denote
its image in $(\OO_K/\fb)^\times/(\Z/p^{k^\prime} \Z)^\times$
by $\overline{u}$.
Let
\[
\phi: \Z^r \rightarrow (\OO_K/\fb )^\times / (\Z/p^{k^\prime}\Z)^\times,
\qquad (n_1,\dotsc,n_r) \mapsto \overline{\delta_1}^{n_1} \cdots \overline{\delta_r}^{n_r}.
\]
Write
\[
\tau_0:=\frac{(\theta_0-\theta)}{\tau}.
\]
Then the support of $\tau_0$ is coprime with $\fb$.
\begin{enumerate}
\item[(ii)] If $\overline{\tau_0}$ does not belong to
$\Img(\phi)$ then \eqref{eqn:valub} holds.
\item[(iii)] Suppose $\overline{\tau_0}=\phi(\ww)$ for some $\ww \in \Z^r$. Let $L=\Ker(\phi)$ and suppose $D(L,\ww) >\cB_2$.
Then \eqref{eqn:valub} holds.
\end{enumerate}
\end{prop}
\begin{proof}
Let $\fp$, $p$, and $k$ be as in the statement of the proposition.
We suppose that
\begin{equation}\label{eqn:kbound}
\ord_\fp(a_0X-\theta Y)  \ge k
\end{equation}
and show that this leads to a contradiction under
the hypotheses of any of (i), (ii), (iii).
Recall $k \ge 1$. From the proof of Lemma~\ref{lem:adequate}, we know $p \nmid Y$.

Since $e(\fp|p)=f(\fp|p)=1$, we have $\OO_K/\fp^k \cong \Z/p^k$.
Thus there is some $\theta_0 \in \Z$
such that $\theta-\theta_0 \equiv 0 \pmod{\fp^k}$.
However,
$a_0 X-\theta Y \equiv 0 \pmod{\fp^k}$
and so therefore
$a_0 X- \theta_0 Y \equiv 0 \pmod{\fp^k}$.
However $a_0 X -\theta_0 Y \in \Z$. Thus, recalling that $e(\fp|p)=1$,
we have
$a_0 X-\theta_0 Y \equiv 0 \pmod{p^k}$. From \eqref{eqn:TMdelta}
\[
Y(\theta_0-\theta) \equiv \tau \cdot \delta_1^{b_1} \cdots \delta_r^{b_r}
\pmod{(p \OO_K)^k}.
\]
Note that the prime $\fp$ belongs to the support of the $\delta_i$.
However the other primes $\fp^\prime \mid p$, $\fp^\prime \ne \fp$
do not belong to the support of the $\delta_i$. We now eliminate $\fp$;
as in the statement of the proposition we take
$\fa:=(p\OO_K)/\fp$.
Then
\[
Y(\theta_0-\theta) \equiv \tau \cdot \delta_1^{b_1} \cdots \delta_r^{b_r}
\pmod{\fa^k}.
\]
Observe that $\fa$ is coprime to the support of the
$\delta_i$ and $Y$.
Recall $\tau \OO_K=\cT_1/\cT_2$ where $\cT_1$, $\cT_2$
are coprime integral ideals. By Lemma~\ref{lem:tauden},
the ideal $\cT_2$ is supported on $S$ and therefore
coprime to $\fa$. We therefore have a contradiction
if ${\gcd(\theta_0-\theta,\fa^k) \ne \gcd(\cT_1,\fa^k)}$.
This proves (i). Suppose $\gcd(\theta_0-\theta,\fa^k) = \gcd(\cT_1,\fa^k)$.
Then
\[
Y\cdot \tau_0 \equiv \delta_1^{b_1} \cdots \delta_r^{b_r}
\pmod{\fb},
\]
where $Y$, $\tau_0$, and $\delta_i$ all have support disjoint from $\fb$.
As in the proposition, $k^\prime$ is the smallest positive integer
such that $\fb \mid p^{k^\prime}$, and thus $(\Z/p^{k^\prime} \Z)^\times$
is a subgroup of $(\OO_K/\fb)^\times$ containing the image of $Y$.
Therefore $\overline{Y}=\overline{1}$ and $\phi(\bb)=\overline{\tau_0}$. If
$\overline{\tau_0} \notin \Img(\phi)$, then
we have a contradiction, and so our original assumption~\eqref{eqn:kbound}
is false. This proves (ii).

Suppose now that $\overline{\tau_0} \in \Img(\phi)$ and write
$\overline{\tau_0}=\phi(\ww)$ with $\ww \in \Z^r$.
Then
$\bb \in \ww+L$. Thus
$\lVert \bb\rVert_2 \ge D(L,\ww)$. If $D(L,\ww) > \cB_2$
then $\lVert \bb \rVert_2 > \cB_2$ and we contradict \eqref{eqn:cB1}.
This proves (iii).
\end{proof}

\medskip

\noindent \textbf{Remarks.}
\begin{itemize}
\item $\theta_0$ can be easily computed using Hensel's Lemma.
\item To apply the proposition in practice,
it is necessary to compute the abelian group structure
of $(\OO_K/\fb)^\times$ for ideals $\fb$ of very large norm (but
supported on the primes above $p$). For this we may apply
the algorithms in \cite[Section 4.2]{CohenAdvanced}.
\item $\cc(-\ww,L)$ (and therefore $D(\ww,L)$) can be computed using a closest
vector algorithm such as Fincke and Pohst \cite{FinckePohst}.
\item To effectively apply Proposition~\ref{prop:valbd}
in practice, we need to guess a value of $k$
such that $D(L,\ww)>\cB_2$.
We expect $D(L,\ww)$ to be around $I^{1/r}$
where $I$ is the index $[\Z^r : L]$.
Let us make two simplifying
assumptions: the first is that $\phi$ is surjective,
and the second is that $\gcd(\fa^k,\cT_1)=1$
so that
$\fb=\fa^k$ and $k^\prime=k$. Then
\[
I=\frac{\# (\OO_K/\fa^k)^\times}{\# (\Z/p^k\Z)^\times}
\approx \frac{\Norm(\fa)^k}{{p^k}}=p^{(d-2)k}
\]
where $d$ is the degree of $K$. Thus we should
expect a contradiction if $p^{(d-2)k/r}$ is much bigger than $\cB_2$,
or equivalently
\[k \gg \frac{r \log{\cB_2}}{(d-2) \log{p}}.\]
This heuristic gives a good guide for which values of $k$ to try.
\end{itemize}


\subsection{Example~\ref{ex:Ex4} continued}\label{page:fpbounds}
We continue giving details for Example~\ref{ex:Ex4},
and in particular for the tuple
$(\tau,\delta_1,\dotsc,\delta_{10})$ alluded to
on page~\pageref{page:taudeltatuple}.
In Section~\ref{sec:sunit}
we noted that $B \le 1.57 \times 10^{222}$. Thus we take
\begin{equation}\label{eqn:cB0bd}
\cB_{\infty}=1.57 \times 10^{222}, \qquad
\cB_2=\sqrt{10} \cdot \cB_{\infty} \approx 4.96 \times 10^{222}.
\end{equation}
We let $\fp=\fp_1=\langle 11, 3+\theta \rangle$
which is a prime above $11$.
The above heuristic suggests that we choose $k$ to
be larger than
\[
\frac{10 \log{\cB_2}}{(11-2) \cdot \log{11}} \approx 237.60\, .
\]
Our program tries $k=238$. It turns out (in the notation
of Proposition~\ref{prop:valbd}) that
$\gcd(\fa^k,\theta-\theta_0) = \gcd(\fa^k,\cT_1)=1$,
thus $\fb=\fa^k$, and moreover $k^\prime=k=238$.
The map $\phi$ is surjective, and thus $L$
does indeed have index
\[
I=\frac{\# (\OO_K/\fa^k)^\times}{\# (\Z/p^k\Z)^\times}
=2^7 \times 3^2 \times 5 \times 11^{2133} \times 61 \times 7321
\approx 5.02 \times 10^{2230}.
\]
We do not give $L$ as its basis vectors are naturally huge.
However, we find that
\[
D(L,\ww) \approx 1.14 \times 10^{223}.
\]
This is much larger than $\cB_2$ and we therefore know from Proposition~\ref{prop:valbd} that
$\ord_\fp(a_0 X-\theta Y) \le k-1=237$.

It is interesting to note that $I^{1/10} \approx 1.18 \times 10^{223}$
which is rather close to $D(L,\ww)$. If instead we take $k=237$,
we find that $I^{1/10} \approx 1.36 \times 10^{222}$ and
${D(L,\ww) \approx 9.55 \times 10^{221}}$ which is somewhat less
than $\cB_2$. We have generally found the above heuristic to
be remarkably accurate in predicting a good choice for $k$.

Now let $\fp_1,\dotsc,\fp_5$ be the primes of $S$
as in \eqref{eqn:primesS}, where $\fp_1=\fp$ as above.
Proposition~\ref{prop:valbd} gives upper bounds
$237$, $292$, $354$, $518$, $821$ for $\ord_{\fp_j}(a_0 X- \theta Y)$
with $j=1,\dotsc,5$ respectively.

\section{Linear Forms in Real Logarithms}\label{sec:LFRL}

In this section, we determine bounds on linear forms in logarithms which we will subsequently use in Section~\ref{sec:red} to successively reduce the large upper bound $\cB_{\infty}$ established in Section~\ref{sec:sunit}.

We let $s:=\#S$ and write
\[
S=\{\fp_1,\dotsc,\fp_s\}.
\]
Using Proposition~\ref{prop:valbd}, we suppose that we have obtained, for $1 \le j \le s$, integers $k_j$
such that
\begin{equation}\label{eqn:kjminus1}
\ord_{\fp_j}(a_0 X- \theta Y) \le k_j-1.
\end{equation}
Recall that
\begin{equation}\label{eqn:varepsilondef}
\delta_1^{b_1} \cdots \delta_r^{b_r} \; = \; \varepsilon
\; = \;
(a_0 X- \theta Y)/\tau \, .
\end{equation}
We write $k_j^\prime:=\ord_{\fp_j}(\tau)$, and
$k_j^{\prime\prime}:=k_j-1-k_j^\prime$. We obtain
\begin{equation}\label{eqn:valbds}
-k_j^\prime \; \le \; \ord_{\fp_j}(\varepsilon) \; \le \; k_j^{\prime\prime} \, .
\end{equation}

\subsection{Updating $\cB_1$ and $\cB_2$}
Recall that $\cB_\infty$, $\cB_1$, $\cB_2$
are respectively
the known bounds for $\lVert \bb \rVert_\infty$,
$\lVert \bb \rVert_1$, $\lVert \bb \rVert_2$
as in \eqref{eqn:cB1}. Initially we take
these as in \eqref{eqn:cB0tocB1}.
In practice, we are often able to update $\cB_1$ and $\cB_2$ with a smaller
bound after each iteration of Proposition~\ref{prop:valbd}.
Let $(u,v)$ be the signature of $K$.
Since $r$ is the rank of the $S$-unit group $\OO_S^\times$, we have
\[
	r=u+v-1+s.
\]
Recall our convention (page~\pageref{page:convention}) on the choice of
$S$-unit basis $\delta_1,\dotsc,\delta_r$: we suppose that
the basis is chosen so that $\delta_1,\dotsc,\delta_{u+v-1}$
is in fact a basis for the unit group modulo torsion.
Thus $\log \lVert \delta_i \rVert_{\nu}=0$
for all $\nu \in M_K^0$ and $1 \le i \le u+v-1$.
 Let $\cM_0$ denote the $s \times s$ matrix
\[\cM_0=(\, \log{\lVert \delta_j \rVert_\nu} \, )_{\nu \in S,~ u+v \le j \le r} \, .\]
In Lemma~\ref{lem:htlbB} let $U=\{\nu_1,\dotsc,\nu_r\}$
where $\nu_1,\dotsc,\nu_{u+v-1}$ are any $u+v-1$ elements of $M_K^\infty$
and the remainder are the elements of $S$. Then, in the notation of
Lemma~\ref{lem:htlbB},
\[
	\cM=
	\left(
	\begin{array}{c|c}
	* & *\\
	\hline
	0 & \cM_0
	\end{array}
	\right).
\]
Since $\cM$ is invertible by Lemma~\ref{lem:htlbB},
it follows that $\cM_0$ is invertible.
We partition our exponent vector $\bb$ as
\[
	\bb=\big[\bb^\prime \vert \bb^{\prime\prime} \big], \qquad
\bb^\prime=\big[b_i\big]_{i=1,\dotsc,u+v-1},
\qquad
\bb^{\prime\prime}=\big[b_i\big]_{i=u+v,\dotsc,r}.
\]
Write
$\uu^{\prime\prime}:=\big[\log{\lVert \varepsilon \rVert_\nu} \big]_{\nu \in S}$ in $\R^s$.
By the above, we have $\uu^{\prime\prime} = \cM_0 \bb^{\prime\prime}$
and thus $\bb^{\prime\prime} = \cM_0^{-1}\uu^{\prime\prime}$. That is, for $1 \leq i \leq s$,
\[b_{u+v-1+i} = m_{i1}\log\lVert \varepsilon \rVert_{\fp_1} + \cdots + m_{is}\log\lVert \varepsilon \rVert_{\fp_s},\]
where $\cM_0^{-1} = \big[m_{ij}\big]$. It follows that
\begin{equation}\label{eqn:finbound}
|b_{u+v-1+i}| \; \leq\;  |m_{i1}|\cdot \lvert\log\lVert \varepsilon \rVert_{\fp_1}\rvert + \cdots + |m_{is}| \cdot \lvert\log\lVert \varepsilon \rVert_{\fp_s}\rvert.
\end{equation}
Applying Proposition~\ref{prop:valbd} to any $\fp_j$ for $1 \leq j \leq s$ and using \eqref{eqn:valbds} and \eqref{eqn:normval}, we obtain
\[
	|\log\lVert \varepsilon \rVert_{\fp_j}| \leq
	\log(\Norm(\fp_j))\cdot\max\{|k_j^{\prime}|,|k_j^{\prime\prime}|\}.
\]
Write
\begin{equation}\label{eqn:rho}
	\rho_i^\prime :=\sum_{j=1}^s
\lvert m_{i,j} \rvert \cdot \log(\Norm(\fp_j))\cdot\max\{|k_j^{\prime}|,|k_j^{\prime\prime}|\}, \qquad \rho_i=\min\{\cB_\infty,\rho_i^\prime\}.
\end{equation}
From equation \eqref{eqn:finbound} it follows that
$\lvert b_{u+v-1+i} \rvert \le \rho_i^\prime$ for $1 \le i \le s$.
However,
$\max\{\lvert b_i \rvert \}_{i=1}^r=\lVert \bb \rVert_\infty \le \cB_\infty$,
so we know that $\lvert b_{u+v-1+i} \rvert \le \cB_\infty$.
We deduce that
\begin{equation}\label{eqn:rhobound}
	\lvert b_{u+v-1+i} \rvert \le \rho_i, \qquad 1 \le i \le s.
\end{equation}
Hence
\[
	\lVert \bb \rVert_1 \; =\; \lVert \bb^\prime\rVert_1+
	\lVert \bb^{\prime\prime}\rVert_1
	 \; \le \;
	(u+v-1) \cB_\infty+\rho_1+\cdots+ \rho_s
\]
and
\[
	\lVert \bb \rVert_2 \; =\;
	\sqrt{\lVert \bb^\prime\rVert_2^2+\lVert \bb^{\prime\prime}
	\rVert_2^2} \;  \le \;
	\sqrt{(u+v-1) \cB_\infty^2+\rho_1^2+\cdots+ \rho_s^2}.
\]
We now update our values for $\cB_1$ and $\cB_2$:
\begin{equation}\label{eqn:cB1new}
	\cB_1=(u+v-1) \cB_\infty+\rho_1+\cdots+ \rho_s,
\end{equation}
\begin{equation}\label{eqn:cB2new}
	\cB_2=\sqrt{(u+v-1) \cB_\infty^2+\rho_1^2+\cdots+ \rho_s^2}.
\end{equation}
Note, since by \eqref{eqn:rho} we have $\rho_i \le \cB_\infty$,
these new values for $\cB_1$ and $\cB_2$ are bounded above
by the old values given in \eqref{eqn:cB0tocB1}.
In practice we usually find that these give
significantly better bounds
for $\lVert \bb \rVert_1$, $\lVert \bb \rVert_2$.

\subsection{Embeddings and Improving the initial bound \eqref{eqn:cB0}}
To improve our initial bound \eqref{eqn:cB0}, we rely on the inequality ${B \le 2 c_{17} \cdot d \cdot h(\varepsilon)}$
furnished by Lemma~\ref{lem:htlbB}. However
$h(\varepsilon)=h(\varepsilon^{-1})$ and so
\[
B \; \le \; 2 c_{17} \sum_{\nu \in M_K} \log{\max\{1,\, \lVert \varepsilon^{-1} \rVert_\nu\}} \, .
\]
Since $\varepsilon$ is an $S$-unit,
for $\nu \notin M_K^\infty \cup S$, we have $\lVert \varepsilon \rVert_\nu=1$.
Thus
\begin{equation}\label{eqn:insum2}
B \; \le \; 2 c_{17} \sum_{\nu \in M_K^\infty \cup S} \log{\max\{1,\, \lVert \varepsilon^{-1} \rVert_\nu\}} \, .
\end{equation}
Therefore, to obtain a better bound for $B$, it is enough to
gain good control on
the contributions to the sum on the right-hand side of \eqref{eqn:insum2}.

\begin{lem}\label{lem:c21}
Let
\[
c_{21} \; = \; \sum_{i=1}^s \max\{0,\, k_i^{\prime\prime}\} \cdot \log{(\Norm(\fp_i))} \, .
\]
Then
\begin{equation}\label{eqn:insum3}
B \; \le \; 2 c_{17}
\left( c_{21} \, +\, \sum_{\nu \in M_K^\infty} \log{\max\{1,\, \lVert \varepsilon^{-1} \rVert_\nu\}} \right) \, .
\end{equation}
\end{lem}
\begin{proof}
From \eqref{eqn:valbds} and \eqref{eqn:normval} we have
\[
\sum_{\nu \in S}
\log\max\{1,\lVert \varepsilon^{-1} \rVert_\nu \}
\; \le \;
c_{21} \, .
\]
The lemma now follows from \eqref{eqn:insum2}.
\end{proof}


We shall write
\[
M_K^\infty \; = \; M_K^{\R} \, \cup \, M_K^{\C} \, ,
\]
where $M_K^{\R}$ and $M_K^{\C}$ are respectively the sets
of real and complex places.
Recall that $(u,v)$ denotes the signature of $K$. Thus
we have embeddings
\[
\sigma_1,\dotsc,\sigma_{u}, \qquad
\sigma_{u+1},\dotsc,\sigma_{u+v},\overline{\sigma_{u+1}},
\dotsc, \overline{\sigma_{u+v}}
\]
of $K$, where $\sigma_i$ are real embeddings for $1 \le i \le u$,
and $\sigma_{u+i}$, $\overline{\sigma_{u+i}}$ are pairs of complex
conjugate embeddings. Let
\begin{equation}\label{eqn:cEdef}
\cE_K^{\R} \; := \; \{\sigma_1,\dotsc,\sigma_{u}\},
\qquad
\cE_K^{\C} \; := \; \{\sigma_{u+1},\dotsc,\sigma_{u+v}\}.
\end{equation}
For the membership of $\cE_K^{\C}$, we are making an arbitrary
choice of a member from each pair of conjugate complex embeddings,
but that is unimportant.
Note that $M_K^\R$ is in one-to-one correspondence
with $\cE_K^{\R}$ and $M_K^{\C}$ is in one-to-one correspondence
with $\cE_K^{\C}$.
We consider the contribution
to the sum \eqref{eqn:insum3} coming from $\nu \in M_K^{\C}$,
or equivalently from $\sigma \in \cE_K^{\C}$.

Let $\Im(z)$ denote the imaginary part of a complex number $z$.
\begin{lem}\label{lem:cmxbd}
Let
\[
c_{22} \; = \; 2 \sum_{\sigma \in \cE_K^{\C}}
\log \max \left\{1,\, \frac{\lvert \sigma(\tau)\rvert}{
\lvert \Im(\sigma(\theta))\rvert } \right\} \, .
\]
Then
\begin{equation}\label{eqn:insum}
B \; \le \; 2 c_{17}
\left( c_{21} \, +\, c_{22} \, +\,
 \sum_{\sigma \in \cE_K^{\R}} \log{\max\{1,\, \lvert \sigma(\varepsilon) \rvert^{-1}\}} \right) \, .
\end{equation}
\end{lem}
\begin{proof}
Note that \eqref{eqn:insum3} can be rewritten as
\[
B \; \le \; 2 c_{17}
\left( c_{21} \, +\,
 2 \sum_{\sigma \in \cE_K^{\C}} \log{\max\{1,\, \lvert \sigma(\varepsilon) \rvert^{-1}\}}
\,
+\,
 \sum_{\sigma \in \cE_K^{\R}} \log{\max\{1,\, \lvert \sigma(\varepsilon) \rvert^{-1}\}} \right) \, .
\]
Let $\sigma \in \cE_K^{\C}$. Then as ${a_0 X- \theta Y=\tau \cdot \varepsilon}$,
we have
\[
\lvert \sigma(\varepsilon) \rvert \; = \;
\frac{1}{\lvert \sigma(\tau) \rvert} \cdot
\lvert \sigma(a_0X- \theta Y) \rvert \; \ge
\;
\frac{1}{\lvert \sigma(\tau) \rvert} \cdot
\lvert Y \rvert \cdot \lvert \Im(\sigma(\theta)) \rvert
\;
\ge
\;
\frac{
\lvert \Im(\sigma(\theta)) \rvert
}{\lvert \sigma(\tau) \rvert}
\, ,
\]
because of our assumption $\lvert Y \rvert \ne 0$. The lemma
follows.
\end{proof}

The following is immediate.
\begin{prop}\label{prop:totallycomplexbound}
If $K$ is totally imaginary then
\[B \; \le \; {2 c_{17}}({c_{21}+c_{22}})\, .\]
\end{prop}


\subsection{The non-totally complex case}
Suppose now that $K$ has one or more
real embeddings. Recall that the signature
of $K$ is $(u,v)$. Thus $u \ge 1$.
\begin{lem}\label{lem:atmostone}
If $u=1$ we let $c_{23}:=1$. If $u \, \ge\, 2$ we let
\[
c_{23}:=\min\left\{
\frac{\lvert \sigma(\theta)-\sigma^\prime(\theta) \rvert}{\lvert \sigma(\tau) \rvert + \lvert \sigma^\prime(\tau) \rvert}
\; : \; \sigma,~\sigma^\prime \in \cE_K^{\R},\; \sigma \ne \sigma^\prime
\right\} \, .
\]
Then there is at most one $\sigma \in \cE_K^\R$ such that
$\lvert \sigma(\varepsilon) \rvert < c_{23}$.
\end{lem}
\begin{proof}
Suppose otherwise. Then there are
$\sigma$, $\sigma^\prime \in \cE_K^\R$ with $\sigma \ne \sigma^\prime$
such that
$\lvert \sigma(\varepsilon) \rvert < c_{23}$ and $\lvert \sigma^\prime(\varepsilon) \rvert< c_{23}$.
As $a_0 X- \theta Y=\tau \cdot \varepsilon$ we find that
\[
\lvert a_0 X- \sigma(\theta) Y \rvert\; < c_{23} \cdot
\lvert \sigma(\tau) \rvert,\qquad
\lvert a_0 X- \sigma^\prime(\theta) Y \rvert\; < c_{23} \cdot
\lvert \sigma^\prime(\tau) \rvert\, .
\]
Thus
\[
\lvert \sigma(\theta)- \sigma^\prime(\theta) \rvert \cdot \lvert Y \rvert <
c_{23} \cdot ( \lvert \sigma(\tau) \rvert+
\lvert \sigma^\prime(\tau) \rvert).
\]
Recall our assumption that $Y \ne 0$. This inequality now
contradicts our definition of $c_{23}$.
\end{proof}

\begin{lem}\label{lem:eta}
Let
\[
c_{24}\; := \; c_{21}\; + \; c_{22} \; + \; (u \, - \, 1) \log\max\{1, \, c_{23}^{-1}\},
\]
and
\[
c_{25}\; := \; \exp(c_{24}) \, ,
\qquad
c_{26} \; := \frac{1}{2 c_{17}} \, .
\]
Suppose $B > 2 c_{17} \cdot c_{24}$.
Let $\sigma \in \cE_K^\R$ be chosen so that
$\lvert \sigma(\varepsilon) \rvert$ is minimal. Then
\begin{equation}\label{eqn:epseta}
\lvert \sigma(\varepsilon) \rvert \; \le \;
c_{25} \cdot \exp(-c_{26} \cdot B) \, .
\end{equation}
\end{lem}
\begin{proof}
From Lemma~\ref{lem:atmostone}, we have
\[
\lvert \sigma^\prime(\varepsilon) \rvert \; \ge c_{23}
\]
for all $\sigma^\prime \in \cE_K^\R$ with $\sigma^\prime \ne \sigma$.
From \eqref{eqn:insum} we deduce that
\[
\begin{split}
B \; & \le \;
2 c_{17}
\left( c_{21} \, +\, c_{22} \, +\,
(u \, -\, 1) \log\max\{1, \, c_{23}^{-1}\}
\, + \,
 \log{\max\{1,\, \lvert \sigma(\varepsilon) \rvert^{-1}\}} \right) \\
& = \;
2 c_{17}
\left( c_{24} \, +\,
 \log{\max\{1,\, \lvert \sigma(\varepsilon) \rvert^{-1}\}} \right) \, .
\end{split}
\]
It follows that
\[
\log \max\{ 1,\, \lvert \sigma(\varepsilon)\rvert^{-1}\}
\; \ge \;
\frac{1}{2 c_{17}} B \, - \, c_{24}
\; = \;
c_{26} \cdot B \, - \, c_{24} \, .
\]
The hypothesis $B > 2 c_{17} \cdot c_{24}$ forces the right-hand side
to be positive, and so the left-hand side must simply be
$\log {\lvert \sigma(\varepsilon)^{-1}\rvert}$. After
exponentiating and rearranging, we obtain \eqref{eqn:epseta}.
\end{proof}


\subsection{Approximate relations}
As in Lemma~\ref{lem:eta} we shall let $\sigma \in \cE_K^{\R}$
be the real embedding that makes
$\lvert \sigma(\varepsilon) \rvert$  minimal. Recall that the signature of $K$ is $(u,v)$; we keep the assumption that $u\geq 1$. Let
\begin{equation}\label{eqn:nrv}
n \; := r+v\, .
\end{equation}
In this section we introduce additional unknown integers
$b_{r+1},\dotsc ,{b_{n}}$, closely related to
the exponents $b_1,\dotsc,b_r$ found in \eqref{eqn:TMdelta}.
We shall use Lemma~\ref{lem:eta} to write down $d-2$
 linear forms in $b_1,\dotsc,b_n$ with real coefficients,
whose values are very small. We shall later give a method,
based on standard ideas originally due to de Weger,
that uses these
\lq approximate relations\rq\
to reduce our bound for $B = \max(\lvert b_1 \rvert,\dotsc,\lvert b_r \rvert,1)$.


We label the elements of $\cE_K^{\R}$ and $\cE_K^{\C}$ as in
\eqref{eqn:cEdef}, where $\sigma_1=\sigma$.
Write
\[
\theta_j=\sigma_j(\theta),
\quad
\tau_j=\sigma_j(\tau),
\quad
\varepsilon_j=\sigma_j(\varepsilon),
\quad
\delta_{i,j}=\sigma_j(\delta_i),
\qquad 1 \le j \le u+v, \; 1 \le i \le r.
\]
Let
$2 \le j \le u+v$ and write
\[
z_j \; := \; \frac{a_0 X- \theta_1 Y}{a_0 X - \theta_j Y} \, .
\]
Observe that
\begin{equation}\label{eqn:prelinform}
\begin{split}
 Y (\theta_1-
\theta_j)  \; &= \;
  (a_0 X - \theta_j Y) \; - \;
(a_0 X- \theta_1 Y)    \\
&= \;  (a_0 X -\theta_j Y) \cdot
 (1-z_j) \\
& = \;
\tau_j \cdot \delta_{1,j}^{b_1} \cdots \delta_{r,j}^{b_r} \cdot (1-z_j).
\end{split}
\end{equation}
In the following lemma, as always, $\log$ denotes the principal
determination of the logarithm (i.e.\ the imaginary part of
$\log$ lies in $(-\pi,\pi]$).
\begin{lem}\label{lem:log1minusz}
Let
\[
c_{27} \; := \; \frac{\lvert \tau_1 \rvert \cdot c_{25}}{\min\{ \, {\lvert \tau_i \rvert} \; : \; \sigma_i \in \cE_K^{\R} \, , \sigma_i \neq \sigma \} \cdot c_{23}}\, ,\qquad
c_{28} \; := \: \frac{\lvert \tau_1 \rvert \cdot c_{25}}{\min\{ \, \lvert \Im(\theta_i) \rvert \; : \; \sigma_i \in \cE_K^{\C} \, \}},\]
and
\[c_{29}(j) \; :=
\begin{dcases}
\frac{\lvert \tau_1 \rvert \cdot c_{25}}{\lvert \tau_j \rvert \cdot c_{23}} & \text{ for }\ 2 \leq j \leq u \\ 
\frac{\lvert \tau_1 \rvert \cdot c_{25}}{\lvert \Im(\theta_j)\rvert} & \text{ for }\ u+1 \leq j \leq u+v. 
\end{dcases}\]
Define
\[c_{30} \; := \;
\max\{ 2 c_{17} \cdot c_{24} \, , \; \log{(2 c_{27})}/c_{26} , \; \log{(2 c_{28})}/c_{26}\}
\]
and suppose
$B \, > \, c_{30}$.
Then
\[
\lvert \, \log (1- z_j) \rvert \;
\le \;
2 c_{29}(j)
\cdot \exp(-c_{26} \cdot B)
 \,  \qquad \text{for } 2 \le j \le u+v.\]
\end{lem}
\begin{proof}
Let $2 \le j \le u+v$.
If $\sigma_j \in \cE_K^{\R}$, Lemma~\ref{lem:atmostone} yields
\[
\lvert a_0 X -\theta_j Y \rvert
\; = \; \lvert \tau_j \rvert \cdot \lvert \varepsilon_j \rvert
\; \ge \;\lvert \tau_j \rvert \cdot c_{23}\, .
\]
Conversely, if $\sigma_j \in \cE_K^{\C}$, following the proof of Lemma~\ref{lem:cmxbd}, we have
\[
\lvert a_0 X -\theta_j Y \rvert
\; = \; \lvert \tau_j \rvert \cdot \lvert \varepsilon_j \rvert
\; \ge \;\lvert \Im(\theta_j)\rvert \, .
\]
Now, by Lemma~\ref{lem:eta} we have
\[
\lvert a_0 X -\theta_1 Y \rvert \; =\;
\lvert \tau_1 \rvert \cdot \lvert \varepsilon_1 \rvert \; \le \;
\lvert \tau_1 \rvert \cdot c_{25} \cdot \exp(-c_{26} \cdot B)\, ;
\]
it is in invoking this lemma that we have
made use of the assumption $B> 2 c_{17} \cdot c_{24}$.
Thus
\[\lvert z_j \rvert \; \le \; c_{29}(j) \cdot \exp(-c_{26} \cdot B) \, .\]
Our assumption $B > c_{30} \ge \log{(2 c_{29}(j))}/c_{26}$
gives $\lvert z_j \rvert < 1/2$.
From the standard Maclaurin expansion for $\log(1-x)$ we conclude
that
$\lvert  \log (1- z_j) \rvert \, \le \, 2  \cdot \lvert z_j \rvert$,
completing the proof.
\end{proof}

To ease notation, let
\begin{equation}\label{eqn:w}
w\; :=\; u+v-2.
\end{equation}
We now give our first set of $w$ approximate relations.
These only involve our original unknown exponents
$b_1,\dotsc,b_r$ found in \eqref{eqn:TMdelta}.
\begin{lem}\label{lem:linformreal}
Suppose $B>c_{30}$ holds.
Let $1 \le j \le w$.
Let
\[
\beta_j :=
\log \left\lvert
\frac{(\theta_1-\theta_2)\cdot \tau_{j+2}}{(\theta_1-\theta_{j+2})\cdot \tau_2}
\right\rvert \, , \qquad
\alpha_{1,j} := \log \left\lvert \frac{\delta_{1,{j+2}}}{\delta_{1,2}}\right\rvert
 \, ,
\ldots, \,
\alpha_{r,j} := \log \left\lvert \frac{\delta_{r,{j+2}}}{\delta_{r,2}}\right\rvert
\, .
\]
Then
\begin{equation}\label{eqn:linformrl}
\lvert \beta_j+b_1 \alpha_{1,j}+\cdots+b_{r} \alpha_{r,j} \lvert
\; \le \;  2 (c_{29}(2) + c_{29}(j+2)) \cdot \exp(-c_{26} \cdot B).
\end{equation}
\end{lem}
\begin{proof}
From \eqref{eqn:prelinform},
\[
\frac{(\theta_1-\theta_2)\cdot \tau_{j+2}}{(\theta_1-\theta_{j+2})\cdot \tau_{2}}
\cdot
\left(\frac{\delta_{1,j+2}}{\delta_{1,2}}\right)^{b_1}
\cdots
\left(\frac{\delta_{r,j+2}}{\delta_{r,2}}\right)^{b_r}
 \; =\;
\frac{1-z_2}{1-z_{j+2}} \, .
\]
Taking absolute values and then logs
gives
\[
\lvert \beta_j+b_1 \alpha_{1,j}+\cdots+b_{r} \alpha_{r,j} \lvert
\; \le \;  \lvert \log \lvert 1-z_2 \rvert \rvert+
\lvert \log\lvert 1-z_{j+2} \rvert \rvert.
\]
For a complex number $z$, we have
\[
\lvert \log \lvert z \rvert \rvert \; \le \; \lvert \log{z} \rvert
\]
since $\log \lvert z \rvert$ is the real part of $\log{z}$.
 The lemma now follows
from Lemma~\ref{lem:log1minusz}.
\end{proof}
In essence, in the above lemma, we have made use
of the fact that
\begin{equation}\label{eqn:logembed}
K^{\times} \rightarrow \R, \qquad \phi \mapsto \log \lvert \sigma(\phi) \rvert
\end{equation}
is a homomorphism for each embedding $\sigma$ of $K$,
and applied this to the approximate multiplicative
relation \eqref{eqn:prelinform} to obtain an
approximate (additive) relation \eqref{eqn:linformrl}.
If $\sigma$ is complex, then $\sigma$ and its conjugate $\overline{\sigma}$
induce the same homomorphism \eqref{eqn:logembed},
and thus we need only consider
the embeddings $\sigma_1,\dotsc,\sigma_{u+v}$.
Note that although these are $u+v$ embeddings,
we have obtained
only $u+v-2$ approximate relations so far.
That is,
we have had to sacrifice embeddings
because
we wanted to eliminate the two unknowns, $X$ and $Y$.
For $\sigma$ real, $\log \lvert \sigma(\phi) \rvert$
determines $\sigma(\phi)$ up to signs.
However, if $\sigma$ is complex, then \eqref{eqn:logembed}
loses the argument of $\sigma(\phi)$. Thus we should
consider another homomorphism
\begin{equation}\label{eqn:cxlogembed}
K^{\times} \rightarrow \R/\Z \pi, \qquad \phi \mapsto \Im(\log  \sigma(\phi))
\end{equation}
where $\Im(z)$ denotes the imaginary part of a complex number $z$.
Observe that $\Im(\log \sigma(\phi))$ denotes the argument of $\sigma(\phi)$
which naturally lives in $\R/\Z 2\pi$, whilst here we use $\R/\Z\pi$
as the codomain. In
practice, we have found that using $\R/\Z 2\pi$ introduces extra factors but
only results in negligible improvements to the bounds. 
Applying these homomorphisms to
\eqref{eqn:prelinform} allows us to obtain additional approximate relations.
Since there are $v$ complex embeddings,
we obtain an additional $v$ approximate relations. However
since these homomorphism are into $\R/\Z\pi$, the approximate
relations are only valid after shifting by an appropriate multiple
of $\pi$; thus for each complex embedding $\sigma_{u+j}$, we will need
an additional parameter $b_{r+j}$.

Recall that $w = u + v -2$.
\begin{lem}\label{lem:linformimag}
Let $1 \le j \le v$.
Let
\[
\beta_{w+j} :=
\Im\left(\log \left(\frac{\theta_1-\theta_{u+j}}{\tau_{u+j}}\right)\right),
\]
\[
\alpha_{1,w+j} := -\Im(\log \delta_{1,u+j}) \, ,
\ldots, \,
\alpha_{r,w+j} := -\Im(\log \delta_{r,u+j})\, ,
\quad
\alpha_{r+j,w+j} := \pi\, .
\]
Suppose $B \, > \, c_{30}$ holds.
Then there is some $b_{r+j} \in \Z$ such that
\begin{multline}\label{eqn:linformcx}
\lvert \beta_{w+j}+b_1 \alpha_{1,w+j}+\cdots+b_{r} \alpha_{r,w+j}+
b_{r+j} \alpha_{r+j,w+j}\lvert
\\
\; \le \;  2 c_{29}(u+j)
\cdot \exp(-c_{26} \cdot B).
\end{multline}
Moreover,
\begin{equation}\label{eqn:bdp1}
\lvert b_{r+j} \rvert \; \le 
\; \lvert b_1 \rvert + \cdots+ \lvert b_r\rvert+\frac{\pi+1}{\pi}\, .
\end{equation}
\end{lem}
\begin{proof}
From \eqref{eqn:prelinform},
and Lemma~\ref{lem:log1minusz},
\begin{multline*}
\lvert \log(Y)+
\log((\theta_1-\theta_{u+j})/\tau_{u+j})
-b_1 \log{\delta_{1,u+j}}-\cdots-b_r \log{\delta_{r,u+j}}+b^\prime \cdot \pi i \rvert
\\
\; \le \;
2 c_{29}(u+j)
\cdot \exp(-c_{26} \cdot B),
\end{multline*}
for some $b^\prime \in \Z$.
Thus
\begin{multline*}
\lvert \Im(\log(Y))+ \beta_{w+j}+
b_1 \alpha_{1,w+j}+\cdots+b_r \alpha_{r,w+j}+b^\prime \pi
\rvert
\\
\; \le \;
2 c_{29}(u+j)
\cdot \exp(-c_{26} \cdot B).
\end{multline*}
Recall that $Y \in \Z \setminus \{0\}$,  so
$\Im(\log(Y))$ is either $0$ or $\pi$
depending on whether $Y$ is positive or negative.
We take $b_{r+j}=b^\prime$ in the former case
and $b_{r+j}=b^{\prime}+1$ in the latter case.
This gives \eqref{eqn:linformcx}.

It remains to prove \eqref{eqn:bdp1}.
Our assumption $B>c_{30}$
gives
\[2 c_{29}(u+j) \cdot \exp(-c_{26} \cdot B)<1.\]
Moreover,
$\lvert \beta_{w+j} \rvert \le \pi$ and
$\lvert \alpha_{i,w+j} \rvert \le \pi$ for $0 \le i \le r$.
From \eqref{eqn:linformcx},
\[
\begin{split}
\pi \cdot \lvert b_{r+j} \rvert
\; & = \; \lvert \alpha_{r+j,w+j}\rvert \cdot \lvert b_{r+j} \rvert \\
& \le \; \lvert \beta_{w+j} \rvert+\lvert \alpha_{1,w+j} \rvert \cdot \lvert b_1 \rvert
+ \cdots+ \lvert \alpha_{r,w+j} \rvert \cdot \lvert b_r\rvert+1\\
& \le \; \pi(1 + \lvert b_1 \rvert
+ \cdots+ \lvert b_r\rvert)+1.\\
\end{split}
\]

The lemma follows.
\end{proof}

Summing up, Lemma~\ref{lem:linformreal} and Lemma~\ref{lem:linformimag}
give us $(u+v-2)+v=d-2$ approximate relations \eqref{eqn:linformrl},
\eqref{eqn:linformcx} in integer unknowns $b_1,\dotsc,b_{r+v}$.


\section{Reduction of Bounds} \label{sec:red}
We do not know which real embedding $\sigma \in \cE_K^\R$ makes
$\lvert \sigma(\varepsilon) \rvert$ minimal. So the procedure
described below for reducing the bound \eqref{eqn:cB0}
needs
to be repeated for each possible choice of embedding $\sigma$ in $\cE_K^\R$.
Thus, for every possible choice of $\sigma \in \cE_K^\R$, we let
$\sigma_1=\sigma$ and we choose an ordering of the other
embeddings as in \eqref{eqn:cEdef}.
Given a real number $\gamma$, we denote by $[\gamma]$
the nearest integer to $\gamma$, with the convention
that $[k+1/2]=k+1$ for $k \in \Z$.
Let $n$ be as in \eqref{eqn:nrv} and observe that
\[
n=(s+1)+d-2,
\]
where we recall that $s=\#S$.
Let $C$ be a positive integer to be chosen later. Let $\mathbf{I}_{m}$ and $\mathbf{0}_{i,j}$ be the $m \times m$ identity matrix and $i \times j$ zero matrix, respectively. Let $M$ be
the following $n \times n$ matrix
\[
  M:=
  \left[\begin{array}{c:@{}c@{}:@{}c@{}}
          \mathbf{0}_{w,s+1}
          & \begin{array}{ccc}
              [C\alpha_{1,1}] & \cdots & [C\alpha_{1,w}]\\ \relax
              \vdots & & \vdots \\ \relax
              [C\alpha_{w,1}] & \cdots & [C\alpha_{w,w}]
            \end{array}
          & \begin{array}{ccc}
              [C\alpha_{1,w+1}] & \cdots & [C\alpha_{1,d-2}] \\ \relax
              \vdots & &\vdots \\ \relax
              [C\alpha_{w,w+1}] & \cdots & [C\alpha_{w,d-2}]
            \end{array} \Bstrut \\ \hdashline
          \mathbf{I}_{s+1}
          & \begin{array}{ccc}
              [C\alpha_{w+1,1}] & \cdots & [C\alpha_{w+1,w}]\\ \relax
              \vdots & & \vdots \\ \relax
              [C\alpha_{r,1}] & \cdots & [C\alpha_{r,w}]
            \end{array}
          & \begin{array}{ccc}
              [C\alpha_{w+1,w+1}] & \cdots & [C\alpha_{w+1,d-2}] \\ \relax
              \vdots & &\vdots \\ \relax
              [C\alpha_{r,w+1}] & \cdots & [C\alpha_{r,d-2}]
            \end{array} \Bstrut \Tstrut \\ \hdashline
          \mathbf{0}_{v,s+1}
          & \mathbf{0}_{v,w}
          & [C\pi] \cdot \mathbf{I}_{v} \Tstrut
        \end{array}\right] \]
and let $L$ be the sublattice of $\Z^n$ spanned by the rows of $M$.
Recall that $\cB_\infty$, $\cB_1$, $\cB_2$ are respectively the known bounds for $\lVert \bb \rVert_\infty$, $\lVert \bb \rVert_1$, $\lVert \bb \rVert_2$. Let
\begin{gather*}
\ww:= (\underbrace{0,0,\dotsc,0}_{s+1},[C \beta_1],\dotsc,[C \beta_{d-2}])
\in \Z^n, \\
\cA_1 \; :=\; \frac{1+\cB_1}{2}, \qquad \cA_2 \; :=\; \frac{2\pi(1+\cB_1) + 1}{2\pi},\\
\cB_3 \; :=\; \sum_{j=1}^w(c_{29}(2) + c_{29}(j+2))^2 + \sum_{j=1}^vc_{29}(u+j)^2,\\
\cB_4 \; :=\; \cA_1\sum_{j=1}^w(c_{29}(2) + c_{29}(j+2)) + \cA_2\sum_{j=1}^vc_{29}(u+j) , \\
\text{ and } \quad \cB_5 \; :=\; \sqrt{\cB_2^2 - w\cB_{\infty}^2+ w\cA_1^2 + v\cA_2^2} \, .
\end{gather*}
By \eqref{eqn:cB2new}, we observe that
\begin{equation*}
  \cB_2^2=(w+1) \cB_\infty^2+\rho_1^2+\cdots+ \rho_s^2
\end{equation*}
so that $\cB_2^2 - w\cB_{\infty}^2 \geq 0$ and 
thus the argument of the square root in the above definition
of $\cB_5$ is positive.

Write
\[
  \bb_e \, := \,
  (b_1,b_2,\dotsc,b_r,b_{r+1},\dotsc,b_{r+v})
\]
where $b_{r+1},\dotsc,b_{r+v}$ are as in Lemma~\ref{lem:linformimag}.
We think of this as the \lq extended exponent vector\rq.
Note that the number of entries in $\bb_e$ is
\[
  r+v=u+v+s-1+v=d+s-1=n.
\]
If $\bb_e$ is known, then the solution is known.

\begin{lem}\label{lem:never}
\begin{equation}\label{eqn:never}
\lVert \bb_e \rVert_2 \; \le \; \sqrt{
  \cB_2^2 + v \left( \cB_1 + \frac{\pi+1}{\pi} \right)^2
}.
\end{equation}
\end{lem}
\begin{proof}
This follows immediately from \eqref{eqn:bdp1} and the definitions of $\cB_1,\cB_2$.
\end{proof}

\begin{prop}\label{prop:boundImprove}
  Suppose $\bb_e \cdot M \ne -\ww$.  Let
  \begin{equation}\label{eqn:Ddef}
    \cD \; := \;
    \begin{dcases}
      D(L,\ww) & \text{ if } \ww \notin L \\
      \min_{\substack{\xx\in L \\ \xx \ne \mathbf{0}}} \lVert \xx \rVert_2 & \text{ if } \ww \in L.
    \end{dcases}
  \end{equation}
  Suppose $\cD > \cB_5$. Then
\begin{equation}\label{eqn:Bfinal}
B \; \le \;
\max \left(
\; c_{30} \quad , \quad
\frac{1}{c_{26}}
\cdot
\log{\left(
\frac{2 C \cdot \cB_3}{\sqrt{\cB_3(\cD^2-\cB_5^2)+\cB_4^2}\, - \, \cB_4}
\right)} \right) \, .
\end{equation}
\end{prop}
\begin{proof}
If $B \le c_{30}$ then \eqref{eqn:Bfinal} holds. We will therefore
suppose that $B> c_{30}$. Thus,
inequalities \eqref{eqn:linformrl} and \eqref{eqn:linformcx} hold.
Write
\[
\ww+\bb_e \cdot M  \; = \;
(b_{u+v-1},b_{u+v},\dotsc,b_{r}, \Theta_1, \Theta_2,\dotsc,\Theta_{d-2})\, ,
\]
where we take this equality as the definition
	of $\Theta_1,\dotsc,\Theta_{d-2}$.
That is, for $1 \le j \le w$,
\[
	\Theta_j \; = \;
	[C \beta_j]+ b_1 [C \alpha_{1,j}]+\cdots+b_r [ C \alpha_{r,j}].
\]
Hence, again for $1 \le j \le w$,
\[
\begin{split}
\lvert \Theta_j \rvert
\; & \le \; \frac{1}{2}+\frac{\lvert b_1 \rvert}{2}+\cdots
+ \frac{\lvert b_r \rvert}{2} \, + \,
C \cdot \lvert \beta_j+b_1\alpha_{1,j}+\cdots + b_r \alpha_{r,j} \rvert\\
\; & \le \frac{1 + \cB_1}{2} \, +
2 C\cdot(c_{29}(2) + c_{29}(j+2)) \cdot \exp(-c_{26} \cdot B) \\
\; &\le \; \cA_1 \, + (c_{29}(2) + c_{29}(j+2))\cdot\eta, \\
\end{split}
\]
where the second inequality follows by \eqref{eqn:cB0} and
\eqref{eqn:linformrl}, and
\[
	\eta := 2 C \cdot \exp(-c_{26} \cdot B).
\]

Recall that $w=u+v-2$. For $1 \le j \le v$,
\[
	\Theta_{w+j} \; = \;
	[C \beta_{w+j}]+ b_1 [C \alpha_{1,w+j}]+\cdots+b_r [ C \alpha_{r,w+j}]+
	b_{r+j} [C \pi] .
\]
Thus, for $1 \le j \le v$,
\begin{align*}
\lvert \Theta_{w+j} \rvert
\; & \le \; \frac{1}{2}+\frac{\lvert b_1 \rvert}{2}+\cdots
+ \frac{\lvert b_r \rvert}{2}+\frac{\lvert b_{r+j}\rvert}{2} \\
& \qquad \qquad+ C \cdot \lvert \beta_{w+j}+b_1\alpha_{1,w+j}+\cdots + b_r \alpha_{r,w+j}+b_{r+j}\pi \rvert\\
\; & \le \; \frac{2\pi + 1}{2\pi}+ \cB_1+ \,
2C \cdot c_{29}(u+j)\cdot \exp(-c_{26} \cdot B)\\
&\le \; \cA_2 + c_{29}(u+j)\cdot \eta,
\end{align*}
where the second inequality follows from \eqref{eqn:cB0}, \eqref{eqn:linformcx}, and \eqref{eqn:bdp1}.

By assumption, $\ww+ \bb_e \cdot M \neq \mathbf{0}$, hence
\begin{equation*}
\begin{split}
\cD^2 \; & \le \; \lVert \ww+ \bb_e \cdot M \rVert_2^2 \\
& = \; b_{u+v-1}^2+\cdots +b_{r}^2+
\Theta_1^2+\cdots+\Theta_{d-2}^2\\
	& \le \;
	b_{u+v-1}^2+\cdots +b_{r}^2+w\cA_1^2+v\cA_2^2+
	2\cB_4\eta+\cB_3\eta^2.
\end{split}
\end{equation*}
However, $\lvert b_{u+v-1} \rvert \le \lVert \bb \rVert_\infty \le \cB_\infty$.
Moreover, by \eqref{eqn:rhobound} we have
$\lvert b_{u+v-1+i} \rvert \le \rho_i$ for $i=1,\dotsc,s$,
where $\rho_i$ is given in \eqref{eqn:rho}.
It follows that
\begin{equation*}
\begin{split}
	\cD^2 \; & \le \;
	\cB_\infty^2+\rho_1^2+\cdots+\rho_s^2+w\cA_1^2+v\cA_2^2+
        2\cB_4\eta+\cB_3\eta^2\\
	&=\; \cB_2^2-w \cB_\infty^2 +w\cA_1^2+v\cA_2^2+
	2\cB_4\eta+\cB_3\eta^2 \qquad \text{(from \eqref{eqn:cB2new}})\\
& = \; \cB_5^2+2\cB_4\eta+\cB_3\eta^2 \\
& = \; \cB_5^2+ \cB_3\left(\eta + \frac{\cB_4}{\cB_3}\right)^2 - \frac{\cB_4^2}{\cB_3}.
\end{split}
\end{equation*}
Recall our assumption that $\cB_5 \, < \, \cD$.
Thus
\[\frac{\cB_4^2}{\cB_3} \; < \; \cD^2 - \cB_5^2 + \frac{\cB_4^2}{\cB_3} \; \le \; \cB_3 \left(\eta + \frac{\cB_4}{\cB_3}\right)^2,\]
and so
\[0 \; < \; \frac{\sqrt{\cB_3(\cD^2-\cB_5^2)+\cB_4^2}\, - \, \cB_4}{\cB_3} \;
\le \; \eta.\]
However $\eta = 2 C \cdot \exp(-c_{26} \cdot B)$.
This yields the bound
\[
B \; \le \;
\frac{1}{c_{26}}
\cdot
\log{\left(
\frac{2 C \cdot \cB_3}{\sqrt{\cB_3(\cD^2-\cB_5^2)+\cB_4^2}\, - \, \cB_4}
\right)},
\]
which gives \eqref{eqn:Bfinal}.
\end{proof}

\noindent \textbf{Heuristic.}
It remains to decide on a reasonable choice
for $C$. We expect that the determinant of the matrix $M$ is approximately $C^{d-2}$. Thus the distance between adjacent vectors in $L$
is expected to be in the region of $C^{(d-2)/n}$, and so we anticipate
(very roughly) that $\cD \sim C^{(d-2)/n}$. We would like
$\cD> \cB_5$. Therefore it is reasonable to choose $C \gg \cB_5^{n/(d-2)}$.
If, for a particular choice of $C$, the condition $\cD>\cB_5$ fails,
then we simply try again with a larger choice of $C$.\\


\noindent \textbf{Remarks.} Our approach is somewhat unusual
in that it uses all $d-2$ available approximate relations
to reduce the initial bound. In contrast, it is much more
common to use one relation (e.g.\ \cite[Section 16]{TW})
to reduce the bound. In most examples, we have found
that both approaches give similar reductions in the size
of the bound and that there is no advantage in using one
over the other. However in some examples the approach
of using only one relation fails spectacularly. Here
are two such scenarios. 
\begin{enumerate}
\item[(i)] Suppose $\delta_1$ (say) belongs to a
proper subfield $K^\prime$ of $K$. Now let
$\sigma_2$, $\sigma_3$ be distinct embeddings of $K$
that agree on $K^\prime$. Then, in the notation
of Lemma~\ref{lem:linformreal} we find
 $\alpha_{1,3}=\log\lvert \sigma_3(\delta_1)/\sigma_2(\delta_1) \rvert=0$, and so the coefficient of the unknown $b_1$ is zero
in the approximate relation \eqref{eqn:linformrl}. Therefore
the one relation \eqref{eqn:linformrl} on its own fails to provide
any information on the size of $b_1$. In practice,
the lattice constructed in \cite[Section 16]{TW}
from this one relation will contain the tiny
vector $(1,0,\dotsc,0)$, and this will result in
the computational failure of the closest vector algorithm.
\item[(ii)] We continue to suppose that
 $\delta_1$ belongs to the proper subfield $K^\prime$
as above. Let $\sigma_{u+1}$ be a complex embedding of $K$
that extends a real embedding of $K^\prime$, and suppose for
simplicity that $\sigma_{u+1}(\delta_1)$ is positive. Then
again, $\alpha_{1,w+1}=0$ in the approximate relation
\eqref{eqn:linformcx}, and so that relation on its own
fails to control $b_1$.
\end{enumerate}
In the above two examples, we chose to illustrate the possible
failure of the approach of using one relation by imposing
$\delta_1 \in K^\prime$ where $K^\prime$ is a subfield of $K$.
However, a similar failure occurs (and is more difficult to
find) if the $S$-unit basis $\delta_1,\dotsc,\delta_r$
is multiplicatively dependent over $K^\prime$, meaning that
there is some non-trivial $(c_1,\dotsc,c_r) \in \Z^r$
such that $\delta_1^{c_1} \cdots \delta_r^{c_r} \in K^\prime$.

\medskip

In Proposition~\ref{prop:boundImprove}, we require $\bb_e \cdot M \ne -\ww$.
Of course, if $M$ is non-singular, we can simply check whether
$\bb_e=-\ww \cdot M^{-1}$ yields a solution, and therefore there
is no harm in making this assumption. In all our examples, $M$
has been non-singular, and we expect that by choosing $C$
large enough we can ensure that this happens. However,
if $M$ is singular then the equation $\bb_e \cdot M=-\ww$
either has no solutions or the solutions belong to
the translate of a sublattice of $\Z^n$ whose rank is
the co-rank of the matrix $M$. A glance at the matrix $M$
reveals that this co-rank is at most $w=u+v-2$. If this
case was to ever arise in practice, we would need
to enumerate all vectors $\bb_e$ satisfying
$\bb_e \cdot M = -\ww$ and the bound in Lemma~\ref{lem:never}
and test if they lead to solutions.


\subsection{Example~\ref{ex:Ex4} continued}
We continue giving details for
the tuple
$(\tau,\delta_1,\dotsc,\delta_{10})$
alluded to on page~\pageref{page:taudeltatuple}.
The values of $\cB_{\infty}$ and $\cB_2$ are given
in \eqref{eqn:cB0bd}. The set $S$
consists of five primes
$\fp_1,\dotsc,\fp_5$
given by \eqref{eqn:primesS}.
By applying Proposition~\ref{prop:valbd} we
had
(page~\pageref{page:fpbounds})
obtained bounds
$237$, $292$, $354$, $518$, $821$ for $\ord_{\fp_j}(a_0 X- \theta Y)$
with $i=1,\dotsc,5$ respectively; these are the values
denoted $k_j-1$ in \eqref{eqn:kjminus1}.
Now $\ord_{\fp_j}(\tau)=0,0,1,0,0$ respectively for
$j=1,\dotsc,5$. Letting $\varepsilon$ be as in \eqref{eqn:varepsilondef},
we may take $(-k_j^\prime,k_j^{\prime\prime})$ in \eqref{eqn:valbds}
to be $(0,237)$, $(0,292)$, $(-1,353)$, $(0,518)$, $(0,821)$
respectively.
This allows us to compute the constant $c_{21}$ defined in Lemma~\ref{lem:c21}.
We find that $c_{21} \approx 2842.79$.
The field $K$ has signature $(u,v)=(1,5)$
and thus there is only one possibility for $\sigma \in \cE_K^{\R}$.
We therefore take $\sigma_1=\sigma$ to be the unique real embedding.
For illustration, we give the values of constants appearing
in Section~\ref{sec:LFRL}:
\begin{gather*}
c_{22}\approx 30.31, \quad
c_{23}=1, \quad
c_{24} \approx 2873.10, \quad
c_{25} \approx 5.91 \times 10^{1247}, \quad
c_{26} \approx 0.35, \\
c_{27} \approx 5.91 \times 10^{1247}, \quad
c_{28} \approx 1.40 \times 10^{1246}, \quad
c_{29}(2) \approx 1.25 \times 10^{1246}, \\
c_{29}(3) \approx 8.30 \times 10^{1245}, \quad
c_{29}(4) \approx 7.83 \times 10^{1245}, \quad
c_{29}(5) \approx 9.21 \times 10^{1245}, \\
c_{29}(6) \approx 1.40 \times 10^{1246}, \quad
c_{30} \approx 8290.02.
\end{gather*}
Lemma~\ref{lem:linformreal} gives $w=u+v-2=4$ approximate relations
and Lemma~\ref{lem:linformimag} gives another $v=5$ relations.
Therefore we have $d-2=9$ relations altogether,
and $n=\#S+d-2=15$. Therefore the matrix $M$ is $15 \times 15$
and the lattice $L$ belongs to $\Z^{15}$.
We find that
\begin{gather*}
\cB_1\approx 7.85 \times 10^{222},\quad
\cA_1 \approx 3.92 \times 10^{222}, \quad
\cA_2 \approx 7.85 \times 10^{222},\quad
\cB_3 \approx 2.59 \times 10^{2493},\\
\cB_4 \approx 7.57 \times 10^{1469},\quad \text{ and } \quad
\cB_5 \approx 1.93 \times 10^{223}.
\end{gather*}
In accordance with the above heuristic, our program chooses
\[
C=\left[\cB_5^{n/(d-2)} \right] \approx 1.39 \times 10^{372}.
\]
The matrix $M$ and the lattice $L$ are too huge to reproduce here,
but we point out that
\[
[\Z^{15} : L] \approx 2.66 \times 10^{3357}; 
\qquad \cD \approx 7.23 \times 10^{223}. 
\]
In this case, $\ww \notin L$ so that $\cD$ is computed using $D(L,\ww)$.
Hence the hypothesis $\cD> \cB_5$ of Proposition~\ref{prop:boundImprove}
is satisfied. We may therefore apply Proposition~\ref{prop:boundImprove}
to obtain a new bound for $B$ given by \eqref{eqn:Bfinal}:
\[
B \le 9270.82. 
\]
We now start again with $\cB_{\infty}=9270$ and repeat the previous
steps, first for obtaining bounds for $\ord_{\fp_j}(a_0 X- \theta Y)$
and then for writing down the lattice $L$ and applying
Proposition~\ref{prop:boundImprove}. The following table
illustrates the results.

\begin{table}[h]
{\renewcommand{\arraystretch}{1.1} 
\begin{tabular}{|c||c||c|c|c|c|c|}
\hline
Iteration & $\cB_{\infty}$ &
\multicolumn{5}{c|}{bounds for $\ord_{\fp_j}(a_0 X- \theta Y)$}\\
& & \multicolumn{5}{c|}{with $1 \le j \le 5$}\\
\hline\hline
0 & $1.57 \times 10^{222}$ &
237 & 292 & 354 & 518 & 821\\
\hline
1 & 9270 & 4 & 5 & 8 & 10 & 15\\
\hline
2 & 251 & 3 & 3 & 5 & 6 & 10 \\
\hline
3 & 190 & 2 & 3 & 5 & 6 & 9 \\
\hline
4 & 180 & 2 & 3 & 5 & 6 & 9 \\
\hline
5 & 180 & 2 & 3 & 5 & 6 & 9 \\
\hline\hline
\end{tabular}}
\vspace{0.5em}
\caption{We successively reduce the bounds for
$B$ and for $\ord_{{\fp_j}}(a_0 X- \theta Y)$,
where $j=1,\dotsc,5$.}\label{table:bounds}
\end{table}

Note that at the fifth iteration we fail to obtain any improvement
on the bounds, and so we stop there. Recall that $r=10$
and that $B=\max(\lvert b_1\rvert,\dotsc,\lvert b_{10}\rvert)$, where
$b_1,\dotsc,b_{10}$ are the exponents in \eqref{eqn:TMdelta}.
Our final bound is $B \le 180$. The set of possible
integer tuples $(b_1,\dotsc,b_{10})$ satisfying this bound
has size
\[
(2 \times 180+1)^{10} \; =\; 362^{10} \approx 3.86 \times 10^{25}.
\]
The huge size of this region does not allow
brute force enumeration of the solutions. Instead, one can reduce
the number of tuples to consider by using the bounds
we have obtained for $\ord_{\fp_j}(a_0 X-\theta Y)$.
We let $\kappa_j= 2$, $3$, $5$, $6$, $9$ for $j=1,\dotsc,5$,
respectively. We know that $0 \le \ord_{\fp_j}(a_0 X-\theta Y) \le \kappa_j$,
and so there are $\kappa_j+1$ possibilities for the
$\ord_{\fp_j}(a_0 X-\theta Y)$. Let $(k_1,\dotsc,k_5)$ be some tuple of integers
satisfying $0 \le k_j \le \kappa_j$. The condition
$\ord_{\fp_j}(a_0 X-\theta Y)=k_j$
simply defines a hyperplane of codimension $1$
in the space of possible $(b_1,\dotsc,b_{10})$.
Imposing all five conditions
$\ord_{\fp_j}(a_0 X-\theta Y)=k_j$
with $j=1,\dotsc,5$ cuts down the dimension from $10$ to $5$.
Thus we expect that the search region should (very roughly) have size
\[
(\kappa_1+1) \cdots (\kappa_{5}+1) \cdot 362^{5}
\approx 3.13 \times 10^{16}.
\]
This is still way beyond brute force enumeration
and motivates our next section.

\section{Sieving}\label{sec:sieve}

In order to resolve the Thue--Mahler equation
\begin{equation*}
  F(X,Y)=a \cdot
  p_1^{z_1} \cdots p_v^{z_v}, \qquad
  X,~Y \in \Z, \;
  \gcd(X,Y)=\gcd(a_0,Y)=1,
\end{equation*}
we have first reduced the problem to that of resolving a number of equations of the form \eqref{eqn:TMdelta}, subject to the restrictions \eqref{eqn:restrictions}, \eqref{eqn:restrictions2}.
Recall that $B=\max\{\lvert b_1 \rvert, \dotsc,\lvert b_r \rvert\}$,
where $\bb=(b_1,\dotsc,b_r) \in \Z^r$ denotes the
vector of unknown exponents to solve for.
For each such equation \eqref{eqn:TMdelta}, we have used the theory of linear forms in logarithms 
to obtain a bound for $B$, and moreover, we have explained
how to repeatedly reduce this bound. 
During each of these iterations, we have simultaneously reduced the bounds on
the $\infty$-norm, the $1$-norm, and the $2$-norm of $\bb$.  Let us
denote the final bound for the $\infty$-norm of $\bb$ by $\cB_f^{\prime}$ and
write $\cB_f$ for the final bound on the $2$-norm of $\bb$.
Thus
\begin{equation}\label{eqn:cBf}
\lVert \bb \rVert_2 \le \cB_{f}, \qquad \lVert \bb \rVert_\infty \le
\cB_f^\prime.
\end{equation}
We have also explained how to obtain and reduce the bounds on
$\ord_\fp(a_0 X- \theta Y)$
for $\fp \in S$. Suppose that at the end of this process,
our bounds are
\begin{equation}\label{eqn:finalvalbd}
0 \le \; \ord_\fp(a_0 X- \theta Y) \; \le \kappa_\fp, \quad \text{for $\fp \in S$}.
\end{equation}
Unfortunately, in our high rank examples (i.e. when the
$S$-unit rank $r$ is large) the final bound $\cB_{f}^{\prime}$
is often
too large to allow for brute force enumeration of solutions. Instead, we shall sieve for solutions using both the primes $\fp$ in $S$,
and also rational primes $p$ whose support in $\OO_K$ is disjoint from $S$.
The objective of the sieve is to show that the solutions $\bb$
belong to a union of a certain (hopefully small) number of cosets $\ww+L$,
where the $L$ are sublattices of $\Z^r$. As the sieve progresses,
the determinants of the lattices $L$ will grow. The larger the
determinant, the fewer vectors we expect belonging to
$\ww+L$ and satisfying $\lVert \bb \rVert_2 \le \cB_{f}$,
and the easier it should be to find these vectors using
the algorithm of Fincke and Pohst \cite{FinckePohst}.  The following lemma
is a helpful guide to when Fincke and Pohst should be applied.
\begin{lem}\label{lem:lambdaL}
Let $L$ be a sublattice of $\Z^r$. Suppose $\lambda(L)>2 \cB_{f}$,
where $\lambda(L)$ denotes the length of the shortest
non-zero vector in $L$. Then there is at most one vector
$\bb$ in the coset $\ww+L$
satisfying $\lVert \bb \rVert_2 \le \cB_{f}$.
Moreover, any such $\bb$ is equal to $\ww+\cc(L,-\ww)$.
\end{lem}
\begin{proof}
Suppose there are vectors $\bb_1$, $\bb_2 \in \ww+L$
both satisfying $\lVert \bb_i \rVert_2 \le \cB_{f}$.
Then $\bb_1-\bb_2 \in L$ and $\lVert \bb_1-\bb_2 \rVert_2 \le 2 \cB_{f}$.
As $\lambda(L)> 2 \cB_{f}$ we see that $\bb_1=\bb_2$.
The second part follows from Lemma~\ref{lem:closest}.
\end{proof}
We continue
sieving until the lattices $L$ satisfy  $\lambda(L)> 2 \cB_{f}$.
We then apply the Fincke-Pohst algorithm to determine
$\cc(L,-\ww)$ and check whether the vector $\bb= \ww+\cc(L,-\ww)$
leads to a solution.


\subsection{Sieving using the primes of $S$}
To recap, we seek solutions $(X,Y,\bb)$ to
\[a_0 X- \theta Y=\tau \cdot \delta_1^{b_1} \cdots \delta_r^{b_r},\]
subject to the conditions
\[X,~Y \in \Z, \quad \gcd(X,Y)=1, \quad \gcd(a_0,Y)=1, \quad
b_i \in \Z,\]
and such that
\[\lVert \bb \rVert_2 \le \cB_{f}, \quad \text{ and } \quad  0 \le \; \ord_\fp(a_0 X- \theta Y) \; \le \kappa_\fp \quad \text{for every $\fp \in S$}.\]
In particular, this last inequality \eqref{eqn:finalvalbd} asserts that
$\ord_\fp(a_0 X-\theta Y)$ belongs to a certain
set of values $0,1,\dotsc,\kappa_\fp$.
The following proposition 
reduces this list somewhat, and for any $k$
in this reduced list, yields a vector $\ww_k$
and a sublattice $L_k$ of $\Z^r$ such that $\bb \in \ww_k+L_k$ whenever
$\ord_{\fp}(a_0 X- \theta Y)=k$.
\begin{prop}\label{prop:siftI}
Let $\fp \in S$.
Let $\theta_0 \in \Z$ satisfy $\theta_0 \equiv \theta \pmod{\fp^{\kappa_{\fp}}}$.
Let $\fa$ and $\cT_1$ be as in \eqref{eqn:facT1}.
Define
\[
\eta \; : \; \Z^r \rightarrow \Z, \qquad
\eta(n_1,\dotsc,n_r)=n_1 \ord_\fp(\delta_1)+\cdots+n_r \ord_\fp(\delta_r),
\qquad
L^{\prime\prime}=\Ker(\eta).
\]
Let
\[
\cK^{\prime\prime} \; := \; \left(\ord_\fp(\tau) + \Img(\eta) \right)
\; \cap \; \{ 0 \le k \le \kappa_\fp \; : \; \gcd(\fa^k,\theta-\theta_0) = \gcd(\fa^k,\cT_1) \}.
\]
For $k \in \cK^{\prime\prime}$,
let $\ww^{\prime\prime}_k$
be any vector in $\Z^r$ satisfying
$\eta(\ww^{\prime\prime}_k)=k-\ord_\fp(\tau)$.
If $k \in \cK^{\prime\prime}$ satisfies $k \ge 1$,
we will let $\phi$ and $\tau_0$ be as in Proposition~\ref{prop:valbd}
(these depend on $k$).
Let
\[
\cK^\prime \; := \begin{cases}
\{ k \in \cK^{\prime\prime} \; :\;
\overline{\tau_0} \in \Img(\phi)
\} & \text{if $0 \notin \cK^{\prime\prime}$}\\
\{0\} \cup \{ k \in \cK^{\prime\prime} \; :\; \overline{\tau_0} \in \Img(\phi)
\} & \text{if $0 \in \cK^{\prime\prime}$}.\\
\end{cases}
\]
For $k \in \cK^\prime$ with $k \ge 1$,
we let $L^\prime_k=\Ker(\phi)$ and $\ww_k^\prime$ be
any vector in $\Z^r$ satisfying $\phi(\ww_k^\prime)=\overline{\tau_0}$.
Let $\ww_0^\prime=\mathbf{0}$, $L_0^\prime=\Z^r$, and
\[
\cK \; :=
\{ k \in \cK^{\prime} \; :\;
(\ww^{\prime\prime}_k+L^{\prime\prime}) \cap (\ww^\prime_k+L_k^\prime) \ne \emptyset
\}.
\]
For $k \in \cK$, write
\[
L_k \; := \; L^{\prime\prime} \cap L_k^\prime
\]
and choose any $\ww_k \in \Z^r$ such that
\[
\ww_k+L_k \; := \;
(\ww^{\prime\prime}_k+L^{\prime\prime}) \cap (\ww^\prime_k+L_k^\prime).
\]
Let $k=\ord_\fp(a_0 X-\theta Y)$. Then $k \in \cK$
and $\bb \in \ww_k+L_k$.
\end{prop}
\begin{proof}
By \eqref{eqn:finalvalbd}, the valuation $k:=\ord_\fp(a_0 X- \theta Y)$
satisfies
$0 \le k \le \kappa_\fp$. Moreover, by Proposition~\ref{prop:valbd}, part (i),
we have $\gcd(\fa^k,\theta-\theta_0)=\gcd(\fa^k,\cT_1)$. By \eqref{eqn:TMdelta}, we know $k \in \ord_\fp(\tau)+\eta(\bb)$ and thus $k \in \cK^{\prime\prime}$
and $\bb \in \ww_k^{\prime\prime}+L^{\prime\prime}$.
In particular, the proposition follows in the case $k=0$.
We therefore suppose $k \ge 1$. By Proposition~\ref{prop:valbd},
part (ii) and its proof, it follows that $k \in \cK^\prime$ and $\bb \in \ww_k^\prime+L_k^\prime$, completing the proof.
\end{proof}


\noindent \textbf{Remark.} For each prime $\fp \in S$, Proposition~\ref{prop:siftI}
yields a number of cosets $\ww_k+L_k$ and tells us that
$\bb$ belongs to one of them. Note that $L^{\prime\prime}$ is a subgroup of
$\Z^r$ of rank $r-1$. Moreover, the subgroup $L_k^\prime$ has
finite index in $\Z^r$. Therefore $L_k:=L_k^\prime \cap L^{\prime\prime}$
has rank $r-1$.
From the remarks following Proposition~\ref{prop:valbd} (where
$L_k^\prime$ is called $L$) we expect $L_k^\prime$ to have
index $p^{(d-2)k}$ in $\Z^r$. In particular, the larger the
value of $k$, the larger the index of $L_k^\prime$.
Of course the number of cosets is bounded above by $\kappa_p+1$.


\subsection{Sieving with other primes}
Given a prime
ideal $\fq$ of $\OO_K$, write $\OO_\fq$ for the localization
of $\OO_K$ at $\fq$,
\[
\OO_\fq=\{ \alpha \in K \; : \; \ord_\fq(\alpha) \ge 0\}.
\]
Now let $q$ be a rational prime. Define
\[
\OO_q=\bigcap_{\fq \mid q} \OO_\fq=\{ \alpha \in K \; : \;
\text{$\ord_\fq(\alpha) \ge 0$ for all $\fq \mid q$}\}.
\]
The group of invertible elements $\OO_q^\times$
consists of all $\alpha \in K$
such that $\ord_\fq(\alpha)= 0$ for all
prime ideals $\fq \mid q$.

Let $\tau$, $\delta_1,\dotsc,\delta_r$ be as in \eqref{eqn:TMdelta}.
Let $q$ be a rational prime coprime to the supports
of $\tau$, $\delta_1,\dotsc,\delta_r$.
Thus $\tau$, $\delta_1,\dotsc,\delta_r$ all belong
to $\OO_q^\times$.
Let
\[
\fA_q:=(\OO_q /q\OO_q)^\times.
\]
This is canonically isomorphic to $(\OO_K/ q\OO_K)^\times$.
Let
\[
\mu : \F_q^\times \hookrightarrow \fA_q, \qquad \alpha+q\Z \mapsto \alpha+q\OO_q
\]
be the natural map, and let
\[
\fB_q:=\fA_q/\mu(\F_q^\times)
\]
be the cokernel of $\mu$.
We denote the induced homomorphism $\OO_q^\times \rightarrow \fB_q$ by
\[
\pi_q \; : \; \OO_q^\times \rightarrow \fB_q,
\qquad \beta \mapsto (\beta+q \OO_q) \cdot \mu(\F_q^\times).
\]

Define
\[
\phi_q \; : \; \Z^r  \rightarrow \fB_q,
\qquad
(m_1,\dotsc,m_r) \mapsto \pi_q(\delta_1)^{m_1} \cdots \pi_q(\delta_r)^{m_r}.
\]

\begin{prop}\label{prop:siftII}
Let
\[
R_q=\{a_0 u-\theta  : u \in \{0,1,\dotsc,q-1\}\} \cup \{ a_0\}
\]
and
\[
S_q=\{\pi_q(r)/\pi_q(\tau) \; : \; r \in R_q \cap \OO_q^\times\}
\subseteq \fB_q.
\]
Let
\[
T_q=S_q \cap \phi_q(\Z^r)
\]
and $L_q=\Ker(\phi_q)$. Finally, let $W_q \subset \Z^r$ be a set
of size $\# T_q$ such that for every $t \in T_q$
there is some $\ww \in W_q$ with $\phi_q(\ww)=t$.
Then $\bb \in W_q+L_q$.
\end{prop}
\begin{proof}
Since $\tau$, $\delta_1,\dotsc,\delta_r \in \OO_q^\times$,
we have $a_0 X-\theta Y \in \OO_q^\times$.
We want to determine the possibilities for
the image of the algebraic integer $a_0 X- \theta Y$ in
$\fB_q$.
Since $X$ and $Y$ are coprime, $q$ divides at most one of $X$, $Y$.
If $q \nmid Y$ then
\[
a_0X- \theta Y \equiv v \cdot (a_0 u -\theta) \pmod{q \OO_q}
\]
for some $u \in \{0,1,\dotsc,q-1\}$ and some $v \in \F_q^\times$.
If $q \mid Y$ then $q \nmid X$ and
\[a_0 X-\theta Y \equiv a_0 v \pmod{q \OO_q}\]
for some $v \in \F_q^\times$. We conclude that
$a_0 X-\theta Y \equiv v \cdot r \pmod{q \OO_q}$
where $v \in \F_q^\times$ and $r \in R_q$. Moreover,
since $a_0X-\theta Y\in \OO_q^\times$ we see
that $r \in R_q \cap \OO_q^\times$. Now
\[
\pi_q(a_0 X - \theta Y)=\pi_q(r) \pi_q(v)=\pi_q(r).
\]
It follows that $\pi_q(a_0 X- \theta Y)/\pi_q(\tau) \in S_q$.
However
\[
\phi_q(\bb)=\pi_q(\delta_1)^{b_1} \cdot \cdots \cdot
\pi_q(\delta_r)^{b_r}=\pi_q(a_0 X-\theta Y)/\pi_q(\tau),
\]
where the first equality follows from the definition of $\phi_q$
and the second from \eqref{eqn:TMdelta}.
Thus $\phi_q(\bb)=t$ for some $t \in T_q$.
By definition of $W_q$, there is some $\ww \in W_q$
with $\phi_q(\ww)=t=\phi_q(\bb)$, thus $\bb-\ww \in L_q$.
\end{proof}


\noindent \textbf{Heuristic.}
It is appropriate that we heuristically \lq measure\rq\ the quality of
information that Proposition~\ref{prop:siftII} gives us about the solutions.
A priori, $\phi_q(\bb)$ could be any element in $\phi_q(\Z^r) \subseteq \fB_q$.
However, the lemma tells us that $\phi_q(\bb)$ belongs to
$T_q$. We want to estimate the ratio
$\# T_q/\#\phi_q(\Z^r)$; the smaller this ratio is, the better
the information is. It is convenient to
suppose that $q$ is unramified in $\OO_K$. Thus
\[
\OO_q/q\OO_q \cong \OO_K/q \OO_K \cong \bigoplus_{\fq \mid q} \OO_K/\fq.
\]
Each summand $\OO_K/\fq$ is a finite field of cardinality $\Norm(\fq)$.
By definition
\[
\fA_q:=(\OO_q/q\OO_q)^\times \cong \prod_{\fq \mid q} (\OO_K/\fq)^\times.
\]
Thus
\[
\# \fA_q=\prod_{\fq \mid q} (\Norm(\fq)-1),
\qquad
\# \fB_q=\frac{1}{q-1}\cdot \prod_{\fq \mid q} (\Norm(\fq)-1).
\]
Moreover
$\prod_{\fq \mid q} \Norm(\fq)=q^d$ where $d=[K:\Q]$
is the degree of the original Thue--Mahler equation.
Thus $\# \fB_q \approx q^{d-1}$. However
$S_q \subseteq \phi(\Z^r) \subseteq \fB_q$ has
at most $q+1$ elements,
and so $\# S_q/\#\fB_q \lessapprox 1/q^{d-2}$. Now
\[
\frac{\#T_q}{\#\phi_q(\Z^r)}=\frac{\# S_q \cap \phi_q(\Z^r)}{
\# \phi_q(\Z^r)}.
\]
It is reasonable to expect that the elements of $S_q$
are uniformly distributed among the elements of $\fB_q$ and so we expect $\#T_q/\#\phi_q(\Z^r) \lessapprox 1/q^{d-2}$.

\subsection{The sieve}
We will sieve with the primes $\fp \in S$ as in Proposition~\ref{prop:siftI}
and also with additional rational primes $q$ as
in Proposition~\ref{prop:siftII}. We would like to choose
a suitable set $\sS$ of such primes $q$.
The most expensive computation we will need to do for $q \in \sS$
is to
compute, for each $t \in T_q$, some $\ww \in \Z^r$
such that $\phi_q(\ww)=t$. This involves
a discrete logarithm computation in the group
$\mathfrak{B}_q$, and to do this quickly
we need $\mathfrak{B}_q$ to be a product
of cyclic factors that have relatively small order.
We therefore like to avoid those $q$ where there are $\fq \mid q$
that have large norm.  In all our examples we found it enough to
take $\sS$ to be the set of primes $q \le 500$, where
each $\fq \mid p$ satisfies $\norm(\fq) \le 10^{10}$ and where
the support of $q$ is disjoint from the supports of
$\tau$, $\delta_1,\dotsc,\delta_r$.

\begin{proc}\label{proc}
$\texttt{Solutions}(L_c,\ww_c,S_c,\sS_c)$.\\[2ex]
\textbf{Input:} $L_c$ sublattice of $\Z^r$, $\ww_c \in \Z^r$, $S_c \subseteq S$,
$\sS_c\subseteq \sS$.\\[2ex]
\textbf{Output:} Set of solutions $(X,Y,\bb)$ to \eqref{eqn:TMdelta},
\eqref{eqn:restrictions} satisfying $\bb \in \ww_c+L_c$
and $\lVert \bb \rVert_2 \le \cB_{f}$.\\[2ex]
\begin{tabbing}
1. \textbf{IF} \= $\lambda(L_c)>2\cB_{f}$ or ($S_c=\emptyset$ and $\sS_c=\emptyset$) \textbf{THEN}\\
2. \> Apply Fincke--Pohst to find all vectors in $\bb \in \ww_c+L$
satisfying $\lVert \bb \rVert_2 \le \cB_{f}$. \\
3. \> Keep only those $\bb$ that lead to
solutions $(X,Y)$ on \eqref{eqn:TMdelta}, \eqref{eqn:restrictions}.\\
4. \> \textbf{RETURN:} Set  of $(X,Y,\bb)$.\\
5. \> \textbf{END}.\\
6. \textbf{ELSE}\\
7. \> \textbf{IF} \= $S_c \ne \emptyset$ \textbf{THEN}\\
8. \>\> Choose $\fp \in S_c$. Let $S_c^\prime=S_c \setminus \{\fp\}$.\\
9. \>\> Compute $\cK$ as in Proposition~\ref{prop:siftI}.\\
10. \>\> For each $k \in \cK$ compute $\ww_k$, $L_k$ as in Proposition~\ref{prop:siftI}.\\
11. \>\>Let $\cK^*$ be the subset of $k\in \cK$ such that $(\ww_k+L_k)\cap (\ww_c+L_c) \ne \emptyset$.\\
12. \>\>For each $k \in \cK^*$ let $L_{c,k}=L_c \cap L_k$.\\
13. \>\> For each $k \in \cK^*$ choose
$\ww_{c,k}\in \Z^r$ so that $\ww_{c,k}+L_{c,k}=(\ww_k+L_k)\cap (\ww_c+L_c)$.\\
14. \>\> \textbf{RETURN:} $\displaystyle \bigcup_{k \in \cK^*} \mathtt{Solutions}(L_{c,k},\ww_{c,k},S_c^\prime,\sS_c)$.\\
15. \>\> \textbf{END}.\\
16. \>\textbf{ELSE}\\
17. \>\> Choose $q \in \sS_c$. Let $\sS_c^\prime=\sS_c \setminus \{q\}$.\\
18. \>\> Compute $W_q$, $L_q$ as in Proposition~\ref{prop:siftII}.\\
19. \>\> Let $L_c^\prime=L_c \cap L_q$.\\
20. \>\> Let $W_q^*=\{\ww_1,\dotsc,\ww_m\}$ be the subset of $\ww \in W_q$ such that
$(\ww+L_q) \cap (\ww_c+L_c) \ne \emptyset$. \\
21. \>\> For $i=1,\dotsc,m$ choose $\ww_{c,i}$ such that
$\ww_{c,i}+L_c^\prime=(\ww+L_q) \cap (\ww_c+L_c)$.\\
22. \>\> \textbf{RETURN:} $\displaystyle \bigcup_{i=1}^m \mathtt{Solutions}(L_{c}^\prime,\ww_{c,i},\emptyset,\sS_c^\prime)$.\\
23. \>\> \textbf{END}.\\
24. \>\textbf{ENDIF}\\
25. \textbf{ENDIF}
\end{tabbing}
\end{proc}
Let us explain how Procedure~\ref{proc} works.
The procedure starts with a coset $\ww_c+L_c$
and sets $S_c \subseteq S$ and $\sS_c \subseteq \sS$
(the subscript $c$ stands for \lq cumulative\rq).
The objective is to return all solutions $(X,Y,\bb)$ to
\eqref{eqn:TMdelta}, \eqref{eqn:restrictions}
with $\bb \in \ww_c+L_c$ and
satisfying $\lVert \bb\rVert_2 \le \cB_f$. The primes in $S_c$ and $\sS_c$
are used, via Propositions~\ref{prop:siftI} and~\ref{prop:siftII},
to replace $\ww_c+L_c$ by a union of cosets of sublattices of $L_c$.

We now explain lines 1--5 of the procedure.
If $\lambda(L)>2 \cB_f$, then by Lemma~\ref{lem:lambdaL}, the coset
$\ww_c+L_c$ has at most one vector $\bb$ that satisfies
$\lVert \bb \rVert_2 \le \cB_f$, and this maybe found by
the algorithm of Fincke and Pohst. If $S_c=\emptyset$ and $\sS_c=\emptyset$,
then we have run out of sieving primes and we simply apply
the Fincke-Pohst algorithm to determine all $\bb \in \ww_c+L_c$
such that $\lVert \bb \rVert_2 \le \cB_f$.
We test
all resulting $\bb$ to see if they lead to solutions
$(X,Y,\bb)$ and return the set of solutions. We end here.
In both these cases,
no further branching of the procedure occurs.

If we have reached line 6, then either $S_c$ is non-empty
or $\sS_c$ is non-empty. We first treat the case where
$S_c$ is non-empty (lines 8--14). We choose $\fp \in S_c \subseteq S$
to sieve with and let $S_c^\prime=S_c \setminus \{\fp\}$.
Here we apply Proposition~\ref{prop:siftI}.
This gives a finite set $\cK$ of values $k$ and lattice
cosets $\ww_k+L_k$ such that $\bb \in \ww_k+L_k$ for some $k \in \cK$.
However, the $\bb$ we are interested in belong to $\ww_c+L_c$.
We let $\cK^*$ be those values $k \in \cK$ such that
$(\ww_c+L_c) \cap (\ww_k+L_k) \ne \emptyset$. It is now clear
that every $\bb$ we seek belongs to
$(\ww_c+L_c) \cap (\ww_k+L_k)$ for some $k \in \cK^*$.
However $(\ww_c+L_c) \cap (\ww_k+L_k)=\ww_{c,k}+L_{c,k}$
where $L_{c,k}=L_c \cap L_k$, for a suitable coset
representative $\ww_{c,k}$. We apply the procedure
to the set $(L_{c,k},\ww_{c,k},S_c^\prime,\sS_c)$
for each $k \in \cK^*$ to compute those $\bb$ belonging
to $\ww_{c,k}+L_{c,k}$ and return the union.

If however $S_c=\emptyset$, then (lines 17--22) we
choose a prime $q \in \sS_c \subseteq \sS$ to sieve with,
and we let $\sS_c^\prime=\sS_c \setminus \{q\}$.
Now we apply Proposition~\ref{prop:siftII}.
This gives a lattice $L_q$ and a finite set $W_q$
such that $\bb \in W_q+L_q$. Therefore there is some
$\ww \in W_q$ such that $\bb \in (\ww+L_q) \cap (\ww_c+L_c)$.
We let $W_q^*$ be the subset of those $\ww \in W_q$
such that $(\ww+L_q) \cap (\ww_c+L_c) \ne \emptyset$,
and write $W_q^*=\{\ww_1,\dotsc,\ww_m\}$.
Now $\bb \in (\ww_i+L_q) \cap (\ww_c+L_c)$ for some
$i=1,\dots,m$. Write $L_{c}^\prime=L_c \cap L_q$.
Then $(\ww_i+L_q) \cap (\ww_c+L_c)$ is a coset of $L_c^\prime$
for $i=1,\dotsc,m$, and we choose $\ww_{c,i}$
so that $\ww_{c,i}+L_c^\prime=(\ww_i+L_q) \cap (\ww_c+L_c)$.
It is therefore enough to find the $\bb$
belonging to each one of these $\ww_{c,i}+L_c^\prime$.
Thus we apply the procedure to $(L_c^\prime,\ww_{c,i},\emptyset,\sS_q^\prime)$
for $i=1,\dotsc,m$, collect the solutions and return their union (line 22).

\medskip

\noindent \textbf{Remarks.}
\begin{itemize}
\item To compute the solutions to \eqref{eqn:TMdelta}
satisfying \eqref{eqn:restrictions}, it is clearly enough
to apply the above procedure to $(\Z^r,\textbf{0},S,\sS)$.
\item Recall that $\delta_1,\dots,\delta_r$ is a basis
for the $S$-units (modulo torsion); in particular this allows us to identify the $S$-units (modulo torsion) with $\Z^r$.
Let $\fp \in S$ and $\eta$ be as in
Proposition~\ref{prop:siftI}. Note that $L_k$ is a subgroup
of finite index in $\Ker(\eta)$. Now $\Ker(\eta)$
itself corresponds to the $(S\setminus\{\fp\})$-units, and therefore
has rank $r-1$. Therefore $L_k$ has rank $r-1$.
That is, if we apply the procedure to
$(\Z^r,\textbf{0},S,\sS)$, then at depth $\#S+1$ (when the set $S$
has been entirely depleted),
the lattice $L_c$ will have rank $r-\#S$ which is the unit rank.
Beyond this depth, the rank remains constant but the determinant
of the lattice grows.
\item The reader will note that we have not specified how to choose
the next prime $\fp \in S$ or $q \in \sS$. In our implementation
we order the primes in $\fp \in S$ by the size of their norms; from
largest to smallest. The reason is that the primes $\fp \in S$
of large norm lead to lattices of large determinants and we
therefore expect few short vectors. Once $S$ is exhausted,
the choices we make for the next $q \in \sS$  actually depend
on the cumulative lattice $L_c$. We choose the prime $q \in \sS_c$
that minimizes $\# W_q / [L_c : L_c \cap L_q]$. Our justification
for this is that we are replacing one coset of $L_c$ with a union
of cosets of $L_c \cap L_q$. The number of such cosets is bounded
by $\#W_q$. The function $q \mapsto \# W_q / [L_c : L_c \cap L_q]$
estimates the \lq relative change in density\rq\ between the old
lattice and the new union for that particular choice of $q$.
\end{itemize}


\subsection{Example~\ref{ex:Ex4} continued}
Recall that $\cB_f^\prime=180$. Following the remark in Section~\ref{sec:LFRL}, we find that $\cB_f \approx 402.67$. 
Consider the information given by Proposition~\ref{prop:siftI}.
Recall there are five possibilities for $\fp \in S$, ordered as $\fp_1,\dotsc,\fp_{5}$,
in order of decreasing norm. Table~\ref{table:bounds}
yields $2$, $3$, $5$, $6$, $9$ for $\kappa_{\fp_j}$ with $j = 1, \dots, 5$, respectively.

\begin{table}[h]
{\renewcommand{\arraystretch}{1.1} 
\begin{tabular}{|c|l|l|}
\hline
$\fp$ & $\cK$ & $\det(L_k)$ with $k \in \cK$ \\
\hline\hline
$\fp_1$ & $\{0,1\}$ &
$1$, $6616761038619033600$ \\
\hline
$\fp_2$ & $\{0,1,2,3\}$ &
$1$, $2114272224838656$, $3442909640611645594437761516544$,\\
& & $5606480875148980721912830543593855583743968256$\\
\hline
$\fp_3$ & $\{0,1,2,3,4,5\}$ &
$1$, $1$, $3800066789376$, $14496104390625000000000000$,\\
& & $55298249781131744384765625000000000000$,\\
& &  $210946082233931520022451877593994140625000000000000$\\
\hline
$\fp_4$ & $\{0, 1, 2, 3, 6\}$ &
$1$, $504631296$, $21722722606780416$, \\
& &  $935091979414469275815936$, \\
& & $54375352676603537816702220559499682956095667933184$
\\
\hline
$\fp_5$ & $\{0, 1, 5\}$ &
$1$, $57600$, $1062532458645670173081600$
\\
\hline \hline
\end{tabular}}
\vspace{0.5em}
\caption{This table gives the sets $\cK$ and the
determinants of the sublattices $L_k \subset \Z^{10}$ with $k \in \cK$
as in Proposition~\ref{prop:siftI}. Observe that the
sublattices $L_k$ all have rank $r-1=9$.}
\end{table}

We take $\sS$ to be the set of rational primes $q<200$ coprime to the prime ideals in $S$ and such that every prime ideal factor of $q\OO_K$ has norm $\le 10^{10}$. This is done in order to keep our computations fast, as previously explained. However, of this set, our program only needs to use the primes $23$ and $71$, selected in that order using the heuristic detailed in the above remarks. For $q = 23$ and $q = 71$, Proposition~\ref{prop:siftII} gives a lattice $L_{q} \subset \Z^{10}$ (now
of rank $10$) and
a set $W_{q}$ such that $\bb \in W_{q}+L_{q}$. We find that
\[
[\Z^{10}: L_{71}]=3253933989048960000 \approx 3.26 \times 10^{18}
\quad \text{ with } \quad
\# W_{71}=71,
\]
and
\[
[\Z^{10}: L_{23}]=41191874887680 \approx 4.12 \times 10^{13}
\quad \text{ with } \quad
\# W_{23}=23.
\]
Observe that in Procedure~\ref{proc} branching occurs
at lines 14, 20. Thus we obtain \lq paths\rq\ through the
algorithm depending on the choice of $k \in \cK^*$
(line 14) or the choice of $\ww_i \in W_q^*$
(line 20). A path \lq dies\rq\ if the criterion of
line 1 is satisfied, or if $\cK^*$ (defined in line 11)
is empty, or if $W_q^*$ (defined in line 20) is
empty. Our program needs to check a total of $98$ paths.
Five of these terminate at line 1 with the condition
$\lambda(L_c)>2 \cB_f$ being satisfied, and the remaining $93$ paths terminate at line 11 with $\cK^*=\emptyset$.
Of the $5$ paths that terminate
at line 1, three of these yield a vector $\bb \in W_c+L_c$
satisfying $\lVert \bb \rVert_2 \le \cB_f$. These three vectors
are
\begin{gather*}
(-1, -1, -2, 0, 2, 0, 3, 0, 1, 1), \quad (0, 0, -1, 1, 1, 1, 1, 1, 0, 0),\\
\text{ and } \quad (-1, -1, -2, 3, 2, 5, 0, 0, 0, 0).
\end{gather*}
These vectors respectively lead to the solutions
\[
F(1,2)=3^3 \cdot 7 \cdot 11,\qquad
F(1,-1)=2 \cdot 3 \cdot 5, \qquad
F(1,1)=2^5.
\]

\bibliographystyle{abbrv}
\bibliography{samir}

\end{document}